\documentclass[11pt,a4paper,oneside]{report}

\usepackage{amssymb} 
\usepackage{mathrsfs} 
\usepackage{amsmath} 
\usepackage{amsthm}
\usepackage{MnSymbol}

\usepackage{enumitem} 

\usepackage[utf8]{inputenc}
\usepackage[english]{babel}

\usepackage{tikz}
\usetikzlibrary{calc}
\usetikzlibrary{matrix}
\usepackage{xargs}
\usepackage{dynkin-diagrams}

\usepackage[textsize=tiny]{todonotes}

\usepackage{epigraph}

\newtheorem{theorem}{Theorem}[section]
\newtheorem*{theorem*}{Theorem}
\newtheorem{lemma}[theorem]{Lemma}
\newtheorem{proposition}[theorem]{Proposition}
\newtheorem{remark}[theorem]{Remark}
\newtheorem{fact}[theorem]{Fact}
\newtheorem{definition}[theorem]{Definition}
\newtheorem{corollary}[theorem]{Corollary}

\newtheorem{observation}[theorem]{Observation}
\newtheorem{convention}[theorem]{Convention}
\newtheorem{example}[theorem]{Example}

\newcounter{claim}[theorem]

\newenvironment{claim}[1][]
{\refstepcounter{claim}\par\textbf{Claim. #1}\rmfamily}
{\par}

\newenvironment{claimProof}
{\par{Proof.} \rmfamily}
{$\blacksquare$\par}

\newenvironment{wide}{%
	\begin{list}{}{%
			\setlength{\topsep}{0pt}%
			\addtolength{\leftmargin}{-1.3cm}%
			\addtolength{\rightmargin}{-1.3cm}%
			\setlength{\listparindent}{\parindent}%
			\setlength{\itemindent}{\parindent}%
			\setlength{\parsep}{\parskip}}%
		\item[]%
	}{%
	\end{list}%
}
\newcommand{\domPnm}{(P_n\times P_2^m,\trianglelefteq)}

\newcommandx{\Vij}[2][1=i, 2=j]{V_{#1, #2}}

\newcommandx{\Pnm}[2][1=n, 2=m]{P_{#1}\times P_2^{#2}}
\newcommandx{\Hnm}[2][1=n, 2=m]{\mathscr{H}_{#1,#2}}
\newcommandx{\Ynm}[2][1=n, 2=m]{\mathrm{Y}_{#1,#2}}
\newcommandx{\Znm}[2][1=n, 2=m]{\mathrm{Z}_{#1,#2}}

\newcommandx{\Snm}[2][1=n, 2=m]{\mathcal{S}_{#1,#2}}
\newcommandx{\Gnm}[2][1=n, 2=m]{\mathrm{G}_{#1,#2}}
\newcommandx{\Fnm}[2][1=n, 2=m]{\mathcal{F}_{#1,#2}}

\newcommandx{\AGnm}[2][1=n, 2=m]{\mathscr{A}_{#1,#2}}

\newcommandx{\hnm}[2][1=n, 2=m]{\eta_{#1,#2}}

\newcommandx{\uz}[2][1=n, 2=m]{u^0_{#1,#2}}
\newcommandx{\uo}[2][1=n, 2=m]{u^1_{#1,#2}}

\newcommandx{\dnm}[2][1=n, 2=m]{d_{#1,#2}}
\newcommandx{\dz}[2][1=n, 2=m]{d^0_{#1,#2}}
\newcommandx{\doo}[2][1=n, 2=m]{d^1_{#1,#2}}

\newcommand{\st}{\operatorname{St}}
\newcommand{\aut}{\operatorname{Aut}}
\newcommand{\rank}{\operatorname{rank}}
\newcommand{\diam}{\operatorname{diam}}
\newcommand{\ecc}{\operatorname{ecc}}
\newcommand{\piv}{\operatorname{Piv}}

\newcommand{\len}{\operatorname{len}}
\newcommand{\ps}{\sigma}
\newcommand{\Icwidth}{\gamma}

\newcommand{\tri}{\trianglelefteq}
\newcommand{\join}{\vee}
\newcommand{\meet}{\wedge}

\begin{document}

\thispagestyle{plain}
\pagenumbering{Roman}\thispagestyle{empty}  \setcounter{page}{-4}
\begin{center}
	\Huge 
	Flip Graphs
	\Large
	\vskip 3cm
	\textbf{Roy Harold Jennings}
	\vskip 1cm
	\textrm{Department of Mathematics}
    \vskip 5cm
	\textrm{Ph.D. Thesis} 
	\vskip 1cm
    \textrm{Submitted to the Senate of
		Bar-Ilan University}
\end{center}
\vskip 1cm \textrm{Ramat-Gan, Israel  \hskip 7.5cm  January 2020}
\newpage 

\begin{wide}
	This work was carried out under the guidance and supervision of
	Prof. Ron Adin and Prof. Yuval Roichman, Department of Mathematics, Bar-Ilan
	University.
\end{wide}
\newpage

\setlength\epigraphwidth{.6\textwidth}
\setlength\epigraphrule{0pt}
\section*{Acknowledgements}

\epigraph{``Perhaps we’ll have some answers, at least, before the end. I always dreamed of dying well-informed."}{---Joe Abercrombie,  \textit{Last Argument of Kings}}

I would like to express my gratitude to my advisors Prof. \nolinebreak Ron M. Adin and Prof. Yuval Roichman for their continuous support of my Ph.D. study, for their patience and immense knowledge. 
Their encouragement and guidance helped me build the confidence and courage needed for the fail and retry iterations I went through during the research and writing of this thesis. 

I would also like to thank my friends and fellow researchers Dr. Menachem (Meny) Shlossberg and Dr. Arnon Netzer for their support and encouragement and for many helpful discussions and suggestions.

Special thanks go to my parents Roger and Hava Ben-Ari for supporting me throughout the writing of this thesis.

Last, but by no means least, I would like to thank Luie Jennings, my best friend and partner for life, for being a source of energy, motivation and inspiration; 
for being so supportive and understanding; but most of all, for always putting things in proportions, for her lifesaving sense of humor and her love.
\newpage

\linespread{1.4}
\setcounter{tocdepth}{3}
\tableofcontents

\newpage
\linespread{1.4}
\setcounter{secnumdepth}{0}
\section{Abstract}
Flip graphs are graphs on combinatorial objects in which the adjacency relation reflects a local change in the underlying objects.
In this thesis we introduce \textit{Yoke graphs}, a family of flip graphs that generalizes previously studied families of flip graphs on colored triangle-free triangulations \nolinebreak \cite{TFT1}, arc permutations \cite{elizalde} and geometric caterpillars \cite{yuval}.
Our main results are the computation of the diameter of an arbitrary Yoke graph and a full characterization of the automorphism group of this family of graphs.
We also show that Yoke graphs are Schreier graphs of the affine Weyl group of type $\tilde{C}_m$.

The approach we take in the computation of the diameter is different from the ones in \cite{TFT1} and \cite{elizalde}. 
We show that the approach in \cite{elizalde} (for arc permutation graphs) does not extend to Yoke graphs.
At the heart of our proof lies the idea of transforming a diameter evaluation into an eccentricity problem.
The characterization of the automorphism group is a new result even for the above mentioned three special families of Yoke graphs.

\newpage
\pagenumbering{arabic}
\setcounter{secnumdepth}{3}


\chapter{Introduction}
Over recent decades, there has been an increasing interest in flip graphs, namely, graphs on combinatorial objects in which the adjacency relation reflects a local change. 
For example, see \cite{bose, fabila, parlier, pournin, tarjan}.
In this work we introduce \textit{Yoke graphs}, a family of flip graphs that generalizes previously studied families of flip graphs on colored triangle-free triangulations (CTFT) \cite{TFT1}, arc permutations \cite{elizalde} and geometric caterpillars \cite{yuval}.

Typical problems related to flip graphs include metric properties, such as distance, diameter, finding antipodes and counting geodesics between them; and algebraic properties, such as presentations using Cayley/Schreier graphs, automorphism groups and eigenvalues.
The generalization to Yoke graphs of the three families of flip graphs mentioned above is motivated by their surprising similarities in terms of algebraic, combinatorial and metric properties.
In particular, they carry similar group actions, are intimately related to posets and have similar diameter formulas. 

\section{Main Results}
In this thesis, as mentioned above, we introduce a family of flip graphs, namely, Yoke graphs $\Ynm$.
We compute the diameter of Yoke graphs and prove (see Theorems \ref{thm:ynm_ecc}, \ref{thm:ecc_eq_diam} and \ref{thm:ecc_znm}):
\begin{theorem*}
    \begin{enumerate}
        \item[]
	    \item If $n=1$, then $\diam(\Ynm)=\binom{\lceil \frac{m}{2}\rceil + 1}{2} + \binom{\lfloor \frac{m}{2}\rfloor + 1}{2}$. 
        \item If $0\leq m\leq n$, then $\diam(\Ynm) = \lfloor\frac{n(m+1)}{2}\rfloor$.
        \item If $2\leq n\leq m$, let $\dz = \binom{\lfloor\frac{m+n}{2}\rfloor+1}{2} + \binom{\lceil\frac{m-n}{2}\rceil+1}{2}$. Then
        \begin{enumerate}
            \item if either $2\divides (m-n)$ or $n\leq\lceil\frac{m+1}{2}\rceil$, then $\diam(\Ynm) = \dz$;
            \item otherwise, $\diam(\Ynm) = \dz + n-\lceil\frac{m+1}{2}\rceil$.
        \end{enumerate}
    \end{enumerate}
\end{theorem*}
We also give a full characterization of the automorphism group of this family of graphs (for all values of $n$ and $m\neq 2$).
For example (see Theorems \ref{thm:anm_structure} and \ref{thm:aut_structure} for the complete result):

\begin{theorem*}
	Let $n\geq 1$ and $m\geq 3$ be two integers such that $(n,m)\neq(1, 3)$.
        \begin{enumerate}
            \item $\aut(\Ynm[1][m]) \cong C_2\times C_2$ ($\forall m>3$). 
            \item If $n>1$ and at least one of $n$ and $m$ is odd, then $\aut(\Ynm) \cong D_{2n}$.
            \item If $n>1$ and both $n$ and $m$ are even, then $\aut(\Ynm) \cong D_{n}\times C_2$.
        \end{enumerate}
\end{theorem*}

Since Yoke graphs generalize the three families of flip graphs mentioned above, these results provide a unified proof for known results regarding \linebreak diameter.
Our approach in this proof is different from the ones in \cite{TFT1} and \cite{elizalde}.
At the heart of the proof lies the idea of transforming a diameter evaluation into an eccentricity problem.
The characterization of the automorphism group is a new result for the mentioned three cases of Yoke graphs.

\section{The Structure of the Thesis}
In Chapter \ref{chp:background}, we recall and develop the theory that will be necessary for the presentation of our main results. 
Chapter \ref{chp:YokeGraphs} introduces Yoke graphs, the main object in this thesis.
As noted in the introduction, the definition of Yoke graphs was motivated by similarities between CTFT graphs, arc permutation graphs and caterpillar graphs, in terms of their algebraic, combinatorial and metric properties.
In Section \ref{sec:special_cases}, we show that these graphs are indeed special cases of Yoke graphs.
In Section \ref{sec:group_actions}, the fact that CTFT graphs and arc permutation graphs are Schreier graphs of the affine Weyl group of type $\tilde{C}_{n}$ is extended to Yoke graphs.

In Chapter \ref{chp:ecc_zero_Ynm}, we compute the eccentricity of the vertex $0$ in Yoke graphs.
Chapter \ref{chp:diameter_of_Ynm} deals with the diameter of Yoke graphs.
In Section \ref{sec:diameter_dominance}, we develop techniques and obtain partial results on the diameter, based on similarities between Yoke graphs and the dominance order on $\mathbb{Z}^n$.
In Sections \ref{sec:dYokeGraphs}, \ref{sec:diam_ecc} and \ref{sec:ecc_zero_Znm}, we show that the problem of the diameter in Yoke graphs can be converted into a problem of eccentricity in larger graphs, namely, \textit{dYoke graphs}.
The terminology and techniques that were introduced in Chapter \ref{chp:ecc_zero_Ynm} are extended and used in order to compute the eccentricity of $0$ in dYoke graphs.

Chapter \ref{chp:automorphisms} deals with the automorphism group of Yoke graphs.
In Section \ref{sec:fundamental_auto}, three fundamental automorphisms of $\Ynm$ are introduced and we compute the group they generate.
In Section \ref{sec:complete_auto_group}, we show that this group is in fact the entire automorphism group of $\Ynm$ whenever $m\neq 2$ and $(n,m)\neq (1,3)$.
We cover the case $(n,m)=(1,3)$ separately.

\chapter{Mathematical Background}\label{chp:background}
In this chapter, we recall and develop the theory that will be necessary for the presentation of our main results. 

\section{Graphs}
A \textbf{graph} $G$ on a \textbf{vertex} set $V$ is an ordered pair $(V, E)$ where $E\subseteq V\times V$ is a binary relation on $V$. 
$E$ is called the \textbf{adjacency} relation of $G$ and its elements are called \textbf{edges}. 
Accordingly, if $(u, v)\in E$, then $u$ and $v$ are said to be \textbf{adjacent} vertices or \textbf{neighbors} in $G$.
An edge $(u,u)\in E$ is called a \textbf{loop}.
If the relation $E$ is symmetric, then we say that $G$ is an \textbf{undirected graph} and denote adjacent vertices $u$ and $v$ by $u\sim v$.
Otherwise, we say that $G$ is a \textbf{directed graph} or a \textbf{digraph}.
A \textbf{simple graph} $G$ is an undirected graph without loops.
In the rest of this work, the term \textit{graph} denotes a simple graph on a finite set of vertices (directed and infinite graphs will be explicitly stated as such).

A graph $G'=(V',E')$ is a \textbf{subgraph} of $G=(V,E)$, denoted $G'\leq G$, if $V'\subseteq V$ and $E'\subseteq E$.
A subgraph $G'$ is said to be an \textbf{induced subgraph} of $G$ if every edge $e$ in $G$ such that $e\in V'\times V'$ is an edge in $G'$.
A \textbf{complete graph} is one in which all of the vertices are adjacent.
The complete graph on $n$ vertices is denoted by $K_n$.

A \textbf{path} $P=(v^0\sim\dots \sim v^d)$ between $v^0$ and $v^d$ in a (simple) graph $G$ is a sequence $v^0, \dots, v^d$ of vertices of $G$ such that $v^{i-1}\sim v^i$ for every $1\leq i\leq d$.
The path is a \textbf{cycle} if $v^0=v^d$.
Here, $d$ is called the \textbf{length} of $P$.
$G$ is said to be \textbf{connected} if every two vertices in $G$ are connected by a path.
A \textbf{geodesic} (or \textbf{shortest path}) between two vertices $u$ and $v$ in $G$ is a path with minimum length between them.
The \textbf{distance} $d_G(v,u)$ between $v$ and $u$ is the length of a geodesic between them.
We write $d(v,u)$ when $G$ is evident.
The \textbf{diameter} $\diam(G)$ of $G$ is the maximal distance between any two vertices in $G$.
Two vertices in $G$ are said to be \textbf{antipodes} if the distance between them is equal to the diameter of the graph.
The \textbf{eccentricity} $\ecc_G(v)$ of a vertex $v$ in a graph $G$ is the maximal distance of the vertex from any other vertex in $G$.
We write $\ecc(v)$ when $G$ is evident.
A vertex $u$ is said to be \textbf{eccentric} to the vertex $v$ if $d(v,u)=ecc(v)$.
The \textbf{valency} $\delta_G(v)$ of a vertex $v$ in a graph $G$ is the number of neighbors of $v$ in $G$.

An \textbf{isomorphism} of graphs $G$ and $H$ is a bijection $\varphi$ between the vertex sets of $G$ and $H$ such that any two vertices $u$ and $v$ of $G$ are adjacent in $G$ if and only if $\varphi(u)$ and $\varphi(v)$ are adjacent in $H$.
An isomorphism of a graph with itself is called an \textbf{automorphism}.
The set of automorphisms of a graph $G$, with respect to composition, forms a subgroup $\aut(G)$ of the symmetric group on $n$ elements, where $n$ is the number of vertices in the graph.
We call this group the \textbf{automorphism group} of the graph.
Throughout this work, composition of functions is computed right-to-left: 
$(gf)(x)=g(f(x))$.

\subsection{Cayley and Schreier Graphs}
Let $G$ be a group and let $S$ be a symmetric ($S^{-1}=S$) generating set of $G$.
The (right) {\bf Cayley graph} $X(G,S)$ is the graph on the set of elements of $G$ in which $g_1\sim g_2$ if $g_2 = g_1 s$ for some $s$ in $S$.
Let $H$ be a subgroup of $G$.
The (right) {\bf Schreier coset graph} (also \textbf{Schreier graph} or \textbf{coset graph}) $X(G/H,S)$ is the graph on the set of (right) cosets $\{Hg:g\in G\}$ of $H$, in which $Hg_1 \sim Hg_2$ if $Hg_2=(Hg_1)s$ for some $s$ in $S$.
Note that the Cayley graph $X(G,S)$ is a special case of a Schreier graph $X(G/H,S)$ for the trivial subgroup $H=\{e\}$. 

\begin{remark}\label{rem:shcreier_graph_iso}
Assume that a group $G$ acts transitively on the vertices of a graph $Y$ such that the edges of $Y$ correspond to the action of a symmetric generating set $S$ of $G$.
In this case, $Y$ is isomorphic to the Schreier graph $X(G/\st(v),S)$ for every $v$ in $Y$ ($\st(v)$ denotes the stabilizer of a vertex $v$ in $Y$).
\end{remark}

\subsection{Flip Graphs}\label{subsec:background_flip_graphs}
This thesis was motivated by three previously studied families of flip graphs on colored triangle-free triangulations (CTFT) \cite{TFT1}, arc permutations \cite{elizalde} and geometric caterpillars \cite{yuval}.
In this subsection we introduce these three families.

\noindent
{\large Colored Triangle-Free Triangulation Graphs\par}
All definitions and claims in this subsection, related to CTFT graphs, can be found in \cite{TFT1}.
The flip graph of triangulations of a convex polygon \cite{tarjan} inspired the definition of a few flip graphs on subsets of triangulations.
One such graph is the CTFT graph.

A {\bf triangulation} of a convex polygon in the plane is its subdivision into triangles using non-intersecting diagonals. 
An edge of a triangulation that belongs to the polygon is called an {\bf external edge}. 
Every other edge is called a {\bf chord}.
Every triangulation of a convex $n$-gon ($n\geq 3$) has exactly $n-2$ triangles and $n-3$ chords.
We denote by $P_{n}$ a convex polygon with $n\geq 5$ vertices labeled (counter clockwise) by the elements of the additive cyclic group $\mathbb{Z}_n$. 

A triangulation is called {\bf inner-triangle-free}, or simply {\bf triangle-free}, if it contains no triangle with $3$ chords. 
A chord which forms a triangle with two external edges is called a {\bf short chord} and the triangle formed by these three edges is called an {\bf outer triangle}.
A triangulation of $P_{n}$ is triangle-free if and only if it contains exactly two outer triangles.
Moreover, every triangle-free triangulation $T$ of $P_{n}$ induces two opposite linear orders on the chords of $T$ in which the two short chords are the minimum and maximum elements of the order (see \cite{TFT1} for more details). 

A {\bf coloring}, or {\bf orientation} of a triangle-free triangulation $T$ of $P_n$ is a labeling of its chords by $\{0,...,n-4\}$ in one of the two linear orders induced by $T$ on its chords. 
Denote the set of all colored triangle-free triangulations of $P_n$ by $CTFT(n)$. 

Each chord in a triangulation of $P_n$ is a diagonal of a unique quadrilateral (the union of two adjacent triangles). 
Replacing a chord in a triangulation by the other diagonal of its corresponding quadrilateral is called a {\bf flip} of the chord (see Figure \ref{fig:ex_triang_flip}).

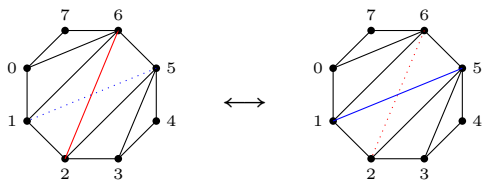
\begin{figure}[hbt]
    \[
    \begin{aligned}
    \begin{tikzpicture}[scale=0.5]
    \fill (2.4,3.4) circle (0.1) node[above]{\tiny 6};
    \fill (3.4,2.4) circle (0.1) node[right]{\tiny 5};
    \fill (3.4,1)   circle (0.1) node[right]{\tiny 4};
    \fill (2.4,0)   circle (0.1) node[below]{\tiny 3};
    \fill (1,0)     circle (0.1) node[below]{\tiny 2};
    \fill (0,1)     circle (0.1) node[left]{\tiny 1};
    \fill (0,2.4)   circle (0.1) node[left]{\tiny 0};
    \fill (1,3.4)   circle (0.1) node[above]{\tiny 7};
    
    \draw (1,3.4)--(2.4,3.4)--(3.4,2.4)--(3.4,1)--(2.4,0)--(1,0)--(0,1)--(0,2.4)--(1,3.4);
    \draw (3.4,2.4)--(1,0);
    \draw (0,1)--(2.4,3.4);
    \draw[red] (2.4,3.4)--(1,0);
    \draw (3.4,2.4)--(2.4,0);
    \draw (0,2.4)--(2.4,3.4);
    \draw[blue,dotted] (3.4,2.4)--(0,1);
    \end{tikzpicture}
    \end{aligned}
    \ \ \ \longleftrightarrow \ \ \
    \begin{aligned}
    \begin{tikzpicture}[scale=0.5]
    \fill (2.4,3.4) circle (0.1) node[above]{\tiny 6};
    \fill (3.4,2.4) circle (0.1) node[right]{\tiny 5};
    \fill (3.4,1)   circle (0.1) node[right]{\tiny 4};
    \fill (2.4,0)   circle (0.1) node[below]{\tiny 3};
    \fill (1,0)     circle (0.1) node[below]{\tiny 2};
    \fill (0,1)     circle (0.1) node[left]{\tiny 1};
    \fill (0,2.4)   circle (0.1) node[left]{\tiny 0};
    \fill (1,3.4)   circle (0.1) node[above]{\tiny 7};
    
    \draw (1,3.4)--(2.4,3.4)--(3.4,2.4)--(3.4,1)--(2.4,0)--(1,0)--(0,1)--(0,2.4)--(1,3.4);
    \draw (3.4,2.4)--(1,0);
    \draw (0,1)--(2.4,3.4);
    \draw (3.4,2.4)--(2.4,0);
    \draw (0,2.4)--(2.4,3.4);
    \draw[red,dotted] (2.4,3.4)--(1,0);
    \draw[blue] (3.4,2.4)--(0,1);
    \end{tikzpicture}
    \end{aligned}
    \]
    \caption{Flip of a chord}
    \label{fig:ex_triang_flip}
\end{figure}

The {\bf flip graph of colored triangle-free triangulations} $\Gamma_n$ is the graph on $CTFT(n)$ where two triangulations are connected by an edge if one is obtained from the other by a flip of a chord.
For example, see $\Gamma_6$ in Figure \ref{fig:ex_gamma6}.
This graph is closely related to a distinguished lower interval in the weak order on the affine Weyl group $\widetilde{C}_n$. 
The diameter of this flip graph was calculated using lattice properties of the order, see \cite[Theorem 5.1]{TFT1}.

\begin{figure}[hbt]
    \[
    \begin{aligned}
    \includegraphics[width=35mm]{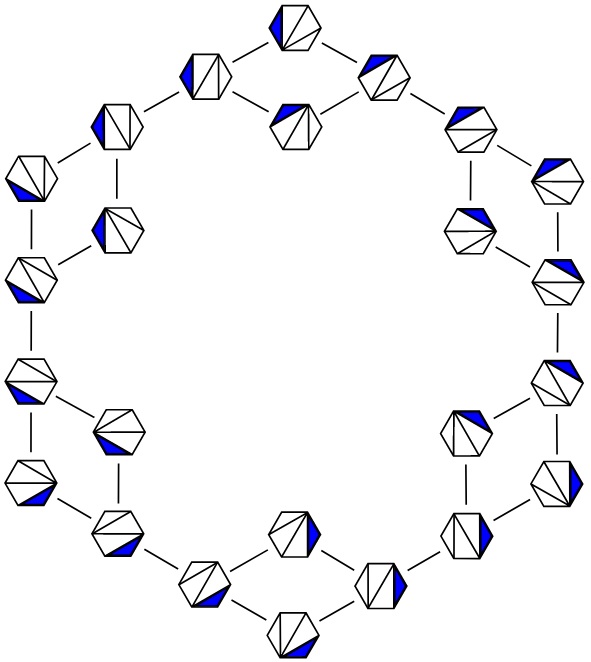}
    \end{aligned}
    \]
    \caption{$\Gamma_6$}
    \label{fig:ex_gamma6}
\end{figure}


\noindent
{\large Arc Permutation Graphs\par}
All definitions and claims relating to arc permutation graphs can be found in \cite{elizalde}.
A permutation $\pi$ in the symmetric group $S_n$ is called an {\bf arc permutation} if every prefix (and suffix) in $\pi$ forms an interval in $\mathbb{Z}_n$, where the element $0$ in $Z_n$ is represented by $n$ for convenience. 
For example, the permutation $3421576$ is an arc permutations in $S_7$ (see Figure \nolinebreak \ref{fig:ex_arc_permutation}) but $5643127$ is not. 
Denote the set of arc permutations in $S_n$ by $\mathcal{A}_n$.

\newcommand{\circRad} {20pt}
\newcommand{\circleScale} {0.7}

\newcommand{\createNodes}[7] {
    \draw (0,0) circle (\circRad);
    \node[draw, fill=#1, minimum size=1pt,shape=circle, scale=\circleScale] (v0) at ({-(360/7)*0 + 90}:\circRad) {\tiny 1};
    \node[draw, fill=#2, minimum size=1pt,shape=circle, scale=\circleScale] (v1) at ({-(360/7)*1 + 90}:\circRad) {\tiny 2};
    \node[draw, fill=#3, minimum size=1pt,shape=circle, scale=\circleScale] (v2) at ({-(360/7)*2 + 90}:\circRad) {\tiny 3};
    \node[draw, fill=#4, minimum size=1pt,shape=circle, scale=\circleScale] (v3) at ({-(360/7)*3 + 90}:\circRad) {\tiny 4};
    \node[draw, fill=#5, minimum size=1pt,shape=circle, scale=\circleScale] (v4) at ({-(360/7)*4 + 90}:\circRad) {\tiny 5};
    \node[draw, fill=#6, minimum size=1pt,shape=circle, scale=\circleScale] (v5) at ({-(360/7)*5 + 90}:\circRad) {\tiny 6};
    \node[draw, fill=#7, minimum size=1pt,shape=circle, scale=\circleScale] (v6) at ({-(360/7)*6 + 90}:\circRad) {\tiny 7};
}

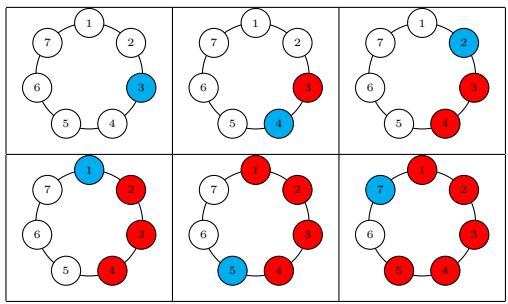
\begin{figure}[hbt]
    \[
    \begin{aligned}
    \begin{tabular}{|c|c|c|}
    \hline
    \begin{tikzpicture}
    \createNodes{white}{white}{cyan}{white}{white}{white}{white}
    \end{tikzpicture} 
    &  
    \begin{tikzpicture}
    \createNodes{white}{white}{red}{cyan}{white}{white}{white}
    \end{tikzpicture} 
    &  
    \begin{tikzpicture}
    \createNodes{white}{cyan}{red}{red}{white}{white}{white}
    \end{tikzpicture}\\
    \hline
    \begin{tikzpicture}
    \createNodes{cyan}{red}{red}{red}{white}{white}{white}
    \end{tikzpicture} 
    &
    \begin{tikzpicture}
    \createNodes{red}{red}{red}{red}{cyan}{white}{white}
    \end{tikzpicture}
    &  
    \begin{tikzpicture}
    \createNodes{red}{red}{red}{red}{red}{white}{cyan}
    \end{tikzpicture} \\
    \hline
    \end{tabular}
    \end{aligned}
    \]
    \caption{$\pi=3421576$ is an arc permutation in $\mathcal{A}_7$}
    \label{fig:ex_arc_permutation}
\end{figure}

The {\bf flip graph of arc permutations} $X_n$ is the graph on $\mathcal{A}_n$ in which two arc permutations $\pi$ and $\sigma$ in $\mathcal{A}_n$ are connected by an edge if \linebreak \mbox{$\pi = \sigma\circ(i,i+1)$} for some $1\leq i \leq n-1$, where $(i,i+1)$ is a transposition written in cycle notation.
For example, see $X_4$ in Figure \ref{fig:ex_arc_permutation_graph}.
A transposition $(i,i+1)$ is also called a simple reflection.
Recall that the set of all simple reflections in $S_n$ is a generating set of $S_n$.
An equivalent definition of $X_n$ is as follows: let $G$ be the Cayley graph associated with $(S_n,S)$, where $S$ is the set of simple reflections in $S_n$.
Then $X_n$ is the subgraph of $G$ that is induced by \nolinebreak $\mathcal{A}_n$.

\newcommand{\addVertex}[7]{%
    
    \filldraw[black] (#1:#2) circle(0.4pt);
    \draw (#1:#3) node[fill=white] {\scriptsize #4};
}

\begin{figure}[hbt]
    \[
    \begin{aligned}
    \begin{tikzpicture}[scale=1.5, cap=round, >=latex]
    
    \addVertex{{90+(360*0)/12}}{1cm}{1.11cm}{4321}{1.27cm}{$\psi=(3,0,0)$}{{90+(360*0)/12}}
    \addVertex{{90+(360*1)/12}}{1cm}{1.19cm}{3421}{1.39cm}{$\psi=(2,1,0)$}{{90+(360*1)/12+4}}
    \addVertex{{90+(360*2)/12}}{1cm}{1.20cm}{3241}{1.42cm}{$\psi=(2,0,1)$}{{90+(360*2)/12-2}}
    \addVertex{{(90+(360*3)/12)}}{1cm}{1.21cm}{3214}{1.39cm}{$\psi=(2,0,0)$}{{(90+(360*3)/12)+6}}
    \addVertex{{90+(360*4)/12}}{1cm}{1.21cm}{2314}{1.37cm}{$\psi=(1,1,0)$}{{90+(360*4)/12+4}}
    \addVertex{{90+(360*5)/12}}{1cm}{1.2cm}{2134}{1.38cm}{$\psi=(1,0,1)$}{{(90+(360*5)/12)}}
    \addVertex{{90+(360*6)/12}}{1cm}{1.2cm}{1234}{1.35cm}{$\psi=(0,1,1)$}{{90+(360*6)/12}}
    \addVertex{{90+(360*7)/12}}{1cm}{1.2cm}{1243}{1.38cm}{$\psi=(0,1,0)$}{{90+(360*7)/12}}
    \addVertex{{90+(360*8)/12}}{1cm}{1.21cm}{1423}{1.37cm}{$\psi=(0,0,1)$}{{90+(360*8)/12-4}}
    \addVertex{{90+(360*9)/12}}{1cm}{1.21cm}{1432}{1.39cm}{$\psi=(0,0,0)$}{{(90+(360*9)/12)-6}}
    \addVertex{{90+(360*10)/12}}{1cm}{1.21cm}{4132}{1.37cm}{$\psi=(3,1,0)$}{{90+(360*10)/12+4}}
    \addVertex{{90+(360*11)/12}}{1cm}{1.2cm}{4312}{1.43cm}{$\psi=(3,0,1)$}{{90+(360*11)/12-4}}
    
    \addVertex{{90+(360*0)/12}}{0.76cm}{0.66cm}{3412}{0.51cm}{$\psi=(2,1,1)$}{{90+(360*0)/12}}
    \addVertex{{(90+(360*3)/12)}}{0.76cm}{0.55cm}{2341}{0.46cm}{$\psi=(1,1,1)$}{{(90+(360*3)/12)+20}}
    \addVertex{{90+(360*6)/12}}{0.76cm}{0.66cm}{2143}{0.51cm}{$\psi=(1,0,0)$}{{90+(360*6)/12}}
    \addVertex{{90+(360*9)/12}}{0.76cm}{0.45cm}{4123}{0.46cm}{$\psi=(3,1,1)$}{{(90+(360*9)/12)-20}}
    
    \draw[thick] ({90+(360*0)/12}:1cm) -- ({90+(360*1)/12}:1cm);
    \draw[thick] ({90+(360*1)/12}:1cm) -- ({90+(360*2)/12}:1cm);
    \draw[thick] ({90+(360*2)/12}:1cm) -- ({90+(360*3)/12}:1cm);
    \draw[thick] ({90+(360*3)/12}:1cm) -- ({90+(360*4)/12}:1cm);
    \draw[thick] ({90+(360*4)/12}:1cm) -- ({90+(360*5)/12}:1cm);
    \draw[thick] ({90+(360*5)/12}:1cm) -- ({90+(360*6)/12}:1cm);
    \draw[thick] ({90+(360*6)/12}:1cm) -- ({90+(360*7)/12}:1cm);
    \draw[thick] ({90+(360*7)/12}:1cm) -- ({90+(360*8)/12}:1cm);
    \draw[thick] ({90+(360*8)/12}:1cm) -- ({90+(360*9)/12}:1cm);
    \draw[thick] ({90+(360*9)/12}:1cm) -- ({90+(360*10)/12}:1cm);
    \draw[thick] ({90+(360*10)/12}:1cm) -- ({90+(360*11)/12}:1cm);
    \draw[thick] ({90+(360*11)/12}:1cm) -- ({90+(360*0)/12}:1cm);
    
    \draw[thick] ({90+(360*1)/12}:1cm) -- ({90+(360*0)/12}:0.76cm);
    \draw[thick] ({90+(360*11)/12}:1cm) -- ({90+(360*0)/12}:0.76cm);
    
    \draw[thick] ({90+(360*2)/12}:1cm) -- ({(90+(360*3)/12)}:0.76cm);
    \draw[thick] ({90+(360*4)/12}:1cm) -- ({(90+(360*3)/12)}:0.76cm);
    
    \draw[thick] ({90+(360*5)/12}:1cm) -- ({90+(360*6)/12}:0.76cm);
    \draw[thick] ({90+(360*7)/12}:1cm) -- ({90+(360*6)/12}:0.76cm);
    
    \draw[thick] ({90+(360*8)/12}:1cm) -- ({90+(360*9)/12}:0.76cm);
    \draw[thick] ({90+(360*10)/12}:1cm) -- ({90+(360*9)/12}:0.76cm);

    \end{tikzpicture} 
    \end{aligned}
    \]
    \caption{$X_4$}
    \label{fig:ex_arc_permutation_graph}
\end{figure}

The diameter of the graph of arc permutations was computed using similarities between the graph and the dominance order on $\mathbb{Z}^n$, see \cite[Subsection \nolinebreak 6.2]{elizalde}.

\noindent
{\large Geometric Caterpillar Graphs\par}
To define a geometric caterpillar, start with the complete graph $K_n$, for $n\geq 3$, whose vertices are labeled by $\mathbb{Z}_n$. 
Embed $K_n$ in the plane such that its vertices are the vertices of a convex polygon (the labels increasing clockwise) and its edges are straight line segments.
Denote this geometric graph by $\hat{K}_n$.
A \textbf{geometric caterpillar} (or caterpillar) of order $n$ is a non-crossing spanning tree of $\hat{K}_n$, such that the vertices that are not leaves in the tree form an interval in $\mathbb{Z}_n$.
The \textbf{spine} of a caterpillar is its longest (simple) path whose edges are on the boundary of the polygon.
Caterpillars are also called fishbones or combs and were studied by Keller, Khachatryan, Perles, Sagan, Wachs and others in various contexts, see, e.g., \cite{keller, perles, yuval2, yuval, martin}.

The {\bf flip graph of caterpillars} $\mathcal{G}_n$ is the graph on the set of caterpillars of order $n$ in which two caterpillars are connected by an edge if one is obtained from the other by shifting an endpoint of an edge along the spine (for an example of two adjacent caterpillar in $\mathcal{G}_8$, see Figure \ref{fig:ex_cat_flip}).

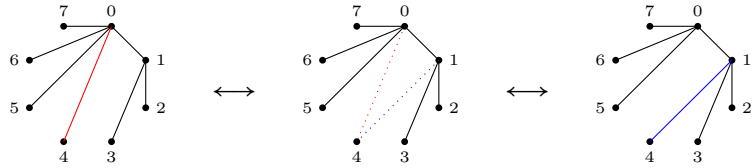
\begin{figure}[hbt]
    \[
    \begin{aligned}
    \begin{tikzpicture}[scale=0.45]
    \fill (2.4,3.4) circle (0.1) node[above]{\tiny 0}; 
    \fill (3.4,2.4) circle (0.1) node[right]{\tiny 1}; 
    \fill (3.4,1) circle (0.1) node[right]{\tiny 2}; 
    \fill (2.4,0) circle (0.1) node[below]{\tiny 3}; 
    \fill (1,0) circle (0.1) node[below]{\tiny 4}; 
    \fill (0,1) circle (0.1) node[left]{\tiny 5}; 
    \fill (0,2.4) circle (0.1) node[left]{\tiny 6}; 
    \fill (1,3.4) circle (0.1) node[above]{\tiny 7};
    \draw (1,3.4)--(2.4,3.4)--(3.4,2.4)--(3.4,1);
    \draw (0,1)--(2.4,3.4); 
    \draw[red] (2.4,3.4)--(1,0);
    \draw (3.4,2.4)--(2.4,0); 
    \draw (0,2.4)--(2.4,3.4);
    \end{tikzpicture}
    \end{aligned}
    \ \ \ \longleftrightarrow \ \ \
    \begin{aligned}
    \begin{tikzpicture}[scale=0.45]
    \fill (2.4,3.4) circle (0.1) node[above]{\tiny 0}; 
    \fill (3.4,2.4) circle (0.1) node[right]{\tiny 1}; 
    \fill (3.4,1) circle (0.1) node[right]{\tiny 2}; 
    \fill (2.4,0) circle (0.1) node[below]{\tiny 3}; 
    \fill (1,0) circle (0.1) node[below]{\tiny 4}; 
    \fill (0,1) circle (0.1) node[left]{\tiny 5}; 
    \fill (0,2.4) circle (0.1) node[left]{\tiny 6}; 
    \fill (1,3.4) circle (0.1) node[above]{\tiny 7};
    \draw (1,3.4)--(2.4,3.4)--(3.4,2.4)--(3.4,1);
    \draw[red,dotted] (2.4,3.4)--(1,0);
    \draw (3.4,2.4)--(2.4,0); 
    \draw (0,2.4)--(2.4,3.4);
    \draw[blue,dotted] (3.4,2.4)--(1,0);
    \draw (0,1)--(2.4,3.4);
    \end{tikzpicture}
    \end{aligned}
    \ \ \ \longleftrightarrow \ \ \
    \begin{aligned}
    \begin{tikzpicture}[scale=0.45]
    \fill (2.4,3.4) circle (0.1) node[above]{\tiny 0}; 
    \fill (3.4,2.4) circle (0.1) node[right]{\tiny 1}; 
    \fill (3.4,1) circle (0.1) node[right]{\tiny 2}; 
    \fill (2.4,0) circle (0.1) node[below]{\tiny 3}; 
    \fill (1,0) circle (0.1) node[below]{\tiny 4}; 
    \fill (0,1) circle (0.1) node[left]{\tiny 5}; 
    \fill (0,2.4) circle (0.1) node[left]{\tiny 6}; 
    \fill (1,3.4) circle (0.1) node[above]{\tiny 7};
    \draw (1,3.4)--(2.4,3.4)--(3.4,2.4)--(3.4,1); 
    \draw (3.4,2.4)--(2.4,0); 
    \draw (0,2.4)--(2.4,3.4);
    \draw[blue] (3.4,2.4)--(1,0); 
    \draw (0,1)--(2.4,3.4);
    \end{tikzpicture}
    \end{aligned}
    \]
    \caption{Adjacent caterpillars in  $\mathcal{G}_8$}
    \label{fig:ex_cat_flip}
\end{figure}

\section{Posets}

A {\bf partial order} $\leq$ on a set $P$ is a binary relation on $P$ which is reflexive, transitive and antisymmetric. 
A set $P$ equipped with a partial order $\leq$ is called a {\bf partially ordered set} (or a {\bf poset}) and is denoted by $(P,\leq)$.
If we refer to a set $P$ as a poset, then either the partial order is evident, or we assume that the partial order on $P$ is denoted by $\leq_P$.


$Q$ is called an {\bf induced subposet} of $P$ if $Q\subseteq P$ and $x\leq_Q y$ if and only if $x\leq_P y$ for every $x,y\in Q$.
In this case, the order on $Q$ is called the {\bf induced order}.
In the rest of this work, a {\bf subposet} means an induced subposet.

%
A map $\varphi:P \rightarrow Q$ between two posets is called {\bf order preserving} if $x\leq_P y$ implies that $\varphi(x)\leq_Q \varphi(y)$. 
$\varphi$ is an \textbf{isomorphism} if it is surjective and $\varphi(x)\leq_Q \varphi(y)$ if and only if $x\leq_P y$ for every $x$ and $y$ in $P$.
(Note that $\varphi(x)=\varphi(y)$ implies both $x\leq y$ and $y\leq x$, therefore $\varphi$ is injective.)
In this case, $P$ and $Q$ are said to be \textbf{isomorphic} and we denote it by $P\cong Q$.


Let $x$ and $y$ be two elements in $(P,\leq)$. If $x\leq y$ or $y\leq x$, then we say that $x$ and $y$ are {\bf comparable}.  Otherwise, $x$ and $y$ are {\bf incomparable}.
We denote by $x<y$ the case $x\leq y$ and $x\neq y$.
A poset in which every two elements are comparable is called a \textbf{chain} or a \textbf{linearly ordered set}.
A poset in which no two elements are comparable is called an {\bf antichain}.
For convenience, we say that $C\subseteq P$ is a chain in $P$, if $C$ is a chain as a subposet of $P$. 


For every $x\leq y$ in a poset $P$, the {\bf interval} $[x,y]$ is the subposet \linebreak $\{z\in P:x\leq z \leq y\}$. 
If every interval in $P$ is finite, then $P$ is called a {\bf locally finite} poset.
In this thesis, every poset is locally finite.
A subposet $Q$ of $P$ is called {\bf convex} if for every $x,y\in Q$ the interval $[x,y]$ in $P$ is contained in $Q$.

Let $x<y$ be two elements in a poset $P$. 
$y$ is said to $\bf cover$ $x$, denoted by $x\lessdot y$, if the interval $[x,y]$ is the two element set $\{x,y\}$. 
The {\bf covering relation} of $P$ is the relation $\lessdot$ on $P$, consisting of all such pairs in $P$.
The \textbf{Hasse diagram} of the poset $(P,\leq)$ is the digraph $\mathscr{H}(P,\leq)$ on $P$ in which $(x,y)$ is an edge if and only if $x\lessdot y$.
Throughout this thesis, by the Hasse diagram of a poset, we mean the underlying undirected graph of the Hasse diagram.


A {\bf maximal} chain in $P$ is one not contained in any other chain in $P$. 
A {\bf maximum} chain in $P$ is one with maximal length.
A chain $C$ is called \textbf{saturated} if there does not exist an element $z\notin C$ and two elements $x\leq y$ in $C$ such that $x \leq z \leq y$ and $z$ can be added to $C$ without losing the property of being totally ordered.

The {\bf length} $\len(C)$ of a finite chain $C$ is equal to $|C|-1$. 
The length $\len(P)$ of a finite poset $P$ is the length of its maximum chains. 
The length of an interval $[x,y]$ is denoted by $\len(x,y)$.

We say that a poset $P$ is a {\bf graded poset} if there exists an order preserving {\bf rank} function $\rho:P\rightarrow \mathbb{N}$ such that $x\lessdot y$ implies that $\rho(x)=\rho(y)-1$.
Equivalently, a poset $P$ is graded if and only if for every $x\leq y$ in $P$ all the maximal chains in $[x,y]$ have the same length.
In a graded poset with a rank function $\rho$, $\len(x,y)=\rho(y)-\rho(x)$ for every two elements $x\leq y$. 

Let $P_1,\dots, P_n$ be $n$ posets. 
The {\bf direct product} (or {\bf product}) order $\leq$ on $P_1\times\dots\times P_n$ is the order in which $(v_1,\dots,v_n) \leq (u_1,\dots,u_n)$ if $v_i\leq_{P_i} u_i$ for every $1\leq i\leq n$.
When we write $P_1\times\dots\times P_n$, we mean the poset on $P_1\times\dots\times P_n$ equipped with the product order, unless stated otherwise.
Note that $(v_1,\dots,v_n) \lessdot (u_1,\dots,u_n)$ in $P_1\times\dots\times P_n$ if and only if $u_i$ covers $v_i$ in $P_i$ for some $1\leq i\leq n$ and $u_j=v_j$ for every other $1\leq j\leq n$.
The poset obtained by a direct product of $(P,\leq)$ with itself $n$ times is denoted by $(P^n,\leq)$ or $P^n$.
In this case, the meaning of $\leq$ should be clear depending on the elements being compared.

\begin{example}\label{example:Z^n_is_graded}
    $(\mathbb{Z}^n,\leq)$ is a graded poset with a rank function $\rank_\leq(v)=\sum_{1}^{n}v_i$.
    The distance between any two elements $u=(u_1,\dots,u_n)$ and $v=(v_1,\dots,v_n)$ in its Hasse diagram $\mathscr{H}(\mathbb{Z}^n,\leq)$ is $\sum_{i=1}^{n}|u_i-v_i|$.
\end{example}

\subsection{Lattices}
Let $P$ be a poset and let $S\subseteq P$.
The {\bf join} of $S$ in $P$ is the least upper bound of $S$ (if it exists). 
The join $s\join t$ of two elements $s$ and $t$ in $P$ is the join of $\{s,t\}$ in $P$. 
The {\bf meet} of $S$ in $P$ is the greatest lower bound of $S$ (if it exists). 
The meet $s\meet t$ of two elements $s$ and $t$ in $P$ is the meet of $\{s,t\}$ in \nolinebreak $P$. 

A poset $L$ is called a {\bf lattice} if every two elements in $L$ have both a meet and a join. 
A subposet $S$ of a lattice $L$ is called a \textbf{sublattice} if $S$ is closed under the meet and join operations of $L$.
For example, every interval $[u,v]$ in a lattice $L$ is a sublattice of $L$.
If $L$ and $M$ are lattices, then so is $L \times M$ equipped with the product order.

\begin{example}\label{example:Z^n_is_a_lattice}
    $(\mathbb{Z}^n,\leq)$ is a lattice with meet and join operations being the point-wise $\min$ and $\max$, respectively. 
\end{example}

%

If $P$ is a poset that has a least element $\hat{0}$ and a greatest element $\hat{1}$ ($\hat{0}\leq p\leq \hat{1}$ for every $p$ in $P$), then we say that $P$ is {\bf bounded}.
Note that every non-empty finite lattice is bounded (by taking the meet and join of all the elements). 



A lattice $L$ is said to be a \textbf{modular lattice} if $x\leq b$ implies that \mbox{$x\join(a\meet b)=(x\join a)\meet b$} for every $x,a,b\in L$.
Clearly, every sublattice of a modular lattice is modular.

\begin{example}
    $(\mathbb{Z}^n,\leq)$ is a modular lattice.
\end{example}

%
%


\begin{theorem}(\cite[Lemma 1.2]{TFT1})\label{thm:dist_modular}
    Let $L$ be a locally finite, modular lattice. 
    Let $\rho$ be its rank function.
    Denote the Hasse diagram of $L$ by $H$. 
    Then for every $s,t\in L$:
    $ d_H(s,t) = \rho(s\join t) - \rho(s\meet t).$
\end{theorem}
%

\subsection{The Dominance Order on $\mathbb{Z}^n$}
\begin{definition}\label{def:dom_order}
    The {\bf dominance order} $\trianglelefteq$ on $\mathbb{Z}^n$ is the order in which \linebreak
    $(v_1,\dots,v_n) \tri (u_1,\dots,u_n)$ if $\sum_{j=1}^i v_j\leq \sum_{j=1}^i u_j$
    for every $1\leq i \leq n$.
\end{definition}

\begin{observation}\label{obs:domprodiso}
    $(\mathbb{Z}^n,\trianglelefteq)$ and $(\mathbb{Z}^n,\leq) $ are isomorphic. 
    The following maps, $\chi$ and its inverse $\chi^{-1}$, are poset isomorphisms.
    $$
    \begin{array}{llll}
    \chi: & (\mathbb{Z}^n,\trianglelefteq) & \longrightarrow & (\mathbb{Z}^n,\leq) \\
    & (v_1,\dots,v_{n})    & \longmapsto & (v_1,v_1+v_2,\dots,v_1+\dots+v_{n})
    \end{array}
    $$
    $$
    \begin{array}{llll}
    \chi^{-1}: & (\mathbb{Z}^n,\leq) & \longrightarrow & (\mathbb{Z}^n,\trianglelefteq)\\
    & (v_1,\dots,v_{n}) & \longmapsto & (v_1,v_2-v_1,\dots,v_{n}-v_{n-1}).
    \end{array}
    $$
\end{observation}

Let $E$ be the standard basis of $\mathbb{R}^n$ as a vector space and let $S=E\cup E^{-1}$.
Note that the isomorphisms $\chi$ and $\chi^{-1}$ are also automorphisms of $\mathbb{Z}^n$ as an additive group.
Therefore, they are determined by the mapping of the standard basis $E$.
For example, for $n=3$
$$\chi(u) = 
\left( \begin{array}{ccc}
1 & 0 & 0 \\
1 & 1 & 0 \\
1 & 1 & 1 \end{array} \right) u $$

$$\chi^{-1}(u) = 
\left( \begin{array}{ccc}
1 & 0 & 0 \\
-1 & 1 & 0 \\
0 & -1 & 1 \end{array} \right) u. $$
This implies that the generating set $\chi^{-1}(S)$ of $\mathbb{Z}^n$ describes the covering relation of $(\mathbb{Z}^n,\tri)$, since the generating set $S$ describes the covering relation of $(\mathbb{Z}^n,\leq)$.
In other words, $\mathscr{H}(\mathbb{Z}^n,\tri)$ is the Cayley graph $X(\mathbb{Z}^n, \chi^{-1}(S))$, since the Hasse diagram $\mathscr{H}(\mathbb{Z}^n,\leq)$ is the Cayley graph $X(\mathbb{Z}^n, S)$.

\begin{observation}\label{obs:domcover}
    The Hasse diagram $\mathscr{H}(\mathbb{Z}^n,\tri)$ is the Cayley graph 
    $X(\mathbb{Z}^n, A)$ where $A=\{\pm(1,-1,0,\dots),\pm(0,1,-1,0,\dots),\dots,\pm(0,\dots,0,1)\}$.
\end{observation}

\begin{observation}\label{obs:domprop}
    Let $v=(v_1,\dots, v_n)$ and $u=(u_1,\dots, u_n)$ be two elements in $\mathbb{Z}^n$.
    Examples \ref{example:Z^n_is_graded}, \ref{example:Z^n_is_a_lattice} and Observation \ref{obs:domprodiso} imply the following properties. 
    \begin{enumerate}
        \item $(\mathbb{Z}^n,\tri)$ is a graded poset with the following rank function:
            \begin{equation*}
             \rank_\tri(v) = \rank_{\leq}(\chi(v)) = \sum_{i=1}^{n}\sum_{j=1}^{i}v_j .
            \end{equation*}        
        \item  The distance between $u$ and $v$ in the Hasse diagram of $(\mathbb{Z}^n,\tri)$ is as follows:
            \begin{equation*}
            \begin{split}
d_{\mathscr{H}(\mathbb{Z}^n,\tri)}(v,u) & =  d_{\mathscr{H}(\mathbb{Z}^n,\leq)}(\chi(v),\chi(u)) \\
& = \sum_{i=1}^{n}|\chi(v)_{i}-\chi(u)_{i}| \\
& = \sum_{i=1}^{n}|\sum_{k=1}^{i}(v_k - u_k)|.
            \end{split}
            \end{equation*}        
        \item $(\mathbb{Z}^n,\tri)$ is a lattice with the following meet and join operations:
    \begin{equation*}
        v\meet u =(\alpha_1,\dots,\alpha_{n}) \text{ where } \alpha_k=\min\{\sum\limits_{i=1}^{k}v_i,\sum\limits_{i=1}^{k}u_i\}-\min\{\sum\limits_{i=1}^{k-1}v_i,\sum\limits_{i=1}^{k-1}u_i\};
    \end{equation*}
    \begin{equation*}
        v\join u =(\beta_1,\dots,\beta_{n}) \text{ where }	\beta_k=\max\{\sum\limits_{i=1}^{k}v_i,\sum\limits_{i=1}^{k}u_i\}-\max\{\sum\limits_{i=1}^{k-1}v_i,\sum\limits_{i=1}^{k-1}u_i\}.
    \end{equation*}
    \end{enumerate}
\end{observation}

\chapter{The Yoke Graphs $\Ynm$}\label{chp:YokeGraphs}
In this chapter, we introduce Yoke graphs, the main object of this thesis.
After providing basic definitions in Section \ref{sec:yoke_graphs_defs}, we review three important instances in Section \ref{sec:special_cases} and discuss some group actions on these graphs in Section \ref{sec:group_actions}.

\section{Basic Definitions}\label{sec:yoke_graphs_defs}
For every non-negative integer $n$, denote the set $\{0,\dots, n-1\}$ by $P_n$.

\begin{definition}\label{def:yoke_graph}
	Let $n\geq 1$ and $m\geq 0$ be two integers.
	The \textbf{Yoke graph} $\Ynm$ is a graph with vertices corresponding to all $v=(v_0\dots,v_{m+1})$ in \mbox{$\mathbb{Z}_n\times P_2^m\times \mathbb{Z}_n$} such that $\sum_{i=0}^{m+1}v_i\equiv 0(\bmod n).$
	Two vertices $u$ and $v$ are adjacent in $\Ynm$, denoted $u\sim v$, if there exists $0\leq i\leq m$ such that $u_j=v_j$ for every $j\notin\{i,i+1\}$ and one of the following two cases holds: either $u_i=v_i+1$ and $u_{i+1}=v_{i+1}-1$, or $u_i=v_i-1$ and $u_{i+1}=v_{i+1}+1$. 
\end{definition}

The name ``Yoke graph" is derived from the shoulder yoke, a tool that can be used to carry two buckets.
Based on this analogy, we refer to the entries $v_0$ and $v_{m+1}$ of a vertex $v$ in $\Ynm$ as the \textbf{left bucket} and the \textbf{right bucket} of $v$, respectively.

\begin{convention}\label{conv:cosets_are_integers}
	By Definition \ref{def:yoke_graph}, the buckets of a vertex $v\in\Ynm$ are elements (cosets) in the quotient group $\mathbb{Z}_n$ of $\mathbb{Z}$.
	Throughout this thesis, a bucket is identified with its smallest non-negative representative in $\mathbb{Z}$.
\end{convention}

According to Convention \ref{conv:cosets_are_integers}, the sum $\sum_{i=0}^{m+1}v_i$ in Definition \ref{def:yoke_graph} is a non-negative integer. 
Note that a vertex $u$ in $\Ynm$ is determined by its first (or last) $m+1$ entries, since $\sum_{i=0}^{m+1}u_i\equiv 0(\bmod n)$.
When convenient, we identify a vertex in $\Ynm$ by specifying its first (or last) $m+1$ entries.
We denote the vertex $(0,\dots,0)\in\Ynm$ by $0$.

In the first case of the adjacency relation in Definition \ref{def:yoke_graph}, where $u_i=v_i+1$ and $u_{i+1}=v_{i+1}-1$ for some $0\leq i\leq m$, we say that $u$ is obtained from $v$ by \textbf{shifting} a unit from entry $i+1$ to the \textbf{left}, and write $u=\overleftarrow{s}_{i}(v)$.
In the second case, where $u_i=v_i-1$ and $u_{i+1}=v_{i+1}+1$, we say that $u$ is obtained from $v$ by \textbf{shifting} a unit in entry $i$ to the \textbf{right}, and write $u=\overrightarrow{s}_{i}(v)$.
For an example of a Yoke graph, see $\Ynm[3][3]$ in Figure \ref{fig:Y33}.
Observe that $\Ynm$ is a connected simple graph.

\begin{figure}[hbt]
	\centering
	\begin{tikzpicture}[scale=0.4]
	\tikzstyle{every node}=[draw,circle,fill=white,minimum size=4pt,inner sep=0pt]
	
	\draw (0,0) node (01110) [label=left:\scriptsize(01110)] {}
	-- ++(-120:2.0cm) node (01101) [label=left:\scriptsize(01101)] {}
	-- ++(-120:4.0cm) node (01011) [label=right:\scriptsize(01011)] {}
	-- ++(-120:4.0cm) node (00111) [label=left:\scriptsize(00111)] {}
	-- ++(-120:2.0cm) node (21111) [label=left:\scriptsize(21111)] {}
	
	-- ++(0:2.0cm) node (21102) [label={[shift={(-90:0.8)}]:\scriptsize(21102)}] {}
	-- ++(0:4.0cm) node (21012) [label={[shift={(90:-0.3)}]:\scriptsize(21012)}] {}
	-- ++(0:4.0cm) node (20112) [label={[shift={(-90:0.8)}]:\scriptsize(20112)}] {}
	-- ++(0:2.0cm) node (11112) [label=right:\scriptsize(11112)] {}
	
	-- ++(120:2.0cm) node (11100) [label=right:\scriptsize(11100)] {}
	-- ++(120:4.0cm) node (11010) [label=left:\scriptsize(11010)] {}
	-- ++(120:4.0cm) node (10110) [label=right:\scriptsize(10110)] {}
	-- (01110) {};
	
	\draw (01101) -- ++(-60:2.0cm) node (10101) [label={[shift={(-90:0.8)}]:\scriptsize(10101)}] {}
	-- ++(0:4.0cm) node (11001) [label={[shift={(90:-0.3)}]:\scriptsize(11001)}] {}
	-- ++(0:2.0cm) node (20001) [label=right:\scriptsize(20001)] {}
	
	-- ++(-120:2.0cm) node (20010) [label=right:\scriptsize(20010)] {}
	-- ++(-120:4.0cm) node (20100) [label=left:\scriptsize(20100)] {}
	-- ++(-120:4.0cm) node (21000) [label=right:\scriptsize(21000)] {}
	-- ++(-120:2.0cm) node (00000) [label=right:\scriptsize 0] {}
	
	-- ++(120:2.0cm) node (00012) [label=left:\scriptsize(00012)] {}
	-- ++(120:4.0cm) node (00102) [label=right:\scriptsize(00102)] {}
	-- ++(120:4.0cm) node (01002) [label=left:\scriptsize(01002)] {}
	-- ++(120:2.0cm) node (10002) [label=left:\scriptsize(10002)] {}
	
	-- ++(0:2.0cm) node (10011) [label={[shift={(90:-0.3)}]:\scriptsize(10011)}] {}
	-- (10101) {};
	
	\draw (10101) -- (10110);
	\draw (11001) -- (11010) -- (20010);
	\draw (11100) -- (20100) -- (20112);
	\draw (21000) -- (21012) -- (00012);
	\draw (21102) -- (00102) -- (00111);
	\draw (01002) -- (01011) -- (10011);
	
	\end{tikzpicture}
	\caption{$\Ynm[3][3]$} \label{fig:Y33}
\end{figure}
\pagebreak
\begin{observation}[The case $m\leq 1$]\label{obs:small_m_is_trivial}
	The cases $\Ynm[1][0]$, $\Ynm[2][0]$ and $\Ynm[1][1]$ are graphs on, at most, two vertices (see Figure \ref{fig:small_yoke_graphs}).
	If $m=0$ and $2<n$, then $\Ynm$ is the cycle graph on $n$ vertices and if $m=1$ and $1<n$, then $\Ynm$ is the cycle graph on $2n$ vertices (see Figure \ref{fig:cycle_yoke_graphs}).
\end{observation}

\begin{figure}[hbt]
    
    \begin{center}
        \begin{minipage}{.3\textwidth}
            \centering
            \begin{tikzpicture}[scale=0.4]
            \tikzstyle{every node}=[draw,circle,fill=white,minimum size=4pt,inner sep=0pt]
            \draw (0,0) node (00) [label=left:\scriptsize(00)] {};
            \end{tikzpicture}
        \end{minipage}
        \begin{minipage}{.3\textwidth}
            \centering
            \begin{tikzpicture}[scale=0.4]
            \tikzstyle{every node}=[draw,circle,fill=white,minimum size=4pt,inner sep=0pt]
            \draw (0,0) node (00) [label=left:\scriptsize(00)] {}
            -- ++(+90:2.0cm) node (11) [label=left:\scriptsize(11)] {};
            \end{tikzpicture}
        \end{minipage}
        \begin{minipage}{.3\textwidth}
            \centering
            \begin{tikzpicture}[scale=0.4]
            \tikzstyle{every node}=[draw,circle,fill=white,minimum size=4pt,inner sep=0pt]
            \draw (0,0) node (000) [label=left:\scriptsize(000)] {}
            -- ++(+90:2.0cm) node (010) [label=left:\scriptsize(010)] {};
            \end{tikzpicture}
        \end{minipage}
    \end{center}
    \caption{$\Ynm[1][0]$, $\Ynm[2][0]$ and $\Ynm[1][1]$}
    \label{fig:small_yoke_graphs}
\end{figure}

\begin{figure}[hbt]
    
    \begin{center}
        \begin{minipage}{.4\textwidth}
            \centering
                \begin{tikzpicture}[scale=0.4]
                \tikzstyle{every node}=[draw,circle,fill=white,minimum size=4pt,inner sep=0pt]
                \draw (0,0) node (00) [label=right:\scriptsize(00)] {} 
                -- ++(60:2.0cm) node (51) [label=right:\scriptsize(51)] {}
                -- ++(120:2.0cm) node (42) [label=right:\scriptsize(42)] {}
                -- ++(180:2.0cm) node (33) [label=left:\scriptsize(33)] {}
                -- ++(240:2.0cm) node (24) [label=left:\scriptsize(24)] {}
                -- ++(300:2.0cm) node (15) [label=left:\scriptsize(15)] {};
                \draw (15) -- (00);
                \end{tikzpicture}
        \end{minipage}
        \begin{minipage}{.4\textwidth}
            \centering
                \begin{tikzpicture}[scale=0.4]
                \tikzstyle{every node}=[draw,circle,fill=white,minimum size=4pt,inner sep=0pt]
                \draw (0,0) node (00) [label=right:\scriptsize(000)] {} 
                -- ++(60:2.0cm) node (51) [label=right:\scriptsize(210)] {}
                -- ++(120:2.0cm) node (42) [label=right:\scriptsize(201)] {}
                -- ++(180:2.0cm) node (33) [label=left:\scriptsize(111)] {}
                -- ++(240:2.0cm) node (24) [label=left:\scriptsize(102)] {}
                -- ++(300:2.0cm) node (15) [label=left:\scriptsize(012)] {};
                \draw (15) -- (00);
                \end{tikzpicture}
        \end{minipage}
    \end{center}
    \caption{$\Ynm[6][0]$ and $\Ynm[3][1]$}
    \label{fig:cycle_yoke_graphs}
\end{figure}

\section{Important Examples of Yoke Graphs}\label{sec:special_cases}
In this section, we show that the three flip graphs from Subsection \ref{subsec:background_flip_graphs} are in fact instances of Yoke graphs.
For every one of these graphs, we give its isomorphism with the appropriate instance of a Yoke graph.

\noindent
{\large Colored Triangle-Free Triangulation Graphs\par}
Recall the flip graph of colored triangle-free triangulations $\Gamma_n$ from Subsection \ref{subsec:background_flip_graphs}. In \cite[Definition 2.8]{TFT1}, a bijection between its vertex set $CTFT(n)$ and $\mathbb{Z}_n\times\{0,1\}^{n-4}$ is defined in order to calculate the cardinality of $CTFT(n)$. 
This map is now used to prove that the graphs $\Gamma_n$ and $\Ynm[n][n-4]$ are isomorphic.

\begin{definition}
    Define the map 
    $g: \Gamma_n \rightarrow \Ynm[n][n-4]$ 
    as follows.
    Let $T\in \Gamma_n$.
    If the (short) chord labeled by $0$ in $T$ is $(a-1, a+1)$ for $a\in\mathbb{Z}_n$, then let  $g(T)_0 = a$.
    Extend the definition of $g(T)_i$ iteratively for every $1\leq i\leq n-4$ as follows.
    First, note that if the chord labeled by $i-1$ in $T$ is $(k,t)$, then the chord labeled by $i$ is either $(k-1,t)$ or $(k,t+1)$.
    Now, let $g(T)_i$ be $0$ in the former case and $1$ in the latter.
\end{definition}
For example, assume that the edge $(0,6)$ in Figure \ref{fig:ex_triang_flip} is labeled by $0$ in $\Gamma_8$ (in both triangulations). 
Then these triangulations are mapped by $g$ from $\Gamma_8$ to $\Ynm[8][4]$.
The left triangulation in the figure is mapped to $(7,1,1,0,1,6)$ and the right is mapped to $(7,1,0,1,1,6)$.

\begin{observation}
    It is straightforward to verify that $g$ is a bijection and that a chord flip of the chord labeled by $i\in\{0,\dots, n-4\}$ in $T\in \Gamma_n$ corresponds to the case in the adjacency relation in $\Ynm[n][n-4]$, where a unit is shifted between the entries indexed by $i$ and $i+1$ of $g(T)$.
    This implies that $g$ is a graph isomorphism.
\end{observation}

Note that if $P_n$ would be extended to be defined similarly in the case $n=4$, then the resulting graph $\Gamma_4$ would be a graph on $2$ vertices, whereas $\Ynm[4][0]$ is a graph on $4$ vertices.
The definition of a coloring of a triangle-free triangulation can be modified (from how it appears in Subsection \ref{subsec:background_flip_graphs} and in \cite{TFT1}) such that the linear order is applied to the triangles instead of the chords.
This will extend the correspondence between $\Gamma_n$ and $\Ynm[n][n-4]$ to the case of $\Gamma_4$.
In this case, however, the flip action of the chord in $\Gamma_4$ is not involutive (in contrast to the case $n\geq 5$) and its definition should be refined to include ``direction" of the flip.

\noindent
{\large Arc Permutation Graphs\par}
Recall the flip graph of arc permutations $X_n$ from Subsection \ref{subsec:background_flip_graphs}.
In \cite[Subsection  6.2]{elizalde}, a bijection between its vertex set $\mathcal{A}_n$ and $\mathbb{Z}_n\times\{0,1\}^{n-2}$ is defined in order to embed $X_n$ in the Hasse diagram of the dominance order on $\{0,\dots, n-1\}\times\{0,1\}^{n-2}$. 
This map is now used to prove that the graphs $X_n$ and $\Ynm[n][n-2]$ are isomorphic.

\begin{definition}
    Define the map 
    $f: X_n \rightarrow \Ynm[n][n-2]$ 
    as follows.
    Let $\pi\in X_n$ and set  $f(\pi)_0 = \pi(1)-1$.
    Extend the definition of $f(\pi)_i$ iteratively for every $1\leq i\leq n-2$ as follows.
    Let $I_i$ be the underlying set $\{\pi(1), \dots, \pi(i)\}$ of the interval formed by the prefix of $\pi$ up to $i$.
    First, note that either $(\pi(i+1)-1)\bmod n\in I_i$ or $(\pi(i+1)+1)\bmod n\in I_i$, since the interval $I_i$ must be extended by $\pi(i+1)$.
    Now, let $f(\pi)_i$ be $1$ in the former case and $0$ in the latter.
\end{definition}

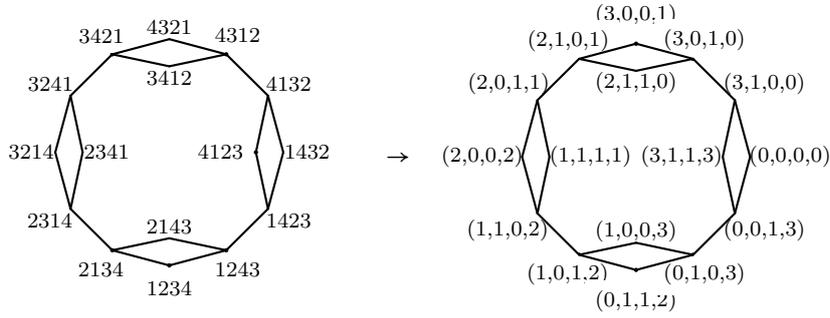
\begin{figure}[hbt]
    
    \begin{center}
        \begin{minipage}{.45\textwidth}
            \centering
            \begin{tikzpicture}[scale=1.5, cap=round, >=latex]
            \addVertex{{90+(360*0)/12}}{1cm}{1.11cm}{4321}{1.27cm}{$\psi=(3,0,0)$}{{90+(360*0)/12}}
            \addVertex{{90+(360*1)/12}}{1cm}{1.19cm}{3421}{1.39cm}{$\psi=(2,1,0)$}{{90+(360*1)/12+4}}
            \addVertex{{90+(360*2)/12}}{1cm}{1.20cm}{3241}{1.42cm}{$\psi=(2,0,1)$}{{90+(360*2)/12-2}}
            \addVertex{{(90+(360*3)/12)}}{1cm}{1.21cm}{3214}{1.39cm}{$\psi=(2,0,0)$}{{(90+(360*3)/12)+6}}
            \addVertex{{90+(360*4)/12}}{1cm}{1.21cm}{2314}{1.37cm}{$\psi=(1,1,0)$}{{90+(360*4)/12+4}}
            \addVertex{{90+(360*5)/12}}{1cm}{1.2cm}{2134}{1.38cm}{$\psi=(1,0,1)$}{{(90+(360*5)/12)}}
            \addVertex{{90+(360*6)/12}}{1cm}{1.2cm}{1234}{1.35cm}{$\psi=(0,1,1)$}{{90+(360*6)/12}}
            \addVertex{{90+(360*7)/12}}{1cm}{1.2cm}{1243}{1.38cm}{$\psi=(0,1,0)$}{{90+(360*7)/12}}
            \addVertex{{90+(360*8)/12}}{1cm}{1.21cm}{1423}{1.37cm}{$\psi=(0,0,1)$}{{90+(360*8)/12-4}}
            \addVertex{{90+(360*9)/12}}{1cm}{1.21cm}{1432}{1.39cm}{$\psi=(0,0,0)$}{{(90+(360*9)/12)-6}}
            \addVertex{{90+(360*10)/12}}{1cm}{1.21cm}{4132}{1.37cm}{$\psi=(3,1,0)$}{{90+(360*10)/12+4}}
            \addVertex{{90+(360*11)/12}}{1cm}{1.2cm}{4312}{1.43cm}{$\psi=(3,0,1)$}{{90+(360*11)/12-4}}
            
            \addVertex{{90+(360*0)/12}}{0.76cm}{0.66cm}{3412}{0.51cm}{$\psi=(2,1,1)$}{{90+(360*0)/12}}
            \addVertex{{(90+(360*3)/12)}}{0.76cm}{0.55cm}{2341}{0.46cm}{$\psi=(1,1,1)$}{{(90+(360*3)/12)+20}}
            \addVertex{{90+(360*6)/12}}{0.76cm}{0.66cm}{2143}{0.51cm}{$\psi=(1,0,0)$}{{90+(360*6)/12}}
            \addVertex{{90+(360*9)/12}}{0.76cm}{0.45cm}{4123}{0.46cm}{$\psi=(3,1,1)$}{{(90+(360*9)/12)-20}}
            
            \draw[thick] ({90+(360*0)/12}:1cm) -- ({90+(360*1)/12}:1cm);
            \draw[thick] ({90+(360*1)/12}:1cm) -- ({90+(360*2)/12}:1cm);
            \draw[thick] ({90+(360*2)/12}:1cm) -- ({90+(360*3)/12}:1cm);
            \draw[thick] ({90+(360*3)/12}:1cm) -- ({90+(360*4)/12}:1cm);
            \draw[thick] ({90+(360*4)/12}:1cm) -- ({90+(360*5)/12}:1cm);
            \draw[thick] ({90+(360*5)/12}:1cm) -- ({90+(360*6)/12}:1cm);
            \draw[thick] ({90+(360*6)/12}:1cm) -- ({90+(360*7)/12}:1cm);
            \draw[thick] ({90+(360*7)/12}:1cm) -- ({90+(360*8)/12}:1cm);
            \draw[thick] ({90+(360*8)/12}:1cm) -- ({90+(360*9)/12}:1cm);
            \draw[thick] ({90+(360*9)/12}:1cm) -- ({90+(360*10)/12}:1cm);
            \draw[thick] ({90+(360*10)/12}:1cm) -- ({90+(360*11)/12}:1cm);
            \draw[thick] ({90+(360*11)/12}:1cm) -- ({90+(360*0)/12}:1cm);
            
            \draw[thick] ({90+(360*1)/12}:1cm) -- ({90+(360*0)/12}:0.76cm);
            \draw[thick] ({90+(360*11)/12}:1cm) -- ({90+(360*0)/12}:0.76cm);
            
            \draw[thick] ({90+(360*2)/12}:1cm) -- ({(90+(360*3)/12)}:0.76cm);
            \draw[thick] ({90+(360*4)/12}:1cm) -- ({(90+(360*3)/12)}:0.76cm);
            
            \draw[thick] ({90+(360*5)/12}:1cm) -- ({90+(360*6)/12}:0.76cm);
            \draw[thick] ({90+(360*7)/12}:1cm) -- ({90+(360*6)/12}:0.76cm);
            
            \draw[thick] ({90+(360*8)/12}:1cm) -- ({90+(360*9)/12}:0.76cm);
            \draw[thick] ({90+(360*10)/12}:1cm) -- ({90+(360*9)/12}:0.76cm);
            \end{tikzpicture} 
        \end{minipage}$\rightarrow$
        \begin{minipage}{.45\textwidth}
            \centering
            \begin{tikzpicture}[scale=1.5, cap=round, >=latex]
            \addVertex{{90+(360*0)/12}}{1cm}{1.26cm}{(3,0,0,1)}{1.27cm}{$\psi=(3,0,0)$}{{90+(360*0)/12}}
            \addVertex{{90+(360*1)/12}}{1cm}{1.19cm}{(2,1,0,1)}{1.39cm}{$\psi=(2,1,0)$}{{90+(360*1)/12+4}}
            \addVertex{{90+(360*2)/12}}{1cm}{1.29cm}{(2,0,1,1)}{1.42cm}{$\psi=(2,0,1)$}{{90+(360*2)/12-2}}
            \addVertex{{(90+(360*3)/12)}}{1cm}{1.35cm}{(2,0,0,2)}{1.39cm}{$\psi=(2,0,0)$}{{(90+(360*3)/12)+6}}
            \addVertex{{90+(360*4)/12}}{1cm}{1.3cm}{(1,1,0,2)}{1.37cm}{$\psi=(1,1,0)$}{{90+(360*4)/12+4}}
            \addVertex{{90+(360*5)/12}}{1cm}{1.2cm}{(1,0,1,2)}{1.38cm}{$\psi=(1,0,1)$}{{(90+(360*5)/12)}}
            \addVertex{{90+(360*6)/12}}{1cm}{1.27cm}{(0,1,1,2)}{1.35cm}{$\psi=(0,1,1)$}{{90+(360*6)/12}}
            \addVertex{{90+(360*7)/12}}{1cm}{1.2cm}{(0,1,0,3)}{1.38cm}{$\psi=(0,1,0)$}{{90+(360*7)/12}}
            \addVertex{{90+(360*8)/12}}{1cm}{1.3cm}{(0,0,1,3)}{1.37cm}{$\psi=(0,0,1)$}{{90+(360*8)/12-4}}
            \addVertex{{90+(360*9)/12}}{1cm}{1.35cm}{(0,0,0,0)}{1.39cm}{$\psi=(0,0,0)$}{{(90+(360*9)/12)-6}}
            \addVertex{{90+(360*10)/12}}{1cm}{1.3cm}{(3,1,0,0)}{1.37cm}{$\psi=(3,1,0)$}{{90+(360*10)/12+4}}
            \addVertex{{90+(360*11)/12}}{1cm}{1.2cm}{(3,0,1,0)}{1.43cm}{$\psi=(3,0,1)$}{{90+(360*11)/12-4}}
            
            \addVertex{{90+(360*0)/12}}{0.76cm}{0.66cm}{(2,1,1,0)}{0.51cm}{$\psi=(2,1,1)$}{{90+(360*0)/12}}
            \addVertex{{(90+(360*3)/12)}}{0.76cm}{0.40cm}{(1,1,1,1)}{0.46cm}{$\psi=(1,1,1)$}{{(90+(360*3)/12)+20}}
            \addVertex{{90+(360*6)/12}}{0.76cm}{0.66cm}{(1,0,0,3)}{0.51cm}{$\psi=(1,0,0)$}{{90+(360*6)/12}}
            \addVertex{{90+(360*9)/12}}{0.76cm}{0.40cm}{(3,1,1,3)}{0.46cm}{$\psi=(3,1,1)$}{{(90+(360*9)/12)-20}}
            
            \draw[thick] ({90+(360*0)/12}:1cm) -- ({90+(360*1)/12}:1cm);
            \draw[thick] ({90+(360*1)/12}:1cm) -- ({90+(360*2)/12}:1cm);
            \draw[thick] ({90+(360*2)/12}:1cm) -- ({90+(360*3)/12}:1cm);
            \draw[thick] ({90+(360*3)/12}:1cm) -- ({90+(360*4)/12}:1cm);
            \draw[thick] ({90+(360*4)/12}:1cm) -- ({90+(360*5)/12}:1cm);
            \draw[thick] ({90+(360*5)/12}:1cm) -- ({90+(360*6)/12}:1cm);
            \draw[thick] ({90+(360*6)/12}:1cm) -- ({90+(360*7)/12}:1cm);
            \draw[thick] ({90+(360*7)/12}:1cm) -- ({90+(360*8)/12}:1cm);
            \draw[thick] ({90+(360*8)/12}:1cm) -- ({90+(360*9)/12}:1cm);
            \draw[thick] ({90+(360*9)/12}:1cm) -- ({90+(360*10)/12}:1cm);
            \draw[thick] ({90+(360*10)/12}:1cm) -- ({90+(360*11)/12}:1cm);
            \draw[thick] ({90+(360*11)/12}:1cm) -- ({90+(360*0)/12}:1cm);
            
            \draw[thick] ({90+(360*1)/12}:1cm) -- ({90+(360*0)/12}:0.76cm);
            \draw[thick] ({90+(360*11)/12}:1cm) -- ({90+(360*0)/12}:0.76cm);
            
            \draw[thick] ({90+(360*2)/12}:1cm) -- ({(90+(360*3)/12)}:0.76cm);
            \draw[thick] ({90+(360*4)/12}:1cm) -- ({(90+(360*3)/12)}:0.76cm);
            
            \draw[thick] ({90+(360*5)/12}:1cm) -- ({90+(360*6)/12}:0.76cm);
            \draw[thick] ({90+(360*7)/12}:1cm) -- ({90+(360*6)/12}:0.76cm);
            
            \draw[thick] ({90+(360*8)/12}:1cm) -- ({90+(360*9)/12}:0.76cm);
            \draw[thick] ({90+(360*10)/12}:1cm) -- ({90+(360*9)/12}:0.76cm);
            \end{tikzpicture} 
        \end{minipage}
    \end{center}
    \caption{$f:X_4\rightarrow\Ynm[4][2]$}
    \label{fig:ex_arc_permutation_isomorphism}
\end{figure}

\begin{observation}
    It is straightforward to verify that $f$ is a bijection and that right multiplication of $\pi\in X_n$ by the transposition $(i,i+1)$ (for \mbox{$1\leq i\leq n-1$}) corresponds to the case in the adjacency relation in $\Ynm[n][n-2]$, where a unit is shifted between the entries indexed by $i-1$ and $i$ of $f(\pi)$.
    This implies that $f$ is a graph isomorphism.
\end{observation}

\noindent
{\large Geometric Caterpillar Graphs\par}
Recall the flip graph of caterpillars $\mathcal{G}_n$ from Subsection \ref{subsec:background_flip_graphs}.
\begin{definition}\label{def:order_caterpillar}
    Let $C\in \mathcal{G}_n$ and let $I(C)=[a,a+1,\dots,a+k]\subseteq \mathbb{Z}_n$ be the interval induced by the spine of $C$ ($2\leq k$).
    Note that $a$ and $a+k$ are leaves of $C$.
    Define a complete ordering $S(C) = (s_0,s_1,\dots, s_{n-3})$ on $n-2$ vertices of $C$ as follows.
    Set $s_0=a$.
    Define the interval $I_0(C)$ to be $[a,a+1]\subset\mathbb{Z}_n$.
    Define $I_i(C)$ and $s_{i}$ iteratively for every $1\leq i\leq n-3$ as follows.
    Assume that $I_{i-1}(C)=[l,\dots,r]$.
    If $(l-1,r)$ is an edge in $C$, then $I_i(C)=[l-1,l,\dots,r]$ and $s_{i}=l-1$.
    Otherwise, $I_i(C)=[l,\dots,r,r+1]$ and $s_{i}=r+1$.
\end{definition}

Note that in the $i$-th iteration in Definition \ref{def:order_caterpillar}, if $I_{i-1}=[l,r]$ and $(l-1,r)$ is not an edge in $C$, then $C$ has no leaf connected to $r$ that is not in $I_{i-1}$.
This implies that the iterations in the definition induce a sequence $S(C)$ of distinct $n-2$ vertices and $n-2$ intervals that monotonously increase in size.

For example, let $C$ be the leftmost caterpillar in Figure \ref{fig:ex_cat_flip}.
Then $S(C)=(7, 6, 5, 4, 1, 3)$,
$I_0(C)=[7, 0]$, $I_1(C)=[6, 7, 0]$, $I_2(C)=[5, 6, 7, 0]$, $I_3(C)=[4, 5, 6, 7, 0]$, $I_4(C)=[4, 5, 6, 7, 0, 1]$ and $I_{5}(C)=[3, 4, 5, 6, 7, 0, 1]$.

\begin{definition}\label{def:caterpillar_iso}
    Let $C\in \mathcal{G}_n$.
    Let $I_i$ (for $0\leq i\leq n-3$) and let $S(C) = (s_0,s_1,\dots, s_{n-3})$ be as in Definition \ref{def:order_caterpillar}.
    Define the map 
    $h: \mathcal{G}_n \rightarrow \Ynm[n][n-3]$ as follows.
    Set $h(C)_0 = s_0$.
    Extend the definition of $h(C)_i$ iteratively for every $1\leq i\leq n-3$ as follows.
    Assume that $I_{i-1}=[l,r]$.
    Note that either $s_{i}=l-1$ or $s_{i}=r+1$.
    Let $h(C)_i$ be $1$ in the former case and $0$ in the latter.    
\end{definition}

For example, let $C\in \mathcal{G}_8$ be the leftmost caterpillar in Figure \ref{fig:ex_cat_flip}.
Then $h$ maps $C$ to $\Ynm[8][5]$ and $h(C)=(7, 1, 1, 1, 0, 1, 5)$.

\begin{observation}
    It is straightforward to verify that $h$ is a bijection and that shifting an edge incident with a leaf along the spine in $C\in \mathcal{G}_n$ corresponds to a unit shift between two entries indexed by $i$ and $i+1$ of $h(C)$ for some $0\leq i\leq m$.
    This implies that $h$ is a graph isomorphism.
\end{observation}

\section{Group Actions on $\Ynm$}\label{sec:group_actions}

\begin{definition}\label{def:flip_action} 
    Let $n\geq 1$ and $m\geq 1$. Define the set $\Snm=\{s_i:0\leq i\leq m\}$ where $s_i$ is the map $s_i:\Ynm\rightarrow\Ynm$ defined as follows:
    \[
    s_0(v)= 
    \begin{cases}
    \overleftarrow{s}_{0}(v),& \text{if } v_{1}=1\\
    \overrightarrow{s}_{0}(v),& \text{if } v_{1}=0
    \end{cases}, \; \; \; \; \;
    s_m(v)= 
    \begin{cases}
    \overleftarrow{s}_{m}(v),& \text{if } v_{m}=0\\
    \overrightarrow{s}_{m}(v),& \text{if } v_{m}=1
    \end{cases}
    \]
    and for every $1\leq i\leq m-1$
    \[
    s_i(v)=(v_0,\dots,v_{i+1},v_{i},\dots,v_{m+1})=
    \begin{cases}
    \overleftarrow{s}_{i}(v),& \text{if } (v_i, v_{i+1})=(0,1)\\
    \overrightarrow{s}_{i}(v),& \text{if } (v_i, v_{i+1})=(1,0)\\
    v,& \text{if } v_i=v_{i+1}
    \end{cases}.
    \]
\end{definition}

\begin{observation}
    Every $s_i$ in $\Snm$ is a distinct involutive permutation of the vertex set of $\Ynm$.
\end{observation}

\begin{definition}
    Denote the group generated by $\Snm$ by $\Gnm$.
\end{definition}

$\Gnm$ is a subgroup of the symmetric group on the $n2^{m}$ vertices of $\Ynm$.
Clearly, every edge $e$ in $\Ynm$ corresponds with the action of a unique generator $s_i\in \Snm$ on the endpoints of $e$. 
Therefore, $\Gnm$ defines a faithful group action on the vertices of $\Ynm$. 
Note that this action is transitive, since $\Ynm$ is connected.
This, combined with Remark \ref{rem:shcreier_graph_iso}, implies the following observation.

\begin{observation}
    Let $n\geq 1$ and $m\geq 1$. 
    Then $\Ynm$ is the Schreier graph of $\Gnm$ with respect to the generating set $\Snm$ and the action of $\Gnm$ on $\Ynm$.
\end{observation}

Recall that the affine Weyl group of type $\tilde{C}_m$ is the group generated by $S=\{\sigma_0,\dots,\sigma_{m}\}$ ($m\geq 2$) whose Coxeter relations are the following (see e.g. \nolinebreak \cite{humphreys}):
\begin{align*}
\sigma_i^2=Id & \quad \mbox{for all } 0\leq i\leq m;        \\
(\sigma_i\sigma_j)^2=Id & \quad \mbox{for all } |j-i|>1; \\
(\sigma_i\sigma_{i+1})^3=Id & \quad \mbox{for all } 1\leq i\leq m-2; \\
(\sigma_0\sigma_{1})^4=(\sigma_{m-1}\sigma_{m})^4=Id. & \quad 
\end{align*}
\pagebreak
\begin{proposition}\label{prop:coxeter_relations_snm}
    Let $n\geq 1$ and $m\geq 2$.
    The generators $\Snm$ of $\Gnm$ satisfy the Coxeter relations of $\tilde{C}_m$.
\end{proposition}
\begin{proof}
    The generators in $\Snm$ are involutive and $s_i$ and $s_j$ clearly commute whenever $|j-i|>1$.
    Therefore the first two relations, $s_i^2=Id$ for all \mbox{$0\leq i\leq m$} and $(\sigma_i\sigma_j)^2=Id$ for all $|j-i|>1$, are true.
    It can be verified that $s_is_{i+1}s_i=s_{i+1}s_is_{i+1}$ for every $1\leq i\leq m-2$.
    This proves the third relation: $(s_is_{i+1})^3=Id$ for all $1\leq i\leq m-2$.
    Note that the orbit of every vertex in $\Ynm$ under the action of $\left< s_0s_1 \right>$ (and $\left< s_{m-1}s_m \right>$) is of size $4$. 
    This proves the final relation: $(s_0s_{1})^4=(s_{m-1}s_{m})^4=Id$.
\end{proof}

\begin{corollary}\label{cor:quotient_of_weyl}
    Let $n\geq 1$ and $m\geq 2$.
    $\Gnm$ is a quotient of the affine Weyl group $\tilde{C}_m$.
\end{corollary}

\begin{corollary}\label{cor:schreier_of_weyl_group}
    Let $n\geq 1$, $m\geq 2$ and let $S=\{\sigma_0,\dots,\sigma_{m}\}$ be the generators of the affine Weyl group of type $\tilde{C}_m$.
    \begin{enumerate}
        \item For every $0\leq i\leq m$, define the action of $\sigma_i$ on $\Ynm$ by $\sigma_i(v)=s_i(v)$ as in Definition \ref{def:flip_action} . 
        Then, by Proposition \ref{prop:coxeter_relations_snm}, the action of $\tilde{C}_m$ on $\Ynm$, induced by extending the actions of $\sigma_i$ multiplicatively, is well defined.
        \item $\Ynm$ is the Schreier graph of $\tilde{C}_m$ with respect to the generating set $S=\{\sigma_0,\dots,\sigma_{m}\}$ and the above action of $\tilde{C}_m$ on $\Ynm$.
    \end{enumerate}
\end{corollary}

Corollary \ref{cor:schreier_of_weyl_group} generalizes \cite[Proposition 3.2]{TFT1} and \cite[Corollary 10.4]{elizalde}, by which the colored triangle-free triangulation graphs $\Gamma_n$ and the arc permutation graphs $X_n$ are Schreier graphs of the affine Weyl group of type $\tilde{C}_{n-4}$ for $n>5$ and $\tilde{C}_{n-2}$ for $n>3$, respectively.

\chapter{The Eccentricity of $0$ in $\Ynm$}\label{chp:ecc_zero_Ynm}
The main purpose of this chapter is to develop the theory that will allow us to state and prove the following theorem.
\begin{theorem}\label{thm:ynm_ecc} Let $n\geq 1$ and $m\geq 0$.
    \begin{enumerate}
	\item If $n=1$, then $\ecc_{\Ynm}(0)=\binom{\lceil \frac{m}{2}\rceil + 1}{2} + \binom{\lfloor \frac{m}{2}\rfloor + 1}{2}$. 
	\item If $0\leq m\leq n$, then $\ecc_{\Ynm}(0) = \lfloor\frac{n(m+1)}{2}\rfloor$.
    \item If $2\leq n\leq m$, let $\dz = \binom{\lfloor\frac{m+n}{2}\rfloor+1}{2} + \binom{\lceil\frac{m-n}{2}\rceil+1}{2}$. Then
	\begin{enumerate}
		\item if either $2\divides (m-n)$ or $n\leq\lceil\frac{m+1}{2}\rceil$, then $\ecc_{\Ynm}(0) = \dz$;
		\item otherwise, $\ecc_{\Ynm}(0) = \dz + n-\lceil\frac{m+1}{2}\rceil$.
	\end{enumerate}
	\end{enumerate}
\end{theorem}
\begin{proof}
    The theorem is the combined result of Observation \ref{obs:m_leq_2} and \linebreak Lemmas \ref{lem:corner_case_neqo}, \ref{lem:eccentricity_in_I0},
    \ref{lem:dist_of_uz}, \ref{lem:close_to_middle_pivot} and \ref{lem:middle_interval_is_n} that appear in the following sections.
\end{proof}

\begin{observation}[The case $m\leq 1$]\label{obs:m_leq_2}
By Observation \ref{obs:small_m_is_trivial}, Theorem \ref{thm:ynm_ecc} clearly holds for $m\leq 1$.
\end{observation}
In the rest of this chapter, unless explicitly stated otherwise, we assume that $m\geq 2$.

\section{The Shift Direction Lemma}
Let $\Snm$ be the set of generators of $\Gnm$ as in Definition \ref{def:flip_action}.
A \textbf{word $w$ in the letters} $\Snm$ is a sequence $f_d\cdots f_1$ where $f_1,\dots,f_d\in\Snm$.
Let $P$ be a path $(v=v^0\sim\dots \sim v^d=u)$ in $\Ynm$. 
Clearly, there is a unique word $f_d\cdots f_1$ such that $f_t(v^{t-1})=v^t$ for every $1\leq t\leq d$ or, equivalently,  $f_t\cdots f_1(v)=v^t$ for every $1\leq t\leq d$.
We say that $f_d\cdots f_1$ is the \textbf{word corresponding to the path $P$ from $v$ to $u$} (or corresponding to the path $P$ starting at $v$). 

Note that a word $f_d\cdots f_1$ does not correspond to a path starting at $v$ if $f_{t}\cdots f_1(v) = f_{t+1}f_t\cdots f_1(v)$ for some $1\leq t\leq d-1$.
For example, the word $s_1s_0$ corresponds to the path $(v=(0, 1, 1, 1)\sim (1, 0, 1, 1)\sim (1, 1, 0, 1))$ in $\Ynm[3][2]$ starting at $v$.
However, $s_0s_1$ does not correspond to a path starting at $v$, since $v$ is a fixed point of $s_1$.

Let $f_d\cdots f_1$ be the word corresponding to a path $(v=v^0\sim\dots \sim v^d)$ in $\Ynm$ starting at $v$.
For convenience, if $v^{t}=\overleftarrow{s}_i(v^{t-1})$ for some $0\leq i\leq m$ and $1\leq t\leq d$, then we write $f_t=\overleftarrow{s}_i$.
Similarly, if $v^{t}=\overrightarrow{s}_i(v^{t-1})$, then we write $f_t=\overrightarrow{s}_i$.

\begin{lemma}[\sc Shift Direction Lemma] \label{lem:shift_direction_lemma_Ynm}
	Let $f_d\cdots f_1$ be the word corresponding to a geodesic in $\Ynm$.
	For every $0\leq i\leq m$, all of the instances of $s_i\in \Snm$ in $f_d\cdots f_1$ shift in the same direction.
\end{lemma}

\begin{proof}
	Let $v$ and $u$ be two vertices in $\Ynm$.
	Assume to the contrary that there exist words corresponding to geodesics from $v$ to $u$ not satisfying the lemma.
	Denote the set of such words by $\mathcal{W}$.
	For every word $f_d\cdots f_1$ in $\mathcal{W}$, there exist $0\leq i\leq m$ and $1\leq t_1, t_2\leq d$ such that $f_{t_1}=\overrightarrow{s}_{i}$ and $f_{t_2}=\overleftarrow{s}_{i}$.
	Let $w=f_d\cdots f_1$ be a word in $\mathcal{W}$ in which $\min\{|t_2-t_1|: f_{t_1}=\overrightarrow{s}_{i},f_{t_2}=\overleftarrow{s}_{i}\}$ is minimal. 
    Assume without loss of generality that $t_1<t_2$.
    
	If $t_2=t_1+1$, then the word obtained by deleting both $f_{t_1}$ and $f_{t_2}$ from $f_d\cdots f_1$ is a word corresponding to a shorter path from $v$ to $u$, contradicting the minimality of $d$.
	Therefore, we can assume that $t_2-t_1>1$.
    
    If $f_{t_1+1}=s_j$ for some $0\leq j\leq m$ such that $|i-j|>1$, then $f_{t_1}$ and $f_{t_1+1}$ commute as functions, and the word obtained by interchanging $f_{t_1}$ and $f_{t_1+1}$ in $w$ corresponds to a path from $v$ to $u$.
    This contradicts the minimality of $t_2-t_1$ in the choice of $w$.
    Therefore, $f_{t_1+1}\in\{\overrightarrow{s}_{i-1}, \overrightarrow{s}_{i+1}\}$.
    	
	Let $v'=f_{t_1+1}\cdots f_1(v)$. 
    If $f_{t_1+1}=\overrightarrow{s}_{i-1}$, then $v'_i=1$. 
	Since $f_{t_2}=\overleftarrow{s}_i$, there must exist some $t_1+1<k<t_2$ such that $f_k=\overleftarrow{s}_{i-1}$ or $f_k=\overrightarrow{s}_{i}$.
	Similarly, if $f_{t_1+1}=\overrightarrow{s}_{i+1}$, then $v'_{i+1}=0$ and there must exist some $t_1+1<k<t_2$ such that $f_k=\overleftarrow{s}_{i+1}$ or $f_k=\overrightarrow{s}_{i}$.
	All of the cases contradict the minimality of $t_2-t_1$.
\end{proof}

\section{Pivots and Walls}\label{sec:pivot_paths_Ynm}

\begin{definition}[\sc Pivot]
	A \textbf{pivot} of a vertex $v\in\Ynm$ is an integer $-1\leq p\leq m+1$ such that $\sum_{i=0}^{p}v_i$ is divisible by $n$. 
	We call $-1$ and $m+1$ the \textbf{outer pivots} of $v$; they are pivots of every $v\in\Ynm$. Every other pivot, if it exists, is called an \textbf{inner pivot}.
	Denote the set of pivots of $v$ by $\piv(v)$.
\end{definition}

For example, in $\Ynm[3][5]$, $\piv((0,1,1,0,1,1,2))=\{-1,0,4,6\}$.
\begin{definition}[\sc Wall]
	Let $P$ be a path from $v$ to $0$ in $\Ynm$ and let $f_d\cdots f_1$ be the word corresponding to $P$.
	An \textbf{inner wall} of $P$ is an integer $0\leq p\leq m$ such that $s_p$ does not appear in $f_d\cdots f_1$.
	$-1$ is a \textbf{left outer wall} of $P$ if $\overleftarrow{s}_0$ does not appear in $f_d\cdots f_1$. 
	Similarly, $m+1$ is a \textbf{right outer wall} of $P$ if $\overrightarrow{s}_m$ does not appear in $f_d\cdots f_1$.
	Finally, we say that $-1\leq p\leq m+1$ is a \textbf{wall} of $P$ if it is either an inner wall or an outer wall of $P$.
\end{definition}

Clearly, paths from $v$ to $0$ with a wall $p$ exist if and only if $p$ is a pivot of $v$.

\begin{definition}[\sc Pivot Path]
	Let $p$ be a pivot of a vertex $v$ in $\Ynm$.
	We say that a path $P$ from $v$ to $0$ is a \textbf{$p$-pivot path} of $v$, if $P$ is shortest among the paths from $v$ to $0$ with a wall $p$.
	Denote the length of a $p$-pivot path of $v$ by $\ps_p(v)$.
	We say that $P$ is a \textbf{pivot path}, if it is a $p$-pivot path for some $p$. 
\end{definition}

For example, let $v=(2,0,1,1,2)$ in $\Ynm[3][3]$. 
$\overrightarrow{s}_3\overleftarrow{s}_0\overleftarrow{s}_1$ is a word corresponding to a $2$-pivot path $P$ of $v$: $P=(v=(2,0,1\|1,2)\sim (2,1,0\|1,2)\sim (0,0,0\|1,2)\sim (0,0,0\|0,0)=0)$ where we added the symbol $\|$ to visualize the wall $2$ of $P$.
For an example of an outer pivot path with a wall -1, let $v=(1,0,1,0)$ in $\Ynm[2][2]$. 
$\overrightarrow{s}_2\overrightarrow{s}_1\overrightarrow{s}_0\overrightarrow{s}_2$ is a word corresponding to a $(-1)$-pivot path $P$ of $v$: $P=(v=(\|1,0,1,0)\sim (\|1,0,0,1)\sim (\|0,1,0,1)\sim (\|0,0,1,1)\sim (\|0,0,0,0)=0)$.

\begin{lemma}\label{lem:geodesic_is_a_pivot_path_ynm}
    Every geodesic in $\Ynm$ ending at $0$ is a pivot path (namely, has a wall).
\end{lemma}
\begin{proof}
    Let $w=f_d\cdots f_1$ be the word corresponding to a geodesic $P$ in $\Ynm$ from $v$ to $0$.
    Assume to the contrary that $P$ has no walls.
    Therefore, both $\overleftarrow{s}_0$ and $\overrightarrow{s}_m$ appear in $w$ (since $P$ has no outer walls);
	and by the Shift Direction Lemma \ref{lem:shift_direction_lemma_Ynm}, exactly one of $\{\overleftarrow{s}_i, \overrightarrow{s}_i\}$ appears in $w$ (at least once) for every $0\leq i\leq m$.
 	Let $1\leq i\leq m$ be minimal such that $\overrightarrow{s}_i$ appears in $w$.
	Therefore, both $\overleftarrow{s}_{i-1}$ and $\overrightarrow{s}_{i}$ appear in $w$. 
	This is a contradiction, since it implies that, throughout $P$, at least two units are shifted from the entry $v_i$ and no units are shifted to $v_i$.
\end{proof}

\begin{observation} By Lemma \ref{lem:geodesic_is_a_pivot_path_ynm},
$$\ecc_{\Ynm}(0)=\max_vd(v,0)=\max_v\min_p\ps_p(v).$$
\end{observation}

\begin{definition}
    Let $v\in\Ynm$ and $p\in\piv(v)$.
    If (one, equivalently all) $p$-pivot paths of $v$ are also geodesics between $v$ and $0$, then we say that $p$ is a \textbf{geodesic pivot} of $v$.
\end{definition}

The Shift Direction Lemma \ref{lem:shift_direction_lemma_Ynm} deals with geodesics, or equivalently (by Lemma \ref{lem:geodesic_is_a_pivot_path_ynm}) with $p$-pivot paths for geodesic pivots $p$.
It can be extended to arbitrary pivot paths, with essentially the same proof.

\begin{lemma}[\sc Pivot Shift Direction Lemma]\label{lem:pivot_shift_direction_lemma_Ynm}
	Let $f_d\cdots f_1$ be the word corresponding to a pivot path in $\Ynm$.
	For every $0\leq i\leq m$, all of the instances of $s_i$ in $f_d\cdots f_1$ shift in the same direction.
\end{lemma}

\begin{observation}\label{obs:shift_direction_pivot_sides}
Let $P$ be a $p$-pivot path of $v$, and let $w$ be the word corresponding to $P$.
Then, for each $i<p$, if $s_i$ appears in $w$, then it appears as $\overleftarrow{s}_i$;
and, for each $i>p$, if $s_i$ appears in $w$, then it appears as $\overrightarrow{s}_i$.
\end{observation}

\begin{corollary}\label{cor:pivot_path_len}
    Let $v\in\Ynm$ and let $p\in\piv(v)$.
    Then 
    $$\ps_p(v)=\sum_{i=1}^{p}iv_i+\sum_{i=p+1}^{m}(m+1-i)v_i.$$
    In particular, if $p$ is an inner pivot, then
    $$\ps_p(v)\leq \binom{p+1}{2}+\binom{m-p+1}{2}.$$
\end{corollary}

\begin{definition}
For $v\in \Ynm$ define:
\begin{enumerate}
    \item $p_l(v)=\max\{p\in \piv(v):p\leq\frac{m}{2}\}$;
    \item $p_r(v)=\min\{p\in \piv(v):p>\frac{m}{2}\}$.
\end{enumerate}
We abbreviate $p_l$, $p_r$ when $v$ is evident.
\end{definition}

Note that if $p_1$ and $p_2$ are two pivots of $v$ such that $-1\leq p_1 < p_2 \leq \frac{m}{2}$, then $\ps_{p_2}(v) \leq \ps_{p_1}(v)$.
Similarly, if $\frac{m}{2}< p_1 < p_2\leq m+1$, then $\ps_{p_1}(v) \leq \ps_{p_2}(v)$.
This implies the following observation.

\begin{observation}\label{obs:closer_pivot_is_better}
	Let $v\in \Ynm$. Then 
	$$d(v,0)=\min_p\ps_p(v)=\min\{\ps_{p_l}(v), \ps_{p_r}(v)\}.$$
\end{observation}

\begin{fact}\label{fac:split_sum_center}
	Let $m,p,q$ be integers such that $0\leq p\leq m$, $0\leq q\leq m$ and $|p-\frac{m}{2}| < |q-\frac{m}{2}|$. 
    Then $\sum_{i=1}^{p}i + \sum_{i=1}^{m-p}i < \sum_{i=1}^{q}i + \sum_{i=1}^{m-q}i$.
\end{fact}

\begin{lemma}\label{lem:ecc_lower_bound_ynm}
	For any $n\geq 1$ and $m\geq 0$, $\ecc_{\Ynm}(0)\geq\lfloor\frac{n(m+1)}{2} \rfloor$.
\end{lemma}
\begin{proof}
	If $m\leq 1$, this lemma follows from Observation \ref{obs:small_m_is_trivial}.
	If $n=1$, it follows from Lemma \ref{lem:corner_case_neqo}.
	Let $n,m\geq 2$, and let $u\in\Ynm$ be defined as follows:
	$u_0=u_{m+1}=\lfloor\frac{n}{2}\rfloor$. 
	If $n$ is even, then $u_i=0$ for all $1\leq i\leq m$, and if $n$ is odd, then $u_{\lfloor\frac{m+1}{2}\rfloor}=1$ and $u_i=0$ for all other $1\leq i\leq m$.
	For example:
\begin{center}
\begin{tabular}{ l l l }
	\hline 
	n & m & u \\ \hline                        
	4 & 2 & $(2,0,0,2)$ \\
	4 & 3 & $(2,0,0,0,2)$ \\
	5 & 4 & $(2,0,1,0,0,2)$ \\
	5 & 5 & $(2,0,0,1,0,0,2)$ \\
	\hline  
\end{tabular}	 
\end{center}

	In all of these cases, $u$ has no inner pivots, that is, $\piv(u)=\{-1,m+1\}$.
	Therefore, by Observation \ref{obs:closer_pivot_is_better}, $d(u,0)=\min\{\ps_{-1}(u), \ps_{m+1}(u)\}$.	
	Note that $\sum_{i=0}^{m+1}u_i=n$. 
	Therefore, by Corollary \ref{cor:pivot_path_len}, $$\ps_{-1}(u)+\ps_{m+1}(u)=\sum_{i=0}^{m}(m+1-i)u_i + \sum_{i=1}^{m+1}iu_i=n(m+1).$$
	It now suffices to show that $|\ps_{-1}(u)-\ps_{m+1}(u)|\leq 1$, since this implies that $\ecc_{\Ynm}(0)\geq d(u,0)=\lfloor\frac{n(m+1)}{2} \rfloor$.
	
	Clearly, if either $2 \divides n$ or $2\divides (m-n)$, then $\ps_{-1}(u)=\ps_{m+1}(u)$. 
	If $2 \ndivides n$ and $2\divides m$, then $|\ps_{-1}(u)-\ps_{m+1}(u)|=1$.
	This implies that $|\ps_{-1}(u)-\ps_{m+1}(u)|\leq 1$, as required.
\end{proof}

\begin{lemma}\label{lem:no_pivots_ynm}
	Let $v\in\Ynm$ such that $v$ has no inner pivots. 
	Then $d(v,0)\leq \lfloor\frac{n(m+1)}{2}\rfloor$.
\end{lemma}
\begin{proof}
	Note that $v\neq 0$ and $\sum_{i=0}^{m+1}v_i = n$, since $v$ has no inner pivots.
	Therefore, by Corollary \ref{cor:pivot_path_len} and Observation \ref{obs:closer_pivot_is_better},
    \begin{align*}
	d(v,0) & = \min\{\ps_{-1}(v),\ps_{m+1}(v)\} \\ 
    & \leq \lfloor\frac{\ps_{-1}(v) + \ps_{m+1}(v)}{2}\rfloor = \lfloor\frac{\sum_{i=0}^{m+1}iv_i + \sum_{i=0}^{m+1}(m+1-i)v_i}{2}\rfloor \\
    & = \lfloor \frac{n(m+1)}{2} \rfloor.
    \end{align*}
\end{proof}


\section{Computation of the Eccentricity}\label{sec:n_leq_m_Ynm}
\begin{lemma}[$n=1$]\label{lem:corner_case_neqo}
    Let $m\geq 0$. Then 
    $$\ecc_{\Ynm[1][m]}(0)= \binom{\lceil \frac{m}{2}\rceil + 1}{2} + \binom{\lfloor \frac{m}{2}\rfloor + 1}{2}.$$
\end{lemma}
\begin{proof}
    The case $m=0$ is trivial. Assume that $m>0$. 
    Note that since $n=1$, $\piv(v)=\{-1,\ldots,m+1\}$ for every $v\in\Ynm$.
    Therefore, by Corollary \ref{cor:pivot_path_len}, Observation \ref{obs:closer_pivot_is_better} and Fact \ref{fac:split_sum_center}, $d(v,0)\leq \ps_{\lceil\frac{m}{2}\rceil}(v)\leq\sum_{i=1}^{\lceil \frac{m}{2}\rceil}i + \sum_{i=1}^{\lfloor \frac{m}{2}\rfloor}i = \binom{\lceil \frac{m}{2}\rceil + 1}{2} + \binom{\lfloor \frac{m}{2}\rfloor + 1}{2}$. This upper bound is obtained by the vertex $v\in\Ynm$ with $v_i=1$ for every $1\leq i\leq m$.
\end{proof}

\begin{lemma}\label{lem:eccentricity_in_I0}
    If $0\leq m \leq n$, then $\ecc_{\Ynm}(0) = \lfloor\frac{n(m+1)}{2} \rfloor$.
\end{lemma}
\begin{proof}
    The case $m\leq 1$ is true by Observation \ref{obs:small_m_is_trivial}. Assume that $m\geq 2$. 
    By Lemma \ref{lem:ecc_lower_bound_ynm}, it is sufficient to show that $\ecc_{\Ynm}(0) \leq \lfloor\frac{n(m+1)}{2} \rfloor$.
    Let $v\in\Ynm$. 
    By Lemma \ref{lem:no_pivots_ynm}, we can assume that $v$ has some inner pivot $p\in\piv(v)$.
    By Corollary \ref{cor:pivot_path_len} and Fact \ref{fac:split_sum_center} (with $q=m$), $d(v,0) \leq \ps_p(v) \leq \binom{p+1}{2}+\binom{m-p+1}{2} \leq \binom{m+1}{2}$.
    Since $m \le n$, $\binom{m+1}{2} \le \frac{n(m+1)}{2}$ and therefore $d(v,0)\leq\lfloor\frac{n(m+1)}{2}\rfloor$.
\end{proof}

Throughout the rest of this subsection, we assume that $2\leq n\leq m$.

\pagebreak
\begin{definition}\label{def:ynm_pivot_notations}
Let $v\in \Ynm$.
\begin{enumerate}
	\item $I_c(v) = [p_l(v)+1,p_r(v)]$. If $v_{p_l+1}=v_{p_r}=0$ (and $p_l+1=p_r$), then we say that $I_c(v)$ is \textbf{trivial}. We abbreviate $I_c$ when $v$ is evident.
	\item $h(v)=\min\{|p-\frac{m}{2}|:p\in \piv(v)\}$.
	\item $\hnm = 
	\begin{cases}
	\frac{n}{2} & \mbox{if } 2\divides (m-n) \\
	\frac{n+1}{2} & \mbox{if } 2\ndivides (m-n)
	\end{cases}.$
\end{enumerate}

\end{definition}

The outline of the rest of this subsection is as follows.
We first construct two candidates for a vertex eccentric to $0$ in Definitions \ref{def:uz} and \ref{def:uo} and compute their distances from $0$ in Lemma \ref{lem:dist_of_uz}.
We then prove that the maximum of the two distances is an upper bound on the distance of an arbitrary vertex $v$ from $0$.
We split this proof into two cases.

\begin{enumerate}
	\item $h(v)<\hnm$ (Lemma \ref{lem:close_to_middle_pivot}).
	\item $h(v)\geq\hnm$ (Lemma \ref{lem:middle_interval_is_n}).
\end{enumerate}

\begin{definition}[\sc $\uz$]\label{def:uz}
	Let $u\in\Ynm$ such that $u_i=1$ for all $1\leq i\leq m$ and $u_0\equiv -\lfloor\frac{m-n}{2} \rfloor \bmod n$. 
	Denote this vertex by $\uz$ and denote $\dz=d(\uz,0)$.
\end{definition}
For example, $\uz[3][5]=(2,1,1,1,1,1,2)$ and $\uz[3][6]=(2,1,1,1,1,1,1,1)$.

\begin{definition}[\sc $\uo$]\label{def:uo}
	Assume that $2\ndivides (m-n)$.
	Let $u\in\Ynm$ such that $u_{i}=0$ for $i=\lceil\frac{m+1}{2}\rceil$, $u_i=1$ for all other $1\leq i\leq m$ and $u_0\equiv -\lfloor\frac{m-n}{2} \rfloor \bmod n$. 
	Denote this vertex by $\uo$ and denote $\doo=d(\uo,0)$.
\end{definition}
For example, $\uo[2][5]=(1,1,1,0,1,1,1)$ and $\uo[3][6]=(2,1,1,1,0,1,1,2)$.

\pagebreak

\begin{observation}\label{obs:h_of_candidates}\ 
\begin{enumerate}
\item If $2\divides (m-n)$, then $h(\uz)=\frac{n}{2}=\hnm$.
\item If $2\ndivides (m-n)$, then $h(\uz)=\frac{n-1}{2}=\hnm-1$ and $h(\uo)=\frac{n+1}{2}=\hnm$.
\end{enumerate}
\end{observation}

\begin{lemma}\label{lem:dist_of_uz}\ 
	\begin{enumerate}
		\item $\dz = \binom{\lfloor\frac{m+n}{2}\rfloor+1}{2} + \binom{\lceil\frac{m-n}{2}\rceil+1}{2}$.
		\item If $2\ndivides (m-n)$, then $\doo= \dz +n - \lceil\frac{m+1}{2}\rceil$.
	\end{enumerate}
	
\end{lemma}
\begin{proof}
    Note that by definition, $p_l(\uz) = p_l(\uo) = \lfloor\frac{m-n}{2} \rfloor$ and $p_r(\uz)=\lfloor \frac{m+n}{2}\rfloor$.
	Regarding $\dz$: if $2\divides (m-n)$, then $|p_l(\uz)-\frac{m}{2}|=|p_r(\uz)-\frac{m}{2}|=\frac{n}{2}$.
	It follows, by Corollary \ref{cor:pivot_path_len} and Observation \ref{obs:closer_pivot_is_better}, that $\dz=d(\uz,0)=\ps_{p_l}(\uz)=\ps_{p_r}(\uz)=\binom{\frac{m+n}{2}+1}{2} + \binom{\frac{m-n}{2}+1}{2}$.
	If $2\ndivides (m-n)$, then $|p_l(\uz)-\frac{m}{2}|=|p_r(\uz)-\frac{m}{2}|+1=\frac{n+1}{2}$.
	Therefore $\dz= d(\uz,0)=\ps_{p_r}(\uz)=\binom{\lfloor\frac{m+n}{2}\rfloor+1}{2} + \binom{\lceil\frac{m-n}{2}\rceil+1}{2}$.
	
	Regarding $\doo$: if $2\divides n$ and $2\ndivides m$, then $|p_l(\uo)-\frac{m}{2}|=|p_r(\uo)-\frac{m}{2}|=\frac{n+1}{2}$ and computation gives $\doo= d(\uo,0)=\ps_{p_l}(\uo)=\ps_{p_r}(\uo)=  \dz +n - \frac{m+1}{2}$.
	If $2\ndivides n$ and $2\divides m$, then again $|p_l(\uo)-\frac{m}{2}|=|p_r(\uo)-\frac{m}{2}|=\frac{n+1}{2}$, but (since the entry at $\lceil\frac{m+1}{2}\rceil$ of $\uo$ is equal to zero) $\ps_{p_r}(\uo)<\ps_{p_l}(\uo)$ and $\doo=d(\uo,0)=\ps_{p_r}(\uo)=\dz +n - \lceil\frac{m+1}{2}\rceil$.
\end{proof}

\begin{definition}[\sc $\dnm$]\label{def:dnm}
	Denote 
	$$
	\dnm = 
	\begin{cases}
	\dz & \mbox{if } 2\divides (m-n); \\
	\max\{\dz, \doo\} & \mbox{if } 2\ndivides (m-n).
	\end{cases}
	$$
\end{definition}

\begin{observation}
    $\ecc_{\Ynm}(0)\geq \dnm$.
\end{observation}
It now suffices to show that $d(v,0)\leq \dnm$ for each $v\in\Ynm$.
As noted above we distinguish two cases.

\pagebreak
\begin{lemma}\label{lem:close_to_middle_pivot}
	Let $v\in\Ynm$ such that $h(v)<\hnm$. Then $d(v,0)\leq \dnm$.
\end{lemma}
\begin{proof}
    We show that actually $d(v,0)\leq \dz$.
	Let $p'$ be an inner pivot of $v$ for which $|p'-\frac{m}{2}|=h(v)$.
	By Observation \ref{obs:h_of_candidates}, if $2\divides(m-n)$, then $h(\uz)=\frac{n}{2}=\hnm$, and if $2\ndivides(m-n)$, then $h(\uz)=\hnm-1$. 
	By assumption, $h(v)\leq \hnm-1$ and therefore $h(v)\leq h(\uz)$.
    By Corollary \ref{cor:pivot_path_len} and Fact \ref{fac:split_sum_center}, $d(v,0)\leq \sum_{i=1}^{p'}i+\sum_{i=1}^{m-p'}i\leq \sum_{i=1}^{p}i+\sum_{i=1}^{m-p}i=\dz$
\end{proof}

The following lemma generalizes Lemma \ref{lem:no_pivots_ynm}.

\begin{lemma}\label{lem:v_hat_dist_from_zero_center_ynm}
	Let $v\in\Ynm$. 
	Then $d(v,0)\leq \sum_{i=1}^{p_l}i + \sum_{i=1}^{m-p_r}i + \lfloor\frac{n(m+1)}{2}\rfloor.$
\end{lemma}
\begin{proof}
    Recall the definition of $I_c=I_c(v)$ (Definition \ref{def:ynm_pivot_notations}).
	The sum $\sum_{i=1}^{p_l}i + \sum_{i=1}^{m-p_r}i$ is an upper bound on the number of steps needed to set to $0$ every $v_i$ with $i\in[1,m]\setminus I_c$, without shifting units into $I_c$; note that necessarily, following these steps, a bucket of $v$ not in $I_c$, is also equal to $0$.
	We can therefore assume that $v_i=0$ for every $i\notin I_c$ and prove that $d(v,0)\leq \lfloor\frac{n(m+1)}{2}\rfloor$.
	Note that $\sum_{i\in I_c}v_i$ is either $0$ or $n$.
    Therefore, as in the proof of Lemma \ref{lem:no_pivots_ynm}, $d(v,0)\leq\min\{\ps_{p_l}(v),\ps_{p_r}(v)\}\leq \lfloor\frac{\ps_{p_l}(v) + \ps_{p_r}(v)}{2}\rfloor \leq \lfloor\frac{n(m+1)}{2}\rfloor.$
\end{proof}


\begin{lemma}\label{lem:middle_interval_is_n}
	Let $v\in\Ynm$ such that $h(v)\geq\hnm$. 
	Then $d(v,0) \leq \dnm$.
\end{lemma}
\begin{proof}
	Assume that $2|(m-n)$.
	By Lemma \ref{lem:v_hat_dist_from_zero_center_ynm}, $d({v},0)\leq \sum_{i=1}^{p_l}i + \sum_{i=1}^{m-p_r}i + \frac{n(m+1)}{2}$.
    Since $h(v)\geq \hnm =\frac{n}{2}$,	$p_l\leq\frac{m-n}{2}$ and $p_r\geq\frac{m+n}{2}$.
	Therefore
	$ d(v,0)\leq \sum_{i=1}^{\frac{m-n}{2}}i + \sum_{i=1}^{m-\frac{m-n}{2}}i +\frac{n(m+1)}{2}=\dz$. 
	
	Assume that $2\ndivides(m-n)$.
	By Lemma \ref{lem:v_hat_dist_from_zero_center_ynm}, $d({v},0)\leq \sum_{i=1}^{p_l}i + \sum_{i=1}^{m-p_r}i + \lfloor\frac{n(m+1)}{2}\rfloor.$
	Since $h(v)\geq \hnm =\frac{n+1}{2}$,	$p_l\leq \frac{m-n-1}{2}=\lfloor\frac{m-n}{2}\rfloor$ and $p_r\geq \frac{m+n+1}{2}=\lceil\frac{m+n}{2}\rceil$.
	Therefore
	$ d(v,0)\leq \sum_{i=1}^{\lfloor\frac{m-n}{2}\rfloor}i + \sum_{i=1}^{m-\lceil\frac{m+n}{2}\rceil}i +\lfloor\frac{n(m+1)}{2}\rfloor=\doo\leq \dnm$. 
\end{proof}

This concludes the proof of Theorem \ref{thm:ynm_ecc}.

\chapter{The Diameter of $\Ynm$}\label{chp:diameter_of_Ynm}
In this chapter, we study the diameter of Yoke graphs.
Since Yoke graphs generalize already studied graphs, we first build upon a method used in order to calculate the diameter of one of these graphs.
Specifically, in \cite{elizalde}, similarities between arc permutation graphs and the Hasse diagram of the dominance order on $\Pnm[n][n-2]$ were used in order to calculate the diameter of arc permutation graphs.
In Section \ref{sec:diameter_dominance}, we generalize this approach in order to calculate the diameter of $\Ynm$ in the case $m\leq n$.

This approach fails for $n<m$. In Sections \ref{sec:dYokeGraphs}, \ref{sec:diam_ecc} and \ref{sec:ecc_zero_Znm}, we present a new approach that works for all instances of Yoke graphs.
In the core of this approach lies the idea to convert the problem of computation of diameter to that of computation of eccentricity.
Specifically, in Section \ref{sec:dYokeGraphs}, we introduce dYoke graphs $\Znm$. 
In Section \ref{sec:diam_ecc}, we show that the diameter of $\Ynm$ is equal to the eccentricity of $0$ in $\Znm$.
In Section \ref{sec:ecc_zero_Znm}, we compute the eccentricity of $0$ in dYoke graphs and show that $\ecc_{\Znm}(0)=\ecc_{\Ynm}(0)$ and therefore $\diam(\Ynm)=\ecc_{\Ynm}(0)$ with an explicit formula for all values of $n$ and $m$.

\section{Proof for $m\leq n$ via the Dominance Order}\label{sec:diameter_dominance}
Recall the dominance order $\trianglelefteq$ in Definition \ref{def:dom_order} and recall that we denote the set $\{0,\dots, n-1\}$ by $P_n$. 

\begin{observation}\label{obs:domPnm_is_sublattice}
	$\domPnm$ is a sublattice of $(\mathbb{Z}^{m+1},\tri)$.
\end{observation}
\begin{proof}
	Let $\chi$ be the isomorphism in Observation \ref{obs:domprodiso} and let \linebreak $L=\chi(\Pnm)\subseteq (\mathbb{Z}^{m+1},\leq)$.
	Note that $x=(x_0,\dots,x_m) \in L$ if and only if $0\leq x_0\leq n-1$ and $x_i - x_{i-1}\in \{0, 1\}$ for every $1\leq i\leq m$.
	$L$ is clearly closed under the meet and join operations of $(\mathbb{Z}^{m+1},\leq)$.
	Therefore, $L$ is a sublattice of $(\mathbb{Z}^{m+1},\leq)$.
	This proves that $\domPnm$ is a sublattice of $(\mathbb{Z}^{m+1},\tri)$.
\end{proof}

\begin{definition}
	Denote $\mathscr{H}\domPnm$ (the Hasse diagram of \linebreak $\domPnm$) by $\Hnm$.
	Denote the elements $(0,\dots,0)$ and $(n-1,1,\dots, 1)$ in $\Pnm$ by $\hat{0}$ and $\hat{1}$, respectively.
\end{definition}

Note that $\hat{0}$ and $\hat{1}$ are the minimum and maximum elements of \linebreak \mbox{$\domPnm$}, respectively.

Recall that vertices in $\Ynm$ are determined by their first $m+1$ entries.
Accordingly (and in accordance with Convention \ref{conv:cosets_are_integers}), we identify vertices in $\Ynm$ with vertices in $\Hnm$ by ignoring the last entry, namely ``forgetting" the right bucket of every vertex in $\Ynm$.

If $m\leq 1$, then the diameter of $\Ynm$ is trivial by Observation \ref{obs:small_m_is_trivial}.
Unless explicitly stated otherwise, we assume that $2\leq m\leq n$ (in $\Ynm$, $\Hnm$ and $\Pnm$) throughout this section.

%

\begin{observation}\label{obs:cover_in_domPnm}
Let $x$ and $y$ be two elements in $\Pnm$.
Then $x \lessdot y$ in $\domPnm$ if and only if $x \lessdot y$ in $(\mathbb{Z}^{m+1},\tri)$.
\end{observation}

By Observations \ref{obs:domcover} and \ref{obs:cover_in_domPnm}, the covering relation of $(\Pnm,\trianglelefteq)$ is precisely the adjacency relation of $\Ynm$ except for $2^{m-1}$ edges between every pair of vertices of the form $(n-1,1,v_2,\dots,v_{m})$ and $(0,0,v_2,\dots,v_{m})$ in $\Ynm$.

\begin{observation}\label{obs:Ynm_as_dominance}
$\Ynm$ is isomorphic to the graph obtained by taking $\Hnm$ and adding to it $2^{m-1}$ edges between every pair of elements of the form $(n-1,1,v_2,\dots,v_{m})$ and $(0,0,v_2,\dots,v_{m})$.
\end{observation}

Recall the rank function $\rank_\tri$ on $(\mathbb{Z}^n,\tri)$ in Observation \ref{obs:domprop}.
By Observation \ref{obs:cover_in_domPnm}, it induces a rank function on $\domPnm$: for every $v=(v_0,\dots,v_m)\in\Pnm$
$$\rank_\tri(v) = \sum_{i=0}^{m}\sum_{j=0}^{i}v_j.$$

\begin{observation}\label{obs:dist_Hnm_eq_domZm}
	If $u$ and $v$ are two elements in $\Pnm$, then
	$$ d_{\Hnm}(u,v) = d_{\mathscr{H}(\mathbb{Z}^{m+1},\tri)}(u,v).$$
\end{observation}
\begin{proof}
	The statement in this observation is implied by Theorem \ref{thm:dist_modular} and Observation \ref{obs:domPnm_is_sublattice}.
	Specifically: 
	$$ d_{\Hnm}(u,v) = d_{\mathscr{H}(\mathbb{Z}^{m+1},\tri)}(u,v) = \rank_\tri(u\join v)-\rank_\tri(u\meet v).$$
\end{proof}

\begin{lemma}\label{lemma:incomparable_in_pnm}
	If $u$ and $v$ are a pair of incomparable elements in \linebreak $\domPnm$, then 
	$$d_{\Hnm}(u,v)\leq 1+\binom{m}{2}.$$
\end{lemma}
\begin{proof}
    Let $\chi$ be the isomorphism in Observation \ref{obs:domprodiso} and let $x=(x_0,\dots,x_m)=\chi(u)$ and $y=(y_0,\dots,y_m)=\chi(v)$.
    Note that $x$ and $y$ are incomparable in $(\mathbb{Z}^{m+1},\leq)$, since $u$ and $v$ are incomparable in $\domPnm$.
    By Observation \ref{obs:dist_Hnm_eq_domZm}, $d_{\Hnm}(u,v)=d_{\mathscr{H}(\mathbb{Z}^{m+1},\leq)}(x,y)$.
    We prove that $d_{\mathscr{H}(\mathbb{Z}^{m+1},\leq)}(x,y)\leq 1+\binom{m}{2}$.
    
    Note that $z=(z_0,\dots,z_m) \in \chi(\Pnm)$ if and only if $0\leq z_0\leq n-1$ and $z_i - z_{i-1}\in \{0, 1\}$ for every $1\leq i\leq m$.
    Therefore  $(x_{i+1}-y_{i+1})-(x_{i}-y_{i})=(x_{i+1}-x_{i})-(y_{i+1}-y_{i})\in\{-1,0,1\}$ for every $1\leq i\leq m$.
    Since $x$ and $y$ are incomparable, there exist $0\leq j_1, j_2 \leq m$ such that $x_{j_1}-y_{j_1}<0$ and $x_{j_2}-y_{j_2}>0$. 
    This implies that there exists some $1\leq k\leq m-1$ such that $|x_k-y_k|=0$.
    Finally $d_{\mathscr{H}(\mathbb{Z}^{m+1},\leq)}(x,y)=\sum_{i=0}^m|x_i-y_i|\leq k + \dots + 1 + 0 + 1 + \dots + m-k=\binom{k+1}{2}+\binom{m-k+1}{2}\leq 1+\binom{m}{2}$. The maximum is obtained for $k=1$ (or $k=m-1$).
\end{proof}

\begin{lemma}\label{lem:modulo_along_edges}
If $u, v$ are two elements in $\Pnm$ such that $|\rank_\tri(v)-\rank_\tri(u)|\leq \lfloor\frac{n(m+1)}{2}\rfloor$, then $d_{\Ynm}(u,v)\geq |\rank_\tri(v)-\rank_\tri(u)|$.
\end{lemma}
\begin{proof}
If an edge $e=(x,y)$ in $\Ynm$ is also an edge in $\Hnm$, then, by Observation \ref{obs:cover_in_domPnm}, $|\rank_\tri(x)-\rank_\tri(y))|=1$. 
Otherwise, by \mbox{Observation \ref{obs:Ynm_as_dominance}}, $e$ is one of the additional $2^{m-1}$ edges in $\Ynm$ for which $|\rank_\tri(x)-\rank_\tri(y)|$ is constant and is equal to $n(m+1)-1$. 
Therefore, $\rank_\tri(x)$ changes by $\pm 1$ modulo $n(m+1)$ along every edge in $\Ynm$.
This implies the claim of this lemma.
\end{proof}

\begin{corollary}\label{cor:saturated_chains_are_deodesics}
    Every saturated chain of length smaller or equal to $\lfloor\frac{n(m+1)}{2}\rfloor$ in $(P_n\times P_2^m,\trianglelefteq)$ is a geodesic in $\Ynm$.
\end{corollary}

Saturated chains of length $\lfloor\frac{n(m+1)}{2}\rfloor$ exist in $\Ynm$, since $\rank_\tri(\hat{1}) \geq \lfloor\frac{n(m+1)}{2}\rfloor$ for any $n$ and $m$.

\begin{corollary}\label{cor:Ynm_diam_lower_by_dominance}
    $diam(\Ynm)\geq \lfloor\frac{n(m+1)}{2}\rfloor$.
\end{corollary}

In the rest of this section, we show that $\lfloor\frac{n(m+1)}{2}\rfloor$ is also an upper bound on the diameter of $\Ynm$, for $n\geq m$.
\begin{definition}
Let $u_0=(n-1,1,0,\dots,0)$ and $u_1=(0,0,1,\dots,1)$ in $\Pnm$.
\begin{enumerate}
\item Denote the interval $[\hat{0},u_0]\subseteq \domPnm$ by $I_0$. Denote by $\Ynm^0$ the subgraph induced by $I_0$ (in $\Ynm$).
\item Denote the interval $[u_1,\hat{1}]\subseteq \domPnm$ by $I_1$. Denote by $\Ynm^1$ the subgraph induced by $I_1$ (in $\Ynm$).
\end{enumerate}
\end{definition}

\begin{observation}\label{obs:intervals_by_sum}
    The elements in $I_0$ and $I_1$ can be characterized as follows.
    \begin{enumerate}
        \item $I_0=\{v\in\Pnm: \sum_{i=0}^{m}v_i \leq n\}$.
        \item $I_1=\{v\in\Pnm: \sum_{i=0}^{m}v_i \geq m-1\}$. 
    \end{enumerate}
    Therefore, the following is true (since $m\leq n$).
    \begin{enumerate}
        \item $I_0\cup I_1=\Pnm$.
        \item $I_0\cap I_1=\{v\in\Ynm:m-1\leq \sum_{i=0}^{m}v_i \leq n\} = [u_1,u_0]$.
    \end{enumerate}
\end{observation}

\begin{definition}\label{def:two_automorphisms}
    Let $\varphi$ and $\tau$ be the following two maps from $\Ynm$ to $\Ynm$:
    $$\begin{array}{c c c l}
    & \varphi(v_0,v_1,\dots,v_m) & = & (v_0+1,v_1,\dots,v_m) \\
    & \tau(v_0,v_1,\dots,v_m)    & = & (-v_0,1-v_1,\dots,1-v_{m}) \\
    \end{array}$$
\end{definition}
It is straightforward to verify that $\varphi$ and $\tau$ are automorphisms of $\Ynm$.

\begin{lemma}\label{lem:interval_iso}
     $\Ynm^0$ and $\Ynm^1$ are isomorphic as graphs.
\end{lemma}
\begin{proof}
Let $\tau$ and $\varphi$ be the automorphisms of $\Ynm$ in Definition \ref{def:two_automorphisms} and let $f=\tau\varphi$. 
Note that $f$ is involutive.
We show that both $f(\Ynm^0) \subseteq \Ynm^1$ and $f(\Ynm^1) \subseteq \Ynm^0$. This implies that $f$ is a graph isomorphism between $\Ynm^0$ and $\Ynm^1$.

Let $x\in\Ynm$ and let $y=f(x)$.
Assume that $x_0=t$ and $\sum_{i=1}^m x_i = k$.
Then $y_0=n-t-1$ and $\sum_{i=1}^m y_i = m-k$.
Therefore $\sum_{i=0}^m y_i = n+m-1-(t+k)$.

By Observation \ref{obs:intervals_by_sum}, if $x\in\Ynm^0$, then $t+k\leq n$.
Therefore, $\sum_{i=0}^m y_i = n+m-1-(t+k) \geq m-1$ and $y\in \Ynm^1$.
Similarly, if $x\in\Ynm^1$, then $t+k\geq m-1$.
Therefore, $\sum_{i=0}^m y_i = n+m-1-(t+k) \leq n$ and $y\in \Ynm^0$.
\end{proof}
\pagebreak
\begin{lemma}\label{lem:I0_diam}
    $\diam(\Ynm^0) = \diam(\Ynm^1) = \big\lfloor\frac{n(m+1)}{2}\big\rfloor.$
\end{lemma}
\begin{proof}
    By Lemma \ref{lem:interval_iso}, it is sufficient to prove that $\diam(\Ynm^0) = \lfloor\frac{n(m+1)}{2}\rfloor.$
    Let $u$ and $v$ be a pair of vertices in $\Ynm^0$.
    Since $\rank_\tri(u_0)=n(m+1)-1\geq \big\lfloor\frac{n(m+1)}{2}\big\rfloor$, $I_0$ contains saturated chains of length $\big\lfloor\frac{n(m+1)}{2}\big\rfloor$.
    Therefore, by Corollary \ref{cor:saturated_chains_are_deodesics}, it is sufficient to show that $d_{\Ynm^0}(u,v)\leq \lfloor\frac{n(m+1)}{2}\rfloor$.
    By Lemma \ref{lemma:incomparable_in_pnm}, we can assume that $u$ and $v$ are comparable, since $n\geq m$. 
    Therefore, there exists a saturated chain $A$ in $I_0$ between $0$ and $u_0$ containing both $u$ and $v$.
    $\len(A)=\rank_\tri(u_0)=n(m+1)-1$.
    There is an edge in $\Ynm^0$ connecting $0$ and $u_0$. 
    It completes $A$ to a cycle $C$ of length $n(m+1)$. 
    This proves that $d_{\Ynm^0}(u,v)\leq d_C(u,v)\leq \lfloor\frac{n(m+1)}{2}\rfloor$.
\end{proof}

\begin{theorem}\label{thm:diam_by_dominance}
    If $0\leq m\leq n$, then $$\diam(\Ynm) = \big\lfloor\frac{n(m+1)}{2}\big\rfloor.$$
\end{theorem}
\begin{proof}
    If $m\leq 1$, then this theorem is trivial by Observation \ref{obs:small_m_is_trivial}.
    Assume that $2\leq m\leq n$.
    
	By Corollary \ref{cor:Ynm_diam_lower_by_dominance}, it is sufficient to prove that $diam(\Ynm)\leq \lfloor\frac{n(m+1)}{2}\rfloor$.
	Let $u,v\in\Ynm$.
	By Observation \ref{obs:intervals_by_sum}, $I_0\cup I_1=\Pnm$.
	If both $u$ and $v$ are in the same interval, say $I_0$, then by Lemma \ref{lem:I0_diam}, $d_{\Ynm}(u,v)\leq d_{\Ynm^0}(u,v)\leq \lfloor\frac{n(m+1)}{2}\rfloor$.
	Otherwise, assume without loss of generality that $u\in I_1$ and $v\in I_0\setminus I_1$. 
    By Observation \ref{obs:intervals_by_sum}, $I_0\setminus I_1 = \{v\in\Ynm:\sum_{i=0}^{m}v_i\leq m-2\}$.
	It is sufficient to show that $v$ can be mapped into $I_0\cap I_1$ via an automorphism of $\Ynm$.

	Let $\varphi$ be the automorphism of $\Ynm$ in Definition \ref{def:two_automorphisms}.
    Note that if $v_0\neq 0$, then $\varphi^{-1}(v)$ is still in $I_0\setminus I_1$. 
    Therefore, we can assume that $v_0=0$.
    Let $c=(m-1)-\sum_{i=0}^{m}v_i$ and let $v'=\varphi^c(v)$.
    Note that $1\leq c \leq m-1\leq n-1$, since $\sum_{i=0}^{m}v_i\leq m-2$.
    Therefore, $v'_0=c$ and $\sum_{i=0}^{m} v'_i=m-1$.
	This proves, by Observation \ref{obs:intervals_by_sum}, that $\varphi^c(v)\in I_0\cap I_1$ as required.
\end{proof}
\pagebreak
The above line of proof fails for $n<m$.
We therefore pursue a new, totally different approach, which works for all values of $n$ and $m$.

\section{The dYoke Graphs $\Znm$}\label{sec:dYokeGraphs}
As noted in the introduction of this thesis, at the heart of our approach to the calculation of the diameter of $\Ynm$ lies the idea of converting the diameter problem of one graph into an eccentricity problem of another graph.
To this end, we introduce a new family of graphs.

\begin{definition}\label{def:dyoke_graph}
    Let $n\geq 1$ and $m\geq 0$ be two integers. 
    Denote the subset $\{-1,0,1\}\subseteq \mathbb{Z}$ by $Q_3$.
    The \textbf{dYoke graph} $\Znm$ is a graph with vertices corresponding to all $u=(u_0\dots,u_{m+1})\in\mathbb{Z}_n\times Q_3^m\times \mathbb{Z}_n$ such that $\sum_{i=0}^{m+1}u_i\equiv 0(\bmod n).$ 
    Two vertices $u$ and $v$ are adjacent in $\Znm$ if there exists $0\leq i\leq m$ such that $u_j=v_j$ for every $j\notin\{i,i+1\}$ and one of the following two cases holds: either $u_i=v_i+1$ and $u_{i+1}=v_{i+1}-1$, or $u_i=v_i-1$ and $u_{i+1}=v_{i+1}+1$.
\end{definition}

As with Yoke graphs, we refer to the entries $u_0$ and $u_{m+1}$ of a vertex $u$ in $\Znm$ as \textbf{buckets}; and Convention \ref{conv:cosets_are_integers} applies to dYoke graphs as well, so that the sum $\sum_{i=0}^{m+1}u_i$ in Definition \ref{def:dyoke_graph} is an integer. 

Note that, as with Yoke graphs, a vertex $u$ in $\Znm$ is determined by its first (or last) $m+1$ entries, since $\sum_{i=0}^{m+1}u_i\equiv 0(\bmod n)$.
We denote the vertex $(0,\dots,0)\in\Znm$ by $0$.
In the first case of the adjacency relation, where $u_i=v_i+1$ and $u_{i+1}=v_{i+1}-1$ for some $0\leq i\leq m$, we say that $u$ is obtained from $v$ by \textbf{shifting} a unit from entry $i+1$ to the \textbf{left}, and write $u=\overleftarrow{s}_{i}(v)$.
In the second case, where $u_i=v_i-1$ and $u_{i+1}=v_{i+1}+1$, we say that $u$ is obtained from $v$ by \textbf{shifting} a unit from entry $i$ to the \textbf{right}, and write $u=\overrightarrow{s}_{i}(v)$.
When $v_0=0$ ($v_{m+1}=0$), we say that the left (right) bucket is \textbf{empty}.

For example, if $v=(3,0,-1,1,2)\in\Znm[5][3]$, then $\overleftarrow{s}_2(v)=(3,0,0,0,2)$ and $\overrightarrow{s}_1(v)=(3,-1,0,1,2)$.
If $i<j$, we say that $v_j$ is an \textbf{entry to the right} of $v_i$ and that $v_i$ is an \textbf{entry to the left} of $v_j$. 
For convenience, we write $s_i$ to indicate a unit shift between the entries indexed by $i$ and $i+1$ without specifying its direction.

Note, that if $m>0$, $1\leq k\leq m$ and $v_k=-1$, then a unit shift from entry $k$ of $v$ is not possible.
Similarly, if $v_k=1$, then a unit shift to entry $k$ of $v$ not possible. 
Nevertheless, for every $0\leq i\leq m$, $\overleftarrow{s}_i$ and $\overrightarrow{s}_i$ induce functions from $\Znm$ to $\Znm$, in which if a unit shift $s_i(v)$ is not possible for some $v\in\Znm$, then $v$ is a fixed point of $s_i$.
Denote the set $\{\overleftarrow{s}_i,\overrightarrow{s}_i:0\leq i\leq m\}$ by $\Fnm$.
If $m=0$, then $\overleftarrow{s}_0$ and $\overrightarrow{s}_0$ are inverses of each other and $\Fnm[n][0]$ generates the cyclic group $\mathbb{Z}_n$. 
If $m>0$, then none of the functions in $\Fnm$ are bijective and $\Fnm$ generates a semigroup.

A \textbf{word in the letters} $\Fnm$ is a sequence $f_d\cdots f_1$ where \mbox{$f_1,\dots,f_d\in\Fnm$}.
Let $P$ be a path $(v=v^0\sim\dots \sim v^d=u)$ in $\Znm$. 
Clearly, there is a unique word $f_d\cdots f_1$ such that $f_t(v^{t-1})=v^t$ for every $1\leq t\leq d$ or, equivalently,  \mbox{$f_t\cdots f_1(v)=v^t$} for every $1\leq t\leq d$.
We say that $f_d\cdots f_1$ is the \textbf{word corresponding to the path $P$ from $v$ to $u$} (or the path $P$ starting at $v$).

Note that a word $f_d\cdots f_1$ does not correspond to a path starting at $v$ if $f_{t}\cdots f_1(v) = f_{t+1}f_t\cdots f_1(v)$ for some $1\leq t\leq d-1$.
For example, the word $\overleftarrow{s_1}\overleftarrow{s_0}$ corresponds to the path $(v=(0, 1, 1, 1)\sim (1, 0, 1, 1)\sim (1, 1, 0, 1))$ in $\Znm[3][2]$ starting at $v$.
However, $\overleftarrow{s_0}\overleftarrow{s_1}$ does not correspond to a path starting at $v$, since $v$ is a fixed point of $\overleftarrow{s_1}$.

Let $f_d\cdots f_1$ be the word corresponding to a path from $v$ to $u$ and let $1\leq t\leq d-1$.
If $f_d\cdots f_{t}f_{t+1}\cdots f_1$ (changing the order of $f_t$ and $f_{t+1}$) is also a word corresponding to a path from $v$ to $u$, then we say that $f_t$ and $f_{t+1}$ \textbf{relatively commute} in (the path) $f_d\cdots f_1(v)$.
Note that the fact that $f_t$ and $f_{t+1}$ relatively commute in $f_d\cdots f_1(v)$, does not imply that $f_t$ and $f_{t+1}$ commute (as functions from the set of vertices of $\Znm$ to itself). 
For example, if $f_d\cdots f_1$ corresponds to a path from $v$ to $u$ such that $f_t=\overrightarrow{s}_i$ and $f_{t+1}=\overleftarrow{s}_{i+1}$ for some $1\leq t \leq d-1$ and $0\leq i\leq m-1$, then $f_t$ and $f_{t+1}$ relatively commute in $f_d\cdots f_1(v)$ but do not commute (as functions).
If, however, $f_t$ and $f_{t+1}$ commute, then they also relatively commute in $f_d\cdots f_1(v)$.
Clearly, two distinct $s_i$ and $s_j$ in $\Fnm$ commute if and only if $|i-j|>1$.

\begin{lemma}[\sc Shift Direction Lemma] \label{lem:shift_direction_lemma_Znm}
    Let $f_d\cdots f_1$ be the word corresponding to a geodesic in $\Znm$.
    For every $0\leq i\leq m$, all of the instances of $s_i$ in $f_d\cdots f_1$ shift in the same direction.
\end{lemma}

\begin{proof}
    Let $v$ and $u$ be two vertices in $\Znm$.
    If $m=0$, then a word corresponding to a geodesic from $v$ to $u$ is of the form $f^k$ where $f\in\Fnm[n][0]$ and $0\leq k\leq \lfloor\frac{n}{2}\rfloor$.
    Assume that $m>0$ and assume to the contrary that there exist words corresponding to geodesics from $v$ to $u$ not satisfying the lemma.
    Denote the set of such words by $\mathcal{W}$.
    For every word $f_d\cdots f_1$ in $\mathcal{W}$, there exist $0\leq i\leq m$ and $1\leq t_1, t_2\leq d$ such that $f_{t_1}=\overrightarrow{s}_{i}$ and $f_{t_2}=\overleftarrow{s}_{i}$.
    Let $w=f_d\cdots f_1$ be a word in $\mathcal{W}$ in which $\min\{|t_2-t_1|: f_{t_1}=\overrightarrow{s}_{i},f_{t_2}=\overleftarrow{s}_{i}\}$ is minimal. 
    Assume without loss of generality that $t_1<t_2$.
    
    If $t_2=t_1+1$, then the word obtained by deleting both $f_{t_1}$ and $f_{t_2}$ from $f_d\cdots f_1$ is a word corresponding to a shorter path from $v$ to $u$, contradicting the minimality of $d$.
    Therefore, we can assume that $t_2-t_1>1$.
    
    If $f_{t_1+1}=s_j$ for some $0\leq j\leq m$ such that $|i-j|>1$, then $f_{t_1}$ and $f_{t_1+1}$ commute as functions.
    If $f_{t_1+1}\in\{\overleftarrow{s}_{i-1}, \overleftarrow{s}_{i+1}\}$, then $f_{t_1}$ and $f_{t_1+1}$ relatively commute in $f_d\cdots f_1(v)$.
    In both cases, the word obtained by interchanging $f_{t_1}$ and $f_{t_1+1}$ in $w$ corresponds to a path from $v$ to $u$.
    This contradicts the minimality of $t_2-t_1$ in the choice of $w$.
    Therefore, $f_{t_1+1}\in\{\overrightarrow{s}_{i-1}, \overrightarrow{s}_{i+1}\}$.
%

    Let $v'=f_{t_1+1}\cdots f_1(v)$. 
    If $f_{t_1+1}=\overrightarrow{s}_{i-1}$, then $v'_i=1$ (otherwise  $f_{t_1}$ and $f_{t_1+1}$ relatively commute in $f_d\cdots f_1(v)$). 
    Since $f_{t_2}=\overleftarrow{s}_i$, there must exist some $t_1+1<k<t_2$ such that $f_k=\overleftarrow{s}_{i-1}$ or $f_k=\overrightarrow{s}_{i}$.
    Similarly, if $f_{t_1+1}=\overrightarrow{s}_{i+1}$, then $v'_{i+1}=-1$ and there must exist some $t_1+1<k<t_2$ such that $f_k=\overleftarrow{s}_{i+1}$ or $f_k=\overrightarrow{s}_{i}$.
    All of the cases contradict the minimality of $t_2-t_1$.
\end{proof}

\section{From Diameter to Eccentricity}\label{sec:diam_ecc}
In this section, we show (Theorem \ref{thm:ecc_eq_diam}) that the diameter of $\Ynm$ is equal to the eccentricity of $0$ in $\Znm$.
Every Yoke graph $\Ynm$ is naturally embedded in the dYoke graph $\Znm$ as an induced subgraph on the vertices with no negative entries. 
Note that for every two vertices $v,u\in\Ynm$, the difference $v-u=(v_0-u_0,v_1-u_1,\dots,v_{m+1}-u_{m+1})$ is in $\Znm$.

\begin{definition}
	For every $u\in\Ynm$ let $\varphi_u:\Ynm\rightarrow \Znm$ be the map (on vertices) defined by $\varphi_u(v)=v-u$. 
\end{definition}

\begin{observation}\label{obs:embedding_Ynm_in_Znm}
	For every $u\in\Ynm$, $\varphi_u$ is a graph isomorphism between $\Ynm$ and the subgraph induced by $\varphi_u(\Ynm)$ (in $\Znm$).
\end{observation}

The following lemma is essential for the proof of Lemma \ref{lem:dist_preserved_under_phi}.
We start with a given $z\in\Znm$ such that $z_i=0$ for some $1\leq i\leq m$, and a geodesic $P$ between $z$ and $0$.
We show that $P$ can be transformed into a geodesic $P'$ in which the $i$th entry is either non-negative or non-positive along the path. We can do this without changing the sets of values of other entries.

\begin{lemma}\label{lem:handle_zeros_in_geodesics_znm}
	Let $z\in\Znm$ such that $z_i=0$ for some $1\leq i\leq m$ and let $P=(z=x^0\sim x^1\sim\dots\sim x^d=0)$ be a geodesic between $z$ and $0$.
	Then a geodesic $P'=(z=y^0\sim y^1\sim\dots\sim y^d=0)$ exists such that:
	{ 
		\begin{enumerate}
			\item either $y^t_i\leq 0$ for every $0\leq t\leq d$ or $y^t_i\geq 0$ for every $0\leq t\leq d$ ($P'$ can be constructed either way);
			\item \label{asdf}$\{y^t_j:0\leq t\leq d\}=\{x^t_j:0\leq t\leq d\}$ for every $j\neq i$ where $0\leq j\leq m+1$.
	\end{enumerate} }
\end{lemma}
\begin{proof}
	We prove the existence of $P'$ such that $y^t_i\leq 0$ for every $0\leq t\leq d$.
	The proof for the case $y^t_i\geq 0$ follows by symmetric arguments.
	Let $Q=(w^0\sim w^1\sim\dots\sim w^l)$ be a path in $\Znm$. Define $O_i(Q)=|\{0\leq t\leq l:w^t_i=1\}|$.
	If $O_i(P)=0$, then set $P'=P$.
	Otherwise, assume that $O_i(P)>0$.
	We show that there exists a geodesic $Q=(z=w^0\sim w^1\sim\dots\sim w^d=0)$ with the following properties:
	{ 
		\renewcommand\labelenumi{(\theenumi)}
		\begin{enumerate}
			\item $O_i(Q)<O_i(P)$.
			\item $\{w^t_j:0\leq t\leq d\}=\{x^t_j:0\leq t\leq d\}$ for every $j\neq i$ where $0\leq j\leq m+1$.
		\end{enumerate} }
	Note that if $O_i(Q)>0$, then the process can be repeated implying the existence of $P'$, as required.
	
	Let $f_d\cdots f_1$ be the word corresponding to $P$.
	Note that by the Shift Direction Lemma \ref{lem:shift_direction_lemma_Znm}, either both $\overrightarrow{s}_{i-1}$ and $\overrightarrow{s}_i$ or both $\overleftarrow{s}_{i-1}$ and $\overleftarrow{s}_i$ appear in $f_d\cdots f_1$, since $x^0_i=0$, $x^d_i=0$ and $O_i(P)>0$. 
	Assume without loss of generality that both $\overrightarrow{s}_{i-1}$ and $\overrightarrow{s}_i$ appear in $f_d\cdots f_1$.
	
	Let $t_1$ be the minimal index such that $x^{t_1}_i=1$ and therefore $f_{t_1}=\overrightarrow{s}_{i-1}$.
	Let $t_2$ be the minimal index such that $t_1<t_2$ and $f_{t_2}=\overrightarrow{s}_{i}$ (such $t_2$ exists since $x^d_i=0$).
	If $t_2-t_1=1$, then $f_{t_1}$ and $f_{t_2}$ relatively commute in $f_d\cdots f_1(z)$, and the path $Q$ obtained by interchanging $f_{t_1}$ and $f_{t_2}$ satisfies properties (1) and (2) as required.
	In the rest of this proof, we assume that $t_2-t_1>1$.
	
	We can assume that $f_d\cdots f_1$ is in the form 
	$$(*)\;\; f_d\dots f_{t_2}w_2w_1f_{t_1}\dots f_1$$
	where $w_1$ and $w_2$ are words (at least one of which is nonempty) such that every letter in $w_1$ shifts a unit between entries to the right of $i$, and every letter in $w_2$ shifts a unit between entries to the left of $i$.
	Indeed, if there exist $t_1<t<t+1<t_2$ such that $f_{t}$ shifts a unit between entries to the left of $i$ and $f_{t+1}$ shifts a unit between entries to the right of $i$, then $f_t$ and $f_{t+1}$ commute and
	the word obtained by interchanging $f_t$ and $f_{t+1}$ corresponds to a path between $z$ and $0$ which satisfies property (2), and which does not change $O_i(P)$.
	By repeatedly interchanging such pairs, we obtain a word in the form $(*)$, since $f_{t}\notin \{\overrightarrow{s}_{i-1}, \overrightarrow{s}_{i}\}$ for every $t_1<t<t_2$.
	
	Note that $f_{t_1}$ commutes with every letter in $w_1$ and $f_{t_2}$ commutes with every letter in $w_2$. 
	Therefore, $f_d\dots w_2f_{t_1}f_{t_2}w_1\dots f_1$ corresponds to a path $Q$ between $z$ and $0$ satisfying properties (1) and (2), similarly to the case $t_2 - t_1 = 1$.
\end{proof}

\begin{lemma}\label{lem:dist_preserved_under_phi}
    Let $v,u\in\Ynm$. Then $$d_{\Ynm}(v,u)=d_{\Znm}(v-u,0).$$
\end{lemma}
\begin{proof}
	By Observation \ref{obs:embedding_Ynm_in_Znm}, $\Ynm$ is isomorphic to the subgraph induced by $\varphi_u(\Ynm)$ in $\Znm$, implying that $d_{\Ynm}(v,u)\geq d_{\Znm}(\varphi_u(v),\varphi_u(u))=d_{\Znm}(v-u,0)$.
	
	Conversely, let $P=(v-u=x^0\sim x^1\sim\dots\sim x^d=0)$ be a geodesic from $v-u$ to $0$ in $\Znm$.
	Note that if $x^0_i\leq 0$ for some $1\leq i\leq m$, then, by Lemma \ref{lem:handle_zeros_in_geodesics_znm}, there exists a geodesic $P'=(v-u=y^0\sim y^1\sim\dots\sim y^d=0)$ such that $y^t_i\leq 0$ for every $0\leq t\leq d$ (the lemma is applied to $(x^k\sim x^{k+1}\sim\dots\sim x^d=0)$, where $k$ is the first index such that $x^k_i=0$).
	Similarly, if $x^0_i\geq 0$ for some $1\leq i\leq m$, then there exists a geodesic $P'=(v-u=y^0\sim y^1\sim\dots\sim y^d=0)$ such that $y^t_i\geq 0$ for every $0\leq t\leq d$.
	
	Therefore, by its second property, Lemma \ref{lem:handle_zeros_in_geodesics_znm} can be applied iteratively to all $1\leq i\leq m$, to construct a geodesic
	$P'=(v-u=y^0\sim y^1\sim\dots\sim y^d=0)$ from the geodesic $P$ that satisfies the following conditions for every $1\leq i\leq m$:
	\begin{enumerate}
		\item If $u_i=1$ (implying that $(v-u)_i\leq 0$), then $y^t_i\leq 0$ for every $0\leq t\leq d$.
		\item If $u_i=0$ (implying that $(v-u)_i\geq 0$), then $y^t_i\geq 0$ for every $0\leq t\leq d$.
	\end{enumerate}
	Clearly, $(v=y^0+u\sim y^1+u\sim\dots\sim y^d+u=u)$ is a path between $v$ and $u$ in $\Ynm$. Therefore $d_{\Ynm}(v,u)\leq d_{\Znm}(v-u,0)$.
\end{proof}

\begin{theorem}\label{thm:ecc_eq_diam}
	$\diam(\Ynm) = \ecc_{\Znm}(0)$.
\end{theorem}
\begin{proof}
	By Lemma \ref{lem:dist_preserved_under_phi}, $d_{\Ynm}(v,u)=d_{\Znm}(v-u,0)\leq \ecc_{\Znm}(0)$ for each pair of vertices $v$ and $u$ in $\Ynm$.
    Therefore, $\diam(\Ynm) \leq \ecc_{\Znm}(0)$.
	On the other hand, for every $z\in\Znm$ there exist $v, u\in\Ynm$ such that $v-u=z$.
	By Lemma \ref{lem:dist_preserved_under_phi}, $d_{\Znm}(z,0) = d_{\Znm}(v-u,0) = d_{\Ynm}(v,u) \leq \diam(\Ynm)$, and therefore $\ecc_{\Znm}(0) \leq \diam(\Ynm)$. This proves equality.
\end{proof}

\section{The Eccentricity of $0$ in $\Znm$}\label{sec:ecc_zero_Znm}
In this section, we compute the eccentricity of $0$ in $\Znm$  (Theorem \ref{thm:ecc_znm}). 
It turns out to be equal to the eccentricity of $0$ in $\Ynm$ as it appears in Theorem \ref{thm:ynm_ecc}.
By Theorem \ref{thm:ecc_eq_diam}, it is also equal to the diameter of $\Ynm$. 

In Subsection \ref{subsec:pivot_paths_Znm}, we study pivot paths and some related concepts. 
Note that many of the definitions that appear in this section are similar to the ones in Chapter \ref{chp:ecc_zero_Ynm}, where these notions were applied to compute the eccentricity of $0$ in Yoke graphs.
Here, these definitions are applied to dYoke graphs and the proofs are slightly more complex due to the richer structure of dYoke graphs. In Subsection \ref{subsec:n_leq_m_Znm}, we compute the eccentricity of $0$ in $\Znm$.

\begin{observation}\label{obs:corner_case_meqz}
    If $m=0$, then $\Znm[n][0]$ is a cycle graph on $n$ vertices.
    Therefore, $\ecc_{\Znm[n][0]}(0)=\lfloor\frac{n}{2}\rfloor=\lfloor\frac{n(m+1)}{2}\rfloor$.
\end{observation}

In the rest of this section, unless explicitly stated otherwise, we assume that $m>0$.

\subsection{Pivots and Walls}\label{subsec:pivot_paths_Znm}

\begin{definition}[\sc Pivot]
	A \textbf{pivot} of a vertex $v\in\Znm$ is an integer $-1\leq p\leq m+1$ such that $\sum_{i=0}^{p}v_i$ is divisible by $n$. 
	We call $-1$ and $m+1$ the \textbf{outer pivots} of $v$; they are pivots of every $v\in\Znm$. Every other pivot, if it exists, is called an \textbf{inner pivot}.
	Denote the set of pivots of $v$ by $\piv(v)$.
\end{definition}

For example, in $\Znm[3][8]$, $\piv((0,1,-1,0,1,1,-1,-1,1,2))=\{-1,0,2,3,7,9\}$.
\begin{definition}[\sc Wall]
	Let $P$ be a path from $v$ to $0$ in $\Znm$ and let $f_d\cdots f_1$ be the word corresponding to $P$.
	An \textbf{inner wall} of $P$ is an integer $0\leq p\leq m$ such that $s_p$ does not appear in $f_d\cdots f_1$.
	$-1$ is a \textbf{left outer wall} of $P$ if $\overleftarrow{s}_0$ does not appear in $f_d\cdots f_1$. 
	Similarly, $m+1$ is a \textbf{right outer wall} of $P$ if $\overrightarrow{s}_m$ does not appear in $f_d\cdots f_1$.
	Finally, we say that $-1\leq p\leq m+1$ is a \textbf{wall} of $P$ if it is either an inner wall or an outer wall of $P$.
\end{definition}

Clearly, paths from $v$ to $0$ with a wall $p$ exist if and only if $p$ is a pivot of $v$.

\begin{definition}[\sc Pivot Path]
    Let $p$ be a pivot of a vertex $v$ in $\Znm$.
	We say that a path $P$ from $v$ to $0$ is a \textbf{$p$-pivot path} of $v$, if $P$ is shortest among the paths from $v$ to $0$ with a wall $p$.
	Denote the length of a $p$-pivot path of $v$ by $\ps_p(v)$.
    We say that $P$ is a \textbf{pivot path}, if it is a $p$-pivot path for some $p$. 
\end{definition}

For example, let $v=(1,0,-1,-1,1)$ in $\Znm[3][3]$. 
Then $\overleftarrow{s}_3\overrightarrow{s}_0\overrightarrow{s}_1$ is a word corresponding to a $2$-pivot path $P$ of $v$: $P=(v=(1,0,-1\,\|-1,1)\sim (1,-1,0\,\|-1,1)\sim (0,0,0\,\|-1,1)\sim (0,0,0\,\|\,0,0)=0)$ where we added the symbol $\|$ to visualize the wall $2$ of $P$.

\begin{lemma}\label{lem:geodesic_is_a_pivot_path}
	Every geodesic in $\Znm$ ending at $0$ is a pivot path (namely, has a wall).
\end{lemma}
\begin{proof}
	Let $w=f_d\cdots f_1$ be the word corresponding to a geodesic $P$ in $\Znm$ from $v$ to $0$.
	Assume to the contrary that $P$ has no walls.
    Therefore, both $\overleftarrow{s}_0$ and $\overrightarrow{s}_m$ appear in $w$ (since $P$ has no outer walls).
	By the Shift Direction Lemma \ref{lem:shift_direction_lemma_Znm}, and since $P$ has no walls, exactly one of $\{\overleftarrow{s}_i, \overrightarrow{s}_i\}$ appears in $w$ (at least once) for every $0\leq i\leq m$.
    Let $1\leq i\leq m$ be minimal such that $\overrightarrow{s}_i$ appears in $w$.
    Then both $\overleftarrow{s}_{i-1}$ and $\overrightarrow{s}_{i}$ appear in $w$. 
    This is a contradiction, since it implies that, throughout $P$, at least two units are shifted from the entry $v_i$ and no units are shifted to $v_i$.
\end{proof}

\begin{observation}\label{obs:closer_pivot_is_better_znm} By Lemma \ref{lem:geodesic_is_a_pivot_path},
    $$\ecc_{\Znm}(0)=\max_vd(v,0)=\max_v\min_p\ps_p(v).$$
\end{observation}

\begin{definition}
    Let $v\in\Znm$ and $p\in\piv(v)$.
    If (one, equivalently all) $p$-pivot paths of $v$ are also geodesics between $v$ and $0$, then we say that $p$ is a \textbf{geodesic pivot} of $v$.
\end{definition}

The Shift Direction Lemma \ref{lem:shift_direction_lemma_Znm} deals with geodesics, or equivalently (by Lemma \ref{lem:geodesic_is_a_pivot_path}) with $p$-pivot paths for geodesic pivots $p$.
It can be extended to arbitrary pivot paths, with essentially the same proof.

\begin{lemma}[\sc Pivot Shift Direction Lemma]\label{lem:pivot_shift_direction_lemma}
    Let $f_d\cdots f_1$ be the word corresponding to a pivot path in $\Znm$.
    For every $0\leq i\leq m$, all of the instances of $s_i$ in $f_d\cdots f_1$ shift in the same direction.
\end{lemma}

\begin{observation}\label{obs:trivial_path}
    Let $v\in\Znm$ and let $p\in\piv(v)$. Then
    $$\ps_p(v)\leq \sum_{i=1}^{p}i|v_i| + \sum_{i=p+1}^{m}(m+1-i)|v_i|.$$
    If, in addition, $v_i\geq 0$ for every $1\leq i\leq m$, then the above is an equality.
\end{observation}
\begin{proof}
There exists a path with a wall $p$ from $v$ to $0$ of length $\sum_{i=1}^{p}i|v_i| + \sum_{i=p+1}^{m}(m+1-i)|v_i|$ (not necessarily a $p$-pivot path). 
It is a path in which the entries $\{v_0,\dots,v_p\}$ and $\{v_{p+1}, \dots, v_{m+1}\}$ of $v$ are handled one by one, left to right and right to left, respectively; each $v_i$, in turn, can be set to $0$ by $i$ unit shifts from (if $v_i=-1$) or to (if $v_i=1$) the bucket on the same side of $p$ as $v_i$.

For example, if $v=(0,-1,1,0)$ in $\Znm[n][2]$ for some $n\geq 1$ and $p=2$, then the path from $v$ to $0$ that corresponds to the word $\overleftarrow{s}_0\overleftarrow{s}_1\overrightarrow{s}_0$ has length $1+2=3$, as above.
Note that it is not a $2$-pivot path, since $\overleftarrow{s}_1(v)=0$.

If, however, $v_i\geq 0$ for every $1\leq i\leq m$, then by the Pivot Shift Direction Lemma \ref{lem:pivot_shift_direction_lemma}, a path as above is, in fact, a $p$-pivot path.
\end{proof}

\begin{corollary}\label{cor:trivial_upper_bound}
    If $p$ is an inner pivot of $v\in\Znm$, then
    $$\ps_p(v)\leq \binom{p+1}{2}+\binom{m-p+1}{2}.$$
\end{corollary}

Comparing Corollary \ref{cor:pivot_path_len} and Observation \ref{obs:trivial_path} leads to the following corollary:

\begin{corollary}\label{cor:dist_equal_when_positive}
    Let $v\in \Ynm$, considered also as a vertex in $\Znm$.
    Then $$d_{\Znm}(v,0)=d_{\Ynm}(v,0).$$
\end{corollary}

\begin{definition}[\sc $P$-interval]
	Let $P$ be a pivot path of $v$ and let $\{p_1,\ldots,p_t\}$ with $-1<p_1<\ldots<p_t<m+1$ be the set of inner walls of $P$.
	Denote $p_0 = -1$, $p_{t+1} = m+1$ and $\piv(P) = \{p_0, p_1, \ldots, p_t, p_{t+1}\}$.
	Note that $\piv(P)\subseteq \piv(v)$.
	We call every $p\in\piv(P)$ a \textbf{pivot of $P$}.
	$\piv(P)$ induces $t+1$ intervals $[p_k + 1,p_{k+1}]$ for $-1\leq k\leq t$.
	Note that these intervals are a partition of $[0,m+1]$.
	We call every such interval a \textbf{$P$-interval}.
	For convenience, we say that an entry $v_i$ is in the interval $I$ if $i\in I$.
\end{definition}

\begin{lemma}\label{lem:same_direction_in_interval}
	Let $f_d\cdots f_1$ be the word corresponding to a pivot path $P$ of $v$ and let $I$ be a $P$-interval.
	Then for every $[i,i+1]\subseteq I$ and $[j,j+1]\subseteq I$, the instances of $s_i$ and $s_j$ in $f_d\cdots f_1$ shift in the same direction.
\end{lemma}
\begin{proof}
	By the Pivot Shift Direction Lemma \ref{lem:pivot_shift_direction_lemma} and the definition of a $P$-interval, exactly one of $\{\overleftarrow{s}_i, \overrightarrow{s}_i\}$ appears in $f_d\cdots f_1$, for every $[i,i+1]\subseteq I$.
	Assume to the contrary, without loss of generality, that both $\overleftarrow{s}_i$ and $\overrightarrow{s}_j$ appear in $f_d\cdots f_1$ for some $i<j$ in $I$.
	Let $k\in I$ such that $i<k$, $s_k$ appears shifting right in $f_d\cdots f_1$ and $k$ is minimal with these properties.
	Then both $\overleftarrow{s}_{k-1}$ and $\overrightarrow{s}_{k}$ appear in $f_d\cdots f_1$.
	A contradiction.
\end{proof}

\begin{definition}[\sc $\sigma_I(P)$]
	By Lemma \ref{lem:same_direction_in_interval}, all of the unit shifts between entries in the same $P$-interval $I$ are in the same direction throughout $P$. 
	If this direction is left (right), we say that \textbf{$P$ shifts left (right) in $I$}.
	Denote the number of unit shifts between entries in some $P$-interval $I\subseteq [0,m+1]$ within a pivot path $P$ by $\sigma_I(P)$.
\end{definition}

\begin{lemma}\label{lem:number_of_shift_in_a_p_interval}
	Let $I=[p_1+1,p_2]$ be a $P$-interval of some pivot path $P$ of $v\in\Znm$.
	If $I$ contains one of the buckets, then also assume that no step in $P$ shifts a unit from an empty bucket (i.e. $\overleftarrow{s}_m((\dots,0))=((\dots,n-1))$ and $\overrightarrow{s}_0((0,\dots))=((n-1,\dots))$ do not appear in $P$). Then
	
	\begin{enumerate}
		\item If $P$ shifts left in $I$, then $\sigma_I(P) = \sum_{i\in I}iv_i$.
		\item If $P$ shifts right in $I$, then $\sigma_I(P) = \sum_{i\in I}(m+1-i)v_i$.
	\end{enumerate}
\end{lemma}
\begin{proof}
	By Lemma \ref{lem:same_direction_in_interval}, $P$ shifts either right or left in $I$.
	Assume that $P$ shifts left in $I$ (the proof of the other case is similar).
	Note that every left unit shift in $P$ between entries in $I$ reduces the value of the sum $\sum_{i\in I}iv_i$ exactly by $1$, since a unit is never shifted left from an empty right bucket.
	Therefore $\sum_{i\in I}iv_i\geq 0$ and $\sigma_I(P) = \sum_{i\in I}iv_i$.
\end{proof}

\begin{lemma}\label{lem:sufficient_for_number_of_shifts}
	Let $I=[p_1+1,p_2]$ be a $P$-interval of some pivot path $P$ of $v\in\Znm$.
	\begin{enumerate}
		\item Assume that $p_2=m+1$ and that $P$ shifts left in $I$. If $\sum_{i=j}^{m+1}v_i\geq 0$ for every $j\in I$, then there is no left unit shift in $P$ from an empty right bucket.
		\item Assume that $p_1=-1$ and that $P$ shifts right in $I$. If $\sum_{i=0}^{j}v_i\geq 0$ for every $j\in I$, then there is no right unit shift in $P$ from an empty left bucket.
	\end{enumerate}
\end{lemma}
\begin{proof}
	We prove the first case and similar arguments apply to the second case.
	Note that a left unit shift from an empty right bucket increases $\sum_{i\in I}iv_i$. 
	On the other hand, every other left unit shift between entries in $I$ (we call such unit shifts \textit{simple} left unit shifts in the rest of this proof) reduces the value of the sum $\sum_{i\in I}iv_i$ exactly by $1$.
	Therefore, it is sufficient to prove the following.
	\begin{claim}
		It is possible to set to $0$ every $v_i$ in $I$ using only simple left unit shifts between entries in $I$.
	\end{claim}
	\begin{claimProof}
	Assume first that $I$ contains no negative entries.
	If $p_1=-1$, then the claim is trivial.
	If $p_1$ is an inner pivot, then $v_i=0$ for every $i\in I$, since, in this case, a left unit shift between entries in $I$ cannot decrease the sum $\sum_{i\in I}v_i$ (implying, in fact, that $v_{m+1}=0$ and $|I|=1$).

	Otherwise, let $j$ be maximal in $I$ such that $v_j=-1$.
    Of course, $1\leq j\leq m$.
	The assumption that $\sum_{i=j}^{m+1}v_i\geq 0$ implies that $\sum_{i=j+1}^{m+1}v_i > 0$.
	Therefore, it is possible to set $v_j$ to $0$ using only simple left unit shifts between entries in $[j,m+1]$ without creating new negative entries in $I$.
	Note that $\sum_{i=j}^{m+1}v_i$ does not change in this process, since all of the unit shifts are simple.
	Therefore, this can be repeated until $I$ contains no negative entries.
	\end{claimProof}
\end{proof}

\begin{lemma}\label{lem:no_pivots}
    Let $v\in\Znm$ such that $v$ has no inner pivots. 
    Then $d(v,0)\leq \lfloor\frac{n(m+1)}{2}\rfloor$.
\end{lemma}
\begin{proof}
    $0<v_0,v_{m+1}<n$ and $\sum_{i=0}^{m+1}v_i = n$, since $v$ has no inner pivots.
	The assumption $\piv(v)=\{-1,m+1\}$ also implies that both $\sum_{i=0}^{j}v_i> 0$ and $\sum_{i=j}^{m+1}v_i> 0$ for every $0\leq j\leq m+1$.
    Note that $I=[0, m+1]$ is a $P$-interval for every $(m+1)$-pivot path $P$ and every such path shifts left in $I$.
    Similarly, every $(-1)$-pivot path $P$ shifts right in $I$. 
	Therefore, by Lemmas \ref{lem:number_of_shift_in_a_p_interval} and \ref{lem:sufficient_for_number_of_shifts}, $\ps_{m+1}(v)=\sum_{i=0}^{m+1}iv_i$ and $\ps_{-1}(v)=\sum_{i=0}^{m+1}(m+1-i)v_i$.
	
	$\sum_{i=0}^{m+1}iv_i + \sum_{i=0}^{m+1}(m+1-i)v_i= (m+1) \sum_{i=0}^{m+1} v_i =n(m+1)$.
	Therefore, by Lemma \ref{lem:geodesic_is_a_pivot_path}, 
	$d(v,0)=\min\{\ps_{-1}(v),\ps_{m+1}(v)\}\leq \lfloor\frac{\ps_{-1}(v)+\ps_{m+1}(v)}{2}\rfloor=\lfloor\frac{n(m+1)}{2}\rfloor.$
	
\end{proof}

\begin{lemma}\label{lem:ecc_lower_bound}
	For every dYoke graph $\Znm$ we have $\ecc_{\Znm}(0)\geq\lfloor\frac{n(m+1)}{2} \rfloor$.
\end{lemma}
\begin{proof}
    By Lemma \ref{lem:ecc_lower_bound_ynm} and Corollary \ref{cor:dist_equal_when_positive}, $\ecc_{\Znm}(0)\geq\ecc_{\Ynm}(0)\geq\lfloor\frac{n(m+1)}{2} \rfloor$.
\end{proof}

\subsection{Computation of the Eccentricity}\label{subsec:n_leq_m_Znm}
\begin{lemma}[$n=1$]\label{lem:corner_case_neqo_znm} Let $m\geq 0$. Then
    $$\ecc_{\Znm[1][m]}(0)= \binom{\lceil \frac{m}{2}\rceil + 1}{2} + \binom{\lfloor \frac{m}{2}\rfloor + 1}{2}.$$
\end{lemma}
\begin{proof}
    The case $m=0$ is trivial. Assume that $m>0$. Then
    $\piv(v)=\{-1,\ldots,m+1\}$ for every $v\in\Znm[1][m]$.
    Therefore, by Fact \ref{fac:split_sum_center}, Observation\nolinebreak \ref{obs:closer_pivot_is_better_znm} and Corollary \ref{cor:trivial_upper_bound}, 
    $d(v,0)\leq \ps_{\lceil\frac{m}{2}\rceil}(v)\leq \binom{\lceil \frac{m}{2}\rceil + 1}{2} + \binom{\lfloor \frac{m}{2}\rfloor + 1}{2}$.
    By Observation \ref{obs:trivial_path}, this upper bound is obtained by the vertex $v\in\Ynm[1][m]$ with $v_i=1$ for every $1\leq i\leq m$.
\end{proof}

\begin{lemma}\label{lem:ecc_n_geq_m}
    If $0\leq m\leq n$, then $\ecc_{\Znm}(0) = \lfloor\frac{n(m+1)}{2}\rfloor$.
\end{lemma}
\begin{proof}
    The case $m=0$ is true by Observation \ref{obs:corner_case_meqz}. Assume that $m>0$. Then
    by Lemma \ref{lem:ecc_lower_bound}, it is sufficient to show that $\ecc_{\Znm}(0) \leq \lfloor\frac{n(m+1)}{2}\rfloor$.
    Let $v\in\Znm$. 
    By Lemma \ref{lem:no_pivots}, we can assume that $v$ has some inner pivot $p\in\piv(v)$.
    By Observation \ref{obs:trivial_path} and Fact \ref{fac:split_sum_center}, $d(v,0) \leq \sum_{i=0}^{p}i  + \sum_{i=0}^{m-p}i \leq \binom{m+1}{2}$.
    Since $m \le n$, $\binom{m+1}{2} \le \frac{n(m+1)}{2}$ and therefore $\binom{m+1}{2}\leq\lfloor\frac{n(m+1)}{2}\rfloor$.
\end{proof}

Throughout the rest of this subsection we assume that $2\leq n\leq m$.
\begin{definition}
For $v\in \Znm$ define:
\begin{enumerate}
	\item $p_l(v)=\max\{p\in \piv(v):p\leq \frac{m}{2}\}$.
	\item $p_r(v)=\min\{p\in \piv(v):p > \frac{m}{2}\}$.
	\item $I_c(v) = [p_l(v)+1,p_r(v)]$.
	\item $h(v)=\min\{|p-\frac{m}{2}|:p\in \piv(v)\}$.
	\item $\hnm = 
	\begin{cases}
	\frac{n}{2} & \mbox{if } 2\divides (m-n) \\
	\frac{n+1}{2} & \mbox{if } 2\ndivides (m-n)
	\end{cases}.$
\end{enumerate}
We abbreviate $p_l$, $p_r$ and $I_c$ when $v$ is evident.
\end{definition}

Recall the definitions of $\uz$, $\dz$, $\uo$, $\doo$ and $\dnm$ (in $\Ynm$) from Definitions \ref{def:uz}, \ref{def:uo} and \ref{def:dnm}.
Note that $\uz$ and $\uo$ can also be considered as vertices in $\Znm$ and by Corollary \ref{cor:dist_equal_when_positive}, $d_{\Znm}(\uz, 0)=d_{\Ynm}(\uz, 0)=\dz$ and $d_{\Znm}(\uo, 0)=d_{\Ynm}(\uo, 0)=\doo$.

\begin{observation}\label{obs:ecc_lower_bound}
    $\ecc_{\Znm}(0) \geq \dnm$.
\end{observation}

In the rest of this subsection, we prove that $\dnm$ is also an upper bound on the distance of an arbitrary vertex $v$ in $\Znm$ from $0$.
Note that $\sum_{i\in I_c}v_i\in\{0,\pm n\}$ for every $v$ in $\Znm$.
We split this proof into three cases.

\begin{enumerate}
	\item $h(v)<\hnm$ (Lemma \ref{lem:close_to_middle_pivot_znm}).
	\item $h(v)\geq\hnm$ and $\sum_{i\in I_c(v)}v_i=\pm n$ (Lemma \ref{lem:middle_interval_is_n_znm}).
	\item $h(v)\geq\hnm$ and $\sum_{i\in I_c(v)}v_i=0$ (Lemma \ref{lem:zero_middle_interval}).
\end{enumerate}

\begin{lemma}\label{lem:close_to_middle_pivot_znm}
	Let $v\in\Znm$ such that $h(v)<\hnm$. Then $d(v,0)\leq\dnm$.
\end{lemma}
\begin{proof}
    The proof of this lemma is identical to the proof of Lemma \ref{lem:close_to_middle_pivot} in which Corollary \ref{cor:pivot_path_len} is replaced by Corollary \ref{cor:trivial_upper_bound}.
\end{proof}

We now deal with the case $h(v)\geq\hnm$, $\sum_{i\in I_c}v_i=\pm n$.
We first reduce to the case $\sum_{i\in I_c}v_i=n$.

\begin{definition}\label{def:aut_mu}
    Let $\mu:\Znm\rightarrow \Znm$ be the map defined by $\mu(v)_i=-v_i$ for every $0\leq i\leq m+1$.
\end{definition}

Note that if $v_0 > 0$, then $\mu(v)_0=n-v_0>0$ and similarly for $v_{m+1}$.
It is straightforward to verify that $\mu$ is an automorphism of $\Znm$.
 \pagebreak
\begin{observation}\label{obs:neg_Ic_to_positive}
    Let $v\in\Znm$ such that $\sum_{i\in I_c}v_i=-n$.
    Then $I_c\subseteq [1,m]$ and $\sum_{i\in I_c}\mu(v)_i=n$.
\end{observation}
\begin{proof}
    If $I_c\nsubseteq [1,m]$, then either $p_l=-1$ or $p_r=m+1$.
    Assume that $p_l=-1$.
    Recall that the buckets are identified with non-negative integers.
    Since $0\notin\piv(v)$, we have $v_0>0$. 
    This implies that $\sum_{i\in I_c}v_i\in\{0,n\}$. A contradiction.
    If $p_r=m+1$, then $\sum_{i\in I_c}v_i\in\{0,n\}$ following similar arguments.
    Finally, if $I_c\subseteq [1,m]$ and $\sum_{i\in I_c}v_i=-n$ then $\sum_{i\in I_c}\mu(v)_i=n$, since  $\piv(v)=\piv(\mu(v))$.
\end{proof}

The following lemma generalizes Lemma \ref{lem:no_pivots}.

\begin{lemma}\label{lem:v_hat_dist_from_zero_center}
	Let $v\in\Znm$ such that $\sum_{i\in I_c}v_i=n$.
	Then 
    $$d(v,0)\leq \sum_{i=1}^{p_l}i + \sum_{i=1}^{m-p_r}i + \lfloor\frac{n(m+1)}{2}\rfloor.$$
\end{lemma}
\begin{proof}
	The sum $\sum_{i=1}^{p_l}i + \sum_{i=1}^{m-p_r}i$ is an upper bound on the number of steps needed to set to $0$ every $v_i$ with $i\in [1,m]\setminus I_c$; note that necessarily, following these steps, a bucket of $v$ not in $I_c$, is also equal to $0$.
	We can therefore assume that $v_i=0$ for every $i\notin I_c$ and prove that $d(v,0)\leq \lfloor\frac{n(m+1)}{2}\rfloor$.
    The assumption implies that $d(v, 0)=\min\{\ps_{p_l}(v), \ps_{p_r}(v)\}$.
	Since $\sum_{i\in I_c}v_i=n$, every $p_r$-pivot path of $v$ shifts left in $[0,p_r]$ and every $p_l$-pivot path of $v$ shifts right in $[p_l+1,m+1]$.
	Moreover, the assumption $\sum_{i\in I_c}v_i=n$ implies that $\sum_{i=j}^{p_r}v_i> 0$ for every $j\in I_c$ and $\sum_{i=p_l+1}^{j}v_i> 0$ for every $j\in I_c$.
	Therefore, by Lemmas \ref{lem:number_of_shift_in_a_p_interval} and \ref{lem:sufficient_for_number_of_shifts}, $d(v,0)=\min\{\ps_{p_l}(v),\ps_{p_r}(v)\}\leq \lfloor\frac{\ps_{p_l}(v)+ \ps_{p_r}(v)}{2}\rfloor= \lfloor\frac{\sum_{i\in I_c}iv_i + \sum_{i\in I_c}(m+1-i)v_i}{2}\rfloor=\lfloor\frac{n(m+1)}{2}\rfloor.$
\end{proof}
%

\begin{observation}\label{obs:dz_doo_split}
	\begin{enumerate}
        \item[]
		\item If $2|(m-n)$, then $\dz = \sum_{i=1}^{\frac{m-n}{2}}i + \sum_{i=1}^{m-\frac{m+n}{2}}i + \frac{n(m+1)}{2}$.
		\item $\doo = \sum_{i=1}^{\lfloor\frac{m-n}{2}\rfloor}i + \sum_{i=1}^{m-\lceil\frac{m+n}{2}\rceil}i + \lfloor\frac{n(m+1)}{2}\rfloor$.
	\end{enumerate}
\end{observation}
\begin{proof}
    A direct (though tedious) comparison with Lemma \ref{lem:dist_of_uz}.
    Alternatively, it follows by observing that $|p_l-\frac{m}{2}|=|p_r-\frac{m}{2}|$ in both cases under consideration.
\end{proof}

\begin{lemma}\label{lem:middle_interval_is_n_znm}
	Let $v\in\Znm$ such that $h(v)\geq\hnm$ and $\sum_{i\in I_c(v)}v_i=\pm n$. 
	Then $d(v,0) \leq \dnm.$
\end{lemma}
\begin{proof}
    By Observation \ref{obs:neg_Ic_to_positive}, we can assume that $\sum_{i\in I_c(v)}v_i=n$.
	Assume that $2|(m-n)$.
	By Observation \ref{obs:dz_doo_split}, 
	$\dnm = \dz = \sum_{i=1}^{\frac{m-n}{2}}i + \sum_{i=1}^{m-\frac{m+n}{2}}i + \frac{n(m+1)}{2}.$
	By Lemma \ref{lem:v_hat_dist_from_zero_center}, $d({v},0)\leq \sum_{i=1}^{p_l}i + \sum_{i=1}^{m-p_r}i + \frac{n(m+1)}{2}.$
	Since $h(v)\geq\hnm$, $p_l\leq\frac{m-n}{2}$ and $p_r\geq\frac{m+n}{2}$.
	Therefore $d(v,0)\leq \dnm$.    
	
	Otherwise assume that $2\ndivides(m-n)$.
	By Observation \ref{obs:dz_doo_split}, 
	$\dnm\geq \doo = \sum_{i=1}^{\lfloor\frac{m-n}{2}\rfloor}i + \sum_{i=1}^{m-\lceil\frac{m+n}{2}\rceil}i + \lfloor\frac{n(m+1)}{2}\rfloor.$
	By Lemma \ref{lem:v_hat_dist_from_zero_center}, $d({v},0)\leq \sum_{i=1}^{p_l}i + \sum_{i=1}^{m-p_r}i + \lfloor\frac{n(m+1)}{2}\rfloor.$
	Since $h(v)\geq\hnm$, $p_l\leq \lfloor\frac{m-n}{2}\rfloor$ and $p_r\geq\lceil\frac{m+n}{2}\rceil$.
	Therefore $d(v,0)\leq \dnm$.    
\end{proof}

\begin{lemma}\label{lem:upper_middle_is_zero}
	Let $v\in\Znm$ such that $I_c(v)\subseteq[1,m]$ and $\sum_{i\in I_c(v)}v_i=0$. Then 
	$$d(v,0) \leq \sum_{i=1}^{p_l}i + \sum_{i=1}^{m-p_r}i + \lfloor\frac{\Icwidth}{2}\rfloor\lceil\frac{\Icwidth}{2}\rceil$$ 
	where $\Icwidth=|I_c(v)|$.
\end{lemma}
\begin{proof}
	As in the beginning of the proof of Lemma \ref{lem:v_hat_dist_from_zero_center}, we can assume that $v_i=0$ for every $i\notin I_c$ and prove that $d(v,0)\leq \lfloor\frac{\Icwidth}{2}\rfloor\lceil\frac{\Icwidth}{2}\rceil$.
	
	Let $P$ be any $p_l$- or $p_r$-pivot path of $v$.
	$I_c$ is clearly a $P$-interval (namely, both $p_l$ and $p_r$ are walls of $P$), since $\sum_{i\in I_c}v_i=0$ and $I_c(v)\subseteq[1,m]$.
	Assume without loss of generality, that $P$ shifts left in $I_c$.
	By Lemma \ref{lem:number_of_shift_in_a_p_interval}, $\ps_{I_c}(P) = \sum_{i\in I_c}iv_i$.
	The number of entries of $v$ that are equal to $1$ is equal to the number of entries of $v$ that are equal to $-1$, since $\sum_{i\in I_c}v_i = 0$.
	This number is at most $\lfloor\frac{\Icwidth}{2}\rfloor$. 
    The maximum value of $\sum_{i\in I_c}iv_i$ is obtained when the first $\lfloor\frac{\Icwidth}{2}\rfloor$ entries in $I_c$ are equal to $-1$ and the last $\lfloor\frac{\Icwidth}{2}\rfloor$ entries are equal to $1$.
	Therefore $\sum_{i\in I_c}iv_i\leq \lfloor\frac{\Icwidth}{2}\rfloor\lceil\frac{\Icwidth}{2}\rceil$.
\end{proof}

\begin{fact}\label{fac:split_sum}
	Let $a$ and $b$ be two non-negative integers.
	Then 
    $$\sum_{i=1}^a i + \sum_{i=1}^{b} i \geq \lfloor\frac{a+b}{2}\rfloor\lceil\frac{a+b}{2}\rceil.$$
\end{fact}
\pagebreak

\begin{lemma}\label{lem:zero_middle_interval}
	Let $v\in\Znm$ such that $h(v)\geq\hnm$ and $\sum_{i\in I_c}v_i=0$.
	Then $d(v,0)\leq\dnm$.
\end{lemma}
\begin{proof}
	If both $p_l$ and $p_r$ are outer pivots, then both buckets are non-empty and $\sum_{i\in I_c}v_i=n$, since in this case, $v$ has no pivots other than $p_l$ and $p_r$; but we assumed $\sum_{i\in I_c}v_i=0$. 
	Assume that exactly one of $p_l$ and $p_r$, say $p_l$, is an outer pivot (and that the other pivot, $p_r$, is an inner pivot).
	Let $\mu$ be the automorphism of $\Znm$ from Definition \ref{def:aut_mu}.
    Clearly, $I_c(v)=I_c(\mu(v))$.
    By assumption $v_0>0$ and therefore $\mu(v)_0=n-v_0>0$; hence $\sum_{i\in I_c}\mu(v)_i=n-\sum_{i\in I_c} v_i=n$.
	Since $0$ is a fixed point of $\mu$, by Lemma \ref{lem:middle_interval_is_n_znm}, $d(v,0)=d(\mu(v),0)\leq \dnm$, as required.
	
	Otherwise both $p_l$ and $p_r$ are inner pivots, that is, $I_c(v)\subseteq[1,m]$.
	By Lemma \ref{lem:upper_middle_is_zero}, 
	$d({v},0) \leq 
	\sum_{i=1}^{p_l}i + \sum_{i=1}^{m-p_r}i + \lfloor\frac{\Icwidth}{2}\rfloor\lceil\frac{\Icwidth}{2}\rceil$ where $\Icwidth=|I_c(v)|$.
	Let $p=\lfloor\frac{m+n}{2}\rfloor$ so that, by Lemma \ref{lem:dist_of_uz}, $\dz=\sum_{i=1}^p i + \sum_{i=1}^{m-p} i$.
	Note that $|p-\frac{m}{2}|\leq \hnm$.
	The assumption $h(v)\geq\hnm$ implies that $p_l\leq p\leq p_r$. Therefore
	\begin{align*}
	\dnm - d(v,0) \geq
	\dz - d(v,0) \geq 
	\sum_{i=p_l+1}^p i + \sum_{i=m-p_r+1}^{m-p} i - \lfloor\frac{\Icwidth}{2}\rfloor\lceil\frac{\Icwidth}{2}\rceil.
	\end{align*}
	Clearly, $\sum_{i=p_l+1}^p i + \sum_{i=m-p_r+1}^{m-p} i \geq \sum_{i=1}^{p-p_l} i + \sum_{i=1}^{p_r-p} i$.
	Let $a=p-p_l$, let $b=p_r-p$ and note that $a+b=\Icwidth$. Therefore, by Fact \ref{fac:split_sum},
	$
	\dnm - d(v,0) \geq \sum_{i=1}^a i + \sum_{i=1}^{b} i - \lfloor\frac{\Icwidth}{2}\rfloor\lceil\frac{\Icwidth}{2}\rceil\geq 0.
	$
\end{proof}

\begin{lemma}\label{lem:ecc_n_leq_m}
	Let $2\leq n\leq m$ be two integers. 
	Then $\ecc_{\Znm}(0) = \dnm$.
\end{lemma}
\begin{proof}
    Combine Observation \ref{obs:ecc_lower_bound} with Lemmas \ref{lem:close_to_middle_pivot_znm}, \ref{lem:middle_interval_is_n_znm} and \ref{lem:zero_middle_interval}.
\end{proof}

\begin{theorem}\label{thm:ecc_znm}
    For all values of $n$ and $m$, the eccentricity of $0$ in $\Znm$ is equal to the eccentricity of $0$ in $\Ynm$ as it appears in Theorem \ref{thm:ynm_ecc}.
\end{theorem}
\begin{proof}
    This theorem is the combined result of Lemmas \ref{lem:corner_case_neqo_znm}, \ref{lem:ecc_n_geq_m} and \ref{lem:ecc_n_leq_m}.
\end{proof}

\chapter{Automorphisms of $\Ynm$}\label{chp:automorphisms}
This chapter deals with automorphisms of $\Ynm$.
In Section \ref{sec:fundamental_auto}, we introduce three fundamental automorphisms of $\Ynm$ and find the structure of the group they generate.
In Section \ref{sec:complete_auto_group}, we show that the group generated by these three fundamental automorphisms is, in fact, precisely the automorphism group of $\Ynm$, whenever $m\neq 2$ and $(n,m)\neq (1,3)$.
We cover the case $(n,m)=(1,3)$ separately.

\section{Fundamental Automorphisms}\label{sec:fundamental_auto}
\begin{definition}\label{def:typical_automorphisms}
	Let $\varphi$, $\psi$ and $\tau$ be the following three maps from $\Ynm$ to $\Ynm$:
	$$\begin{array}{c c c l}
	& \varphi(v_0,v_1,\dots,v_m,v_{m+1}) & = & (v_0+1,v_1,\dots,v_m,v_{m+1}-1) \\
	& \psi(v_0,v_1,\dots,v_m,v_{m+1})    & = & (v_{m+1},v_m,\dots,v_1,v_{0}) \\
	& \tau(v_0,v_1,\dots,v_m,v_{m+1})    & = & (-v_0,1-v_1,\dots,1-v_{m},-(m+v_{m+1})) \\
	\end{array}$$
\end{definition}

Note that $\varphi$ and $\tau$ are essentially as in Definition \ref{def:two_automorphisms}. We define them here again for completeness.

\begin{observation}
$\varphi$, $\psi$ and $\tau$ are automorphisms of $\Ynm$.
\end{observation}

\begin{definition}
	Denote $\AGnm=\left<\varphi, \psi, \tau\right>$; the subgroup of $\aut(\Ynm)$ generated by $\{\varphi, \psi, \tau\}$.
\end{definition}

Recall that the dihedral group $D_n$ ($\forall n\geq 1$) is the group of order $2n$ with presentation:
$$D_n=\left<g,h\; |\; g^n=h^2=1, hgh=g^{-1}\right>.$$

Denote by $o(g)$ the order of an element $g$ in a finite group.

\begin{theorem}[The structure of $\AGnm$]
    \label{thm:anm_structure}
	\begin{enumerate}
    \item[]
    \item $\AGnm[1][0]\cong\{1\}$ and $\AGnm[2][0]\cong \AGnm[1][1]\cong C_2$, the cyclic group of order $2$.
    \item $\AGnm[n][0] \cong D_n$ ($\forall n>2$).
    \item $\AGnm[1][m] \cong C_2\times C_2$ ($\forall m>1$). 
    \item If $n>1$ and $m>0$, then
    \begin{enumerate}
        \item if at least one of $n$ and $m$ is odd, then $\AGnm \cong D_{2n}$;
        \item otherwise ($n$ and $m$ are even), $\AGnm \cong D_{n}\times C_2$.
    \end{enumerate}
    \end{enumerate}
\end{theorem}

This theorem is the combined result of forthcoming Lemmas \ref{lem:automorphism_group_special_cases} and \ref{lem:automorhism_group_structure}.

\begin{observation}\label{obs:anm_relations}
    The following two properties hold in $\AGnm$ for any $n$ and $m$.
    \begin{enumerate}
        \item $\psi\varphi=\varphi^{-1}\psi$ and $\tau\varphi=\varphi^{-1}\tau$.
        \item $(\tau\psi)^2=\varphi^m$.
    \end{enumerate}
\end{observation}
\begin{proof}
    The first property is trivial.
    Let $u\in\Ynm$ and $f=\tau\psi$. Then 
    \begin{equation*}
    \begin{split}
    (u_0,\dots,u_{m+1}) & \overset{f}{\mapsto} 
    (-u_{m+1},1-u_m,\dots,1-u_1,-(m+u_0)) \\ 
    & \overset{f}{\mapsto} (u_0+m,u_1,\dots,u_m,u_{m+1}-m).
    \end{split}
    \end{equation*}
    This proves the second property.
\end{proof}

\begin{fact}\label{fac:dihedral_gens}
	Let $n>2$. If $g$ and $h$ are two elements in a group such that $o(g)=n$, $o(h)=2$ and $gh=hg^{-1}$, then $\left<g,h\right>\cong D_n$.
\end{fact}


%
%
%
\begin{observation}\label{obs:anm_equal}
    \begin{enumerate}
        \item[]        
        \item $\varphi = \psi$ if and only if $n=1$ and $m\leq 1$;
        \item $\varphi = \tau$ if and only if $n=1$ and $m=0$;
        \item $\psi = \tau$ if and only $m=0$;
    \end{enumerate}
\end{observation}

\begin{observation}\label{obs:anm_order}
    The orders of $\varphi$, $\psi$ and $\tau$ in $\AGnm$ are as follows:
    \begin{enumerate}
        \item $\psi=Id$ if and only if $(n,m)\in\{(1,0), (1,1), (2,0)\}$, otherwise $o(\psi)=2$;
        \item $\tau=Id$ if and only if $(n,m)\in\{(1,0), (2,0)\}$, otherwise $o(\tau)=2$;
        \item $o(\varphi)=n$.
    \end{enumerate}
\end{observation}

\begin{corollary}\label{cor:anm_generators_distinct}
    If $n>1$ and $m>0$, then $\varphi$, $\psi$ and $\tau$ are distinct and of orders $n$, $2$ and $2$, respectively.
\end{corollary}

\begin{lemma}\label{lem:automorphism_group_special_cases}
	\begin{enumerate}
		\item[]
        \item $\AGnm[1][0]\cong\{1\}$ and $\AGnm[2][0]\cong \AGnm[1][1]\cong C_2$.
		\item $\AGnm[n][0] \cong D_n$ ($\forall n>2$).
		\item $\AGnm[1][m] \cong C_2\times C_2$ ($\forall m>1$). 
	\end{enumerate}
\end{lemma}
\begin{proof}
    The first case is trivial.
    By Observations \ref{obs:anm_equal} and \ref{obs:anm_order}, if $n>2$ and $m=0$, then $\psi=\tau\neq Id$. Moreover, by Observation \ref{obs:anm_relations} (Property $1$), Observation \ref{obs:anm_order} and Fact \ref{fac:dihedral_gens}, $\left<\varphi,\psi\right>$ is isomorphic to $D_n$.
    This proves the second case.
    If $n=1$ and $m>1$, then $\varphi=Id$, $\psi\neq\tau$, $o(\psi)= o(\tau)=2$ and $\psi\tau=\tau\psi$.
    This proves the third case.
\end{proof}
\pagebreak
\begin{lemma}\label{lem:automorphism_group_order}
If $m>0$ and $(n,m)\neq (1,1)$, then $|\AGnm|=4n$.
\end{lemma}
\begin{proof}
    If $n=1$ (and $m>1$), then $|\AGnm[1][m]|=4$, by Lemma \ref{lem:automorphism_group_special_cases}.
    Assume that $n>1$ (and $m>0$).     
	By Observation \ref{obs:anm_relations}, every element in $\AGnm$ can be written in the form $\varphi^k\tau^i\psi^j$ where $0\leq k\leq n-1$ and $i,j\in\{0,1\}$.
    Therefore $|\AGnm|\leq 4n$.

	Assume to the contrary that $\varphi^k\tau^i\psi^j=Id$ for $(k, i, j)\neq (0,0,0)$.
    By Corollary \ref{cor:anm_generators_distinct}, $\varphi$, $\psi$ and $\tau$ are distinct and of orders $n$, $2$ and $2$, respectively.
    Therefore, $k\neq 0$.
    If $i=0$, then $\varphi^k\psi=Id$. 
    This is impossible, since $\varphi^k\psi(0)\neq 0$.
    On the other hand, if $i=1$, then $\varphi^k\tau\psi^j=Id$.
    This is also impossible, again, since $\varphi^k\tau\psi^j(0)\neq 0$.
    This is a contradiction, completing the proof.
\end{proof}

\begin{lemma}\label{lem:automorhism_group_structure}
	If $n>1$ and $m>0$, then
	\begin{enumerate}
		\item if at least one of $n$ and $m$ is odd, then $\AGnm \cong D_{2n}$;
		\item otherwise, $\AGnm \cong D_{n}\times C_2$.
	\end{enumerate}
\end{lemma}
\begin{proof}
	We start with Case 1.
	Assume that at least one of $n$ and $m$ is odd.
	Therefore, there exists some integer $k$ such that $m+2k$ and $n$ are coprime.
	For this $k$ define $\alpha=\varphi^k\tau\psi$ and $\beta=\tau\varphi^k$.
	We show that  $\left<\alpha, \beta\right>\cong D_{2n}$.
	This proves Case 1, since by Lemma \ref{lem:automorphism_group_order}, $|\AGnm|=4n$.
	
	$o(\alpha)$ is even, since $m>0$ and $(\alpha^i(0))_j=1$ for every odd integer $i$ and $1\leq j\leq m$. Therefore, $o(\alpha)=2o(\alpha^2)$.
	By Observation \ref{obs:anm_relations}, $$\alpha^2=\varphi^k\tau\psi\varphi^k\tau\psi=\varphi^k\tau\varphi^{-k}\psi\tau\psi=\varphi^{2k}\tau\psi\tau\psi=\varphi^{m+2k}.$$
	Since $m+2k$ and $n$ are coprime, this implies that $o(\alpha^2)=n$.
	This proves that $o(\alpha)=2n$.
	Now $\beta(0)\neq 0$, so by Observation \ref{obs:anm_relations},  $o(\beta)=2$ and $\beta\alpha\beta=(\tau\varphi^k)(\varphi^k\tau\psi)(\tau\varphi^k)=\psi\tau\varphi^{-k}=\alpha^{-1}$.
    By Fact \ref{fac:dihedral_gens}, this proves that $\left<\alpha, \beta\right>\cong D_{2n}$, as required.
	
	We proceed to prove Case 2.
    $\AGnm\cong D_n\times C_2$ if and only if there exist two subgroups $N\cong D_n$ and $M\cong C_2$ of $\AGnm$ such that $N\cap M=\{Id\}$, $NM=\AGnm$ and the generator of $M$ commutes with the generators of $N$.
    Let $N=\left<\varphi, \psi\right>$ and let $M=\left<\gamma\right>$ where $\gamma=\varphi^{\frac{m}{2}}\psi\tau$ ($=\varphi^{-\frac{m}{2}}\tau\psi$).

    If $n=2$ (and $m>0$), then $\varphi$ and $\psi$ clearly commute.
    Moreover, by Observations \ref{obs:anm_equal} and \ref{obs:anm_order}, they are also distinct and of order $2$.
    Therefore, $N\cong C_2\times C_2\cong D_2$.
    For $n>2$, $N\cong D_n$ by Observations \ref{obs:anm_relations} and \ref{obs:anm_order} and Fact \ref{fac:dihedral_gens}.
    $M\cong C_2$, since $\gamma^2=(\varphi^{-\frac{m}{2}}\tau\psi)^2=\varphi^{-m}(\tau\psi)^2=Id$ and $\gamma(0)\neq 0$.
    By Observation \ref{obs:anm_relations}, $\gamma$ commutes with $\varphi$, 
    and $\psi\gamma\psi=\psi(\varphi^{-\frac{m}{2}}\tau\psi)\psi=\varphi^{\frac{m}{2}}\psi\tau=\gamma$. Therefore, $\gamma$ commutes also with $\psi$.
    For every $1\leq j\leq m$, $\varphi(0)_j=\psi(0)_j=0$ but $\gamma(0)_j=1$.
    Therefore $\gamma\notin N$ and $N\cap M=\{Id\}$.
    $\AGnm\subseteq NM$, since $\tau=\psi\varphi^{-\frac{m}{2}}\gamma\in NM$.
    This completes the proof.
%
%
\end{proof}

\section{The Automorphism Group of $\Ynm$}\label{sec:complete_auto_group}
The main result of this section is the following theorem.

\begin{theorem}
    \label{thm:aut_structure}
    \begin{enumerate}
        \item[]
        \item $\aut(\Ynm[1][3])\cong D_4\times C_2.$
        \item $\aut(\Ynm)\cong \AGnm$ for all $n\geq 1$ and $m\geq 0$ s.t. $m\neq 2$ and \mbox{$(n,m)\neq (1, 3)$}.
    \end{enumerate}
\end{theorem}

This theorem is the combined result of forthcoming Lemmas \ref{lem:y13}, \ref{lem:Anm_is_aut} and Observation \ref{obs:Anm_is_aut}.

\begin{lemma}\label{lem:y13}
    $$\aut(\Ynm[1][3])\cong D_4\times C_2.$$
\end{lemma}
\begin{proof}
    Since $v_0=v_4=0$ for every vertex $v\in\Ynm[1][3]$, we identify each vertex $v$ with $(v_1v_2v_3)$ (since $v_i\in\{0,1\}$ for every $1\leq i\leq 3$, the vertices are well defined without commas), see Figure \ref{fig:Y13}.
    Let $f_1$, $f_2$ and $f_3$ be three automorphisms of $\Ynm[1][3]$, each one interchanging a single pair of vertices and fixing the rest of the graph:
    $f_1$ interchanges $(100)$ and $(001)$, $f_2$ interchanges $(101)$ and $(010)$ and $f_3$ interchanges $(110)$ and $(011)$.
    Let $\tau$ be as in Definition \ref{def:typical_automorphisms} and note that $\tau$ interchanges $(000)$ and $(111)$. 
    
    $(000)$ and $(111)$ are the only two vertices of valency $2$ in $\Ynm[1][3]$.
    Therefore, for every $\sigma\in\aut(\Ynm[1][3])$, either $\sigma\in \st(000)$, the stabilizer of $(000)$, or $\tau\sigma\in \st(000)$.
    This implies that $|\aut(\Ynm[1][3])|=2|\st(000)|$.
    There are precisely two vertices at each distance $1$, $2$ and $3$ from $(000)$ (and a unique vertex at distance $4$). 
    Therefore $|\st(000)|\leq 2^3=8$.
    $(000)$ is fixed by $f_1$, $f_2$ and $f_3$ and $|\left<f_1, f_2, f_3\right>|= |C_2\times C_2\times C_2|=8$.
    This implies that $|\st(000)|=8$ and therefore, $|\aut(\Ynm[1][3])|=16$.
    
    Let $\eta = \tau f_1$.
    Then $o(\eta)=4$ and $\eta f_3=f_3\eta^{-1}$.
    Therefore, by Fact \ref{fac:dihedral_gens}, $\left<\eta, f_3\right>\cong D_4$.
    $f_2\notin \left<\eta, f_3\right>$ and $f_2$ commutes with both $\eta$ and $f_3$.
    Therefore, $\left<\eta, f_2, f_3\right>\cong D_4\times C_2$.
    Since $|D_4\times C_2|=16=|\aut(\Ynm[1][3])|$, we have $\aut(\Ynm[1][3])\cong\left<\eta, f_2, f_3\right>\cong D_4\times C_2$.
\end{proof}
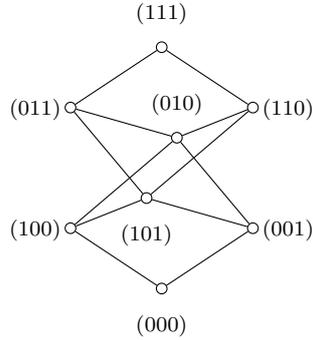
\begin{figure}[hbt]
    \begin{center}
        \begin{tikzpicture}[scale=0.4]
        \tikzstyle{every node}=[draw,circle,fill=white,minimum size=4pt,inner sep=0pt]

        \node (000) at (0,-2) [label=below:\scriptsize(000)] {};
        \node (100) at (-3,0) [label=left:\scriptsize(100)] {};
        \node (001) at (3,0) [label=right:\scriptsize(001)] {};
        \node (101) at (-0.5,1) [label=below:\scriptsize(101)] {};
        \node (010) at (0.5,3) [label=above:\scriptsize(010)] {};
        \node (011) at (-3,4) [label=left:\scriptsize(011)] {};
        \node (110) at (3,4) [label=right:\scriptsize(110)] {};
        \node (111) at (0,6) [label=above:\scriptsize(111)] {};

        \draw (000) -- (100);
        \draw (000) -- (001);
        \draw (100) -- (101);
        \draw (100) -- (010);
        \draw (001) -- (101);
        \draw (001) -- (010);
        \draw (101) -- (110);
        \draw (101) -- (011);
        \draw (010) -- (110);
        \draw (010) -- (011);
        \draw (110) -- (111);
        \draw (011) -- (111);

        \end{tikzpicture}
    \end{center}
    \caption{$\Ynm[1][3]$ (vertices are drawn without the buckets)}
    \label{fig:Y13}
\end{figure}

\begin{observation}\label{obs:Anm_is_aut}
    For $m\leq 1$, $\aut(\Ynm)=\AGnm$.
\end{observation}
\begin{proof}
    The cases $\Ynm[1][0]$, $\Ynm[2][0]$ and $\Ynm[1][1]$ are trivial.
    By Observation \ref{obs:small_m_is_trivial}, if $m=0$ and $n>2$, then $\Ynm$ is the cycle graph on $n$ vertices.
    If $m=1$ and $n>1$, then $\Ynm$ is the cycle graph on $2n$ vertices. 
    Therefore, in both cases, $\aut(\Ynm)$ is the corresponding dihedral group.
    By Theorem \ref{thm:anm_structure}, it follows that $\aut(\Ynm)=\AGnm$ when $m\leq 1$.
\end{proof}

In the rest of this section, unless explicitly stated otherwise, we assume that $m\geq 3$ and $(n,m)\neq (1,3)$ in $\Ynm$.
The purpose of the rest of this section is the proof of Lemma \ref{lem:Anm_is_aut} where we show that indeed $\aut(\Ynm)=\AGnm$.
Specifically, Lemma \ref{lem:N_zero_equal} implies that if there exists $\pi\in\aut(\Ynm)$ such that $\pi\notin\AGnm$, then we can assume that $\pi$ maps between a unique pair of vertices in $\Ynm$:  $(n-1,0,0,1,0,\dots,0)$ and $(n-2,1,1,0,0,\dots,0)$. 
This leads to a contradiction.

\begin{definition}\label{def:aut_notations}Let $x\in\Ynm$ and let $k\in\{0,\dots,n-1\}$. 
    We use the following notations throughout this section.
	\begin{enumerate}
		\item $0_\ell=\overrightarrow{s}_0(0)$ and $0_r=\overleftarrow{s}_m(0)$.
		\item $0_k=\varphi^k(0)=(k,0,\dots,0,n-k)$.
		\item $1_k=\varphi^k(0,1,\dots,1,-m)=(k,1,\dots,1,-(m+k))$.
		\item $N^0(x)=\{y:y\sim x\mbox{ and }d(y,0)=d(x,0)-1\}$.
		\item $n^0(x)=|N^0(x)|$.
	\end{enumerate}
    If $y\in N^0(x)$, we say that $x$ \textbf{covers} $y$ and write $y\lessdot x$.
\end{definition}

We use diagrams to visualize covering relations between vertices.
For example, in the following diagram $w = s_j(v) \lessdot v$ and 
the directed edge $(u, w)$, combined with the label $\overleftarrow{s}_i$, indicates that $w = \overleftarrow{s}_i(u) \lessdot u$.

\begin{center}
    \begin{tikzpicture}
    \matrix (m) [matrix of math nodes,row sep=2em,column sep=2em,minimum width=2em]
    {
        u & \; &  v\\
        \;  & w & \; \\};
    \path[-stealth]
    (m-1-1) edge [->] node [above] {$\overleftarrow{s}_i$} (m-2-2)
    (m-2-2) edge [-] node [above] {$s_j$} (m-1-3);
    \end{tikzpicture}
\end{center}

In Definition \ref{def:uvw_setup}, we setup three vertices in $\Ynm$ in a way that will recur throughout the statements in this section.
The example diagram above visualizes the covering relations in this setup.

\begin{definition}\label{def:uvw_setup}
	Let $u, v\in \Ynm$ be two distinct vertices such that $N^0(u)= N^0(v)$ and $d(u,0)=d(v,0)\geq 2$. 
    Assume that there exists $w\in N^0(u)= N^0(v)$ such that $w=\overleftarrow{s}_i(u)=s_j(v)$ for some $0\leq i,j\leq m$.
    We say that $(u, v, w)$ are in \textbf{$\Vij$-position}.
\end{definition}


\begin{lemma}\label{lem:N_zero_size_far}
    Let $(u, v, w)$ be in $\Vij$-position.
    Then $|i-j|>1$.
\end{lemma}
\begin{proof}
$i\neq j$, since $u\neq v$.
Assume to the contrary that $|i-j|=1$ and assume that $j=i-1$.
By definition, $u, v$ and $w$ are distinct and $s_i(u)$ shifts left, therefore $s_{i-1}(w)$ must shift left as well.
We visualize the scenarios in this proof using diagrams in which the three entries that $s_i$ and $s_j$ change are all non-bucket entries.
Note that the proof also covers the cases in which one of these entries is a bucket (at most one of the entries, since $m\geq 3$).
\begin{center}
    \begin{tikzpicture}
    \matrix (m) [matrix of math nodes,row sep=0.75cm,column sep=1,minimum width=2em]
    {
        (u_{i-1},u_{i},u_{i+1})=(0,0,1)\\
          (w_{i-1},w_{i},w_{i+1})=(0,1,0)  \\
           (v_{i-1},v_{i},v_{i+1})=(1,0,0)\\};
    \path[-stealth]
    (m-1-1) edge [->] node [left] {$\overleftarrow{s}_i$} (m-2-1)
    (m-2-1) edge [->] node [left] {$\overleftarrow{s}_{i-1}$} (m-3-1);
    \end{tikzpicture}
\end{center}
By Lemma \ref{lem:geodesic_is_a_pivot_path_ynm} and Observation \ref{obs:shift_direction_pivot_sides}, $u$ has a geodesic pivot $p$ such that $p\geq i+1$.
$p$ is also a geodesic pivot of $w$; therefore, by Observation \ref{obs:shift_direction_pivot_sides}, $\overleftarrow{s}_{i-1}(w)=v \lessdot w$, in contradiction to $w\in N^0(v)$.

Assume that $j=i+1$.
Similarly to the previous case, $s_{i+1}(w)$ must shift left as well,
$u$ has a geodesic pivot $p$ such that $p\geq i+1$ and
$p$ is also a geodesic pivot of $w$.
In this case, since $\overleftarrow{s}_{i+1}(w) = v \gtrdot w$, $w$ has no geodesic pivot greater than $i+1$.
Therefore, $i+1$ is a geodesic pivot of $w$.
$i$ is also a geodesic pivot of $w$ following similar arguments.
This implies that $i$ and $i+1$ are geodesic pivots of $v$ and $u$ respectively.
In the following diagram, we use the symbol $\|$ to visualize a geodesic pivot.

\begin{center}
    \begin{tikzpicture}
    \matrix (m) [matrix of math nodes,row sep=2em,column sep=0.5em,minimum width=2em]
    {
        (u_{i},u_{i+1},u_{i+2})=(0,1\|1) & \; &  (v_{i},v_{i+1},v_{i+2})=(1\|1,0)\\
        \;  & (w_{i},w_{i+1},w_{i+2})=(1\|0\|1) & \; \\};
    \path[-stealth]
    (m-1-1) edge [->] node [above] {$\overleftarrow{s}_i$} (m-2-2)
    (m-2-2) edge [->] node [above] {$\overleftarrow{s}_{i+1}$} (m-1-3);
    \end{tikzpicture}
\end{center}

We split the rest of the proof of this case ($j=i+1$) into the following three sub-cases.
Assume that $n=1$ and that $m$ is even ($3<m$).
Then, by Observation \ref{obs:closer_pivot_is_better}, $\frac{m}{2}$ is a geodesic pivot of every vertex in $\Ynm[1][m]$.
Note that a unit shift between $\frac{m}{2}$ and $\frac{m}{2}+1$ does not change the distance of any vertex from $0$. 
This implies that there does not exist $w\in \Ynm[1][m]$ and $i\in\{0,\dots m\}$ such that both $w\lessdot \overrightarrow{s}_{i}(w)$ and $w\lessdot \overleftarrow{s}_{i+1}(w)$.

In the last two cases, we show that $N^0(u)\neq N^0(v)$.
Assume that $n=1$ and that $m$ is odd ($3<m$).
Then, by Observation \ref{obs:closer_pivot_is_better}, $\lfloor\frac{m}{2}\rfloor$ and $\lceil\frac{m}{2}\rceil$ are geodesic pivots of every vertex in $\Ynm[1][m]$.
This implies that $i=\lfloor\frac{m}{2}\rfloor$.
For example, $u=(0,0,0,1,1,0,0)$, $w=(0,0,1,0,1,0,0)$ and $v=(0,0,1,1,0,0,0)$ in $\Ynm[1][5]$.
Therefore, there exists some $t<i$ such that $\overleftarrow{s}_{t}(v)=v'\lessdot v$ and clearly $v'\notin N^0(u)$.

Assume that $1<n$ (and $3\leq m$).
Then at least one of the following is true: either $0<i$ or $i+2<m+1$.
If $0<i$, then there exists some $t<i$ such that $\overleftarrow{s}_{t}(v)=v'\lessdot v$ and $v'\notin N^0(u)$.
If $i+2<m+1$, then there exists some $t\geq i+2$ such that $\overrightarrow{s}_{t}(u)=u'\lessdot u$ and $u'\notin N^0(v)$. 
In both cases $N^0(u)\neq N^0(v)$.
\end{proof}

\begin{lemma}\label{lem:bowtie_in_Ic}
    Let $u\in\Ynm$ and let $s_t$ and $s_k$ be two distinct elements in $\Snm$ such that $1<|t-k|$, $s_t(u)\lessdot u$ and $s_k(u)\lessdot s_ks_t(u)$.
    Then
    \begin{enumerate}
    \item both $s_t(u)$ and $s_k(u)$ shift units from entries in $I_c(u)$;
    \item $s_t(u)$ and $s_k(u)$ shift in opposite directions;
    \item $p_l(u)$ and $p_r(u)$ are geodesic pivots.
    \end{enumerate}
\end{lemma}
\pagebreak
\begin{proof}
    $s_t$ and $s_k$ commute, since $1<|t-k|$.
    Denote $v=s_ts_k(u)=s_ks_t(u)$, $w=s_t(u)$ and $w'=s_k(u)$.
    The assumptions imply that $u$, $v$, $w$ and $w'$ are distinct and the covering relations between them can be presented as follows:
\begin{center}
\begin{tikzpicture}
\matrix (m) [matrix of math nodes,row sep=3em,column sep=4em,minimum width=2em]
{
    u &  v\\
    w & w' \\};
\path[-stealth]
(m-1-1) edge [-] node [left] {$s_t$} (m-2-1) edge [-] node [above] {$s_k$} (m-2-2)
(m-1-2) edge [-] node [right] {$s_t$} (m-2-2) edge [-] (m-2-1);
\end{tikzpicture}
\end{center}
    

    Assume that $s_t(u)$ shifts left (the proof is similar if $s_t(u)$ shifts right).
    $t+1\leq p_r(u)$, since $\overleftarrow{s}_t(u)\lessdot u$.
    If $t+1\leq p_l(u)$, then $s_t$ shifts left in every geodesic pivot path of $u$.
    This implies that $\overleftarrow{s}_ts_k(u)\lessdot s_k(u)$.
    A contradiction.
    Therefore $s_t$ shifts a unit from an entry in $I_c(u)$.
    $s_k(u)$ also shifts a unit from an entry in $I_c(u)$, following similar arguments.
    Note that $s_k(u)$ and $s_k(w)$ shift in the same direction, since $1<|t-k|$.
    This implies that $s_t(u)$ and $s_k(u)$ shift in opposite directions, otherwise $s_k(w)\lessdot w$.
    Case 3 is implied by the first two, by Observation \ref{obs:closer_pivot_is_better}.
\end{proof}

\begin{corollary}\label{cor:uvw_setup_in_Ic}
    Let $(u, v, w)$ be in $\Vij$-position.
    If $s_j(u)\lessdot u$, then 
    \begin{enumerate}
    \item $\{i+1, j\}\subseteq I_c(u)$, $\{i, j+1\}\subseteq I_c(v)$;
    \item $s_j(u)$ shifts right;
    \item $p_l(u), p_r(u), p_l(v)$ and $p_r(v)$ are geodesic pivots.
    \end{enumerate}
\end{corollary}

\begin{lemma}\label{lem:possible_uv}
    Let $x\in\Ynm$ such that
    \begin{enumerate}
        \item $n^0(x)=2$;
        \item $I_c(x)$ is not trivial (i.e. $x_{p_l+1}\neq 0$, $x_{p_r}\neq 0$ and $\sum_{i\in I_c}x_i=n$);
        \item both $p_l(x)$ and $p_r(x)$ are geodesic pivots.
    \end{enumerate}
\pagebreak
    Then $x$ is in one of the following three forms:\\
    Either $p_l=-1$ and $p_r=m+1$: 
       \begin{enumerate}
        \item[(a)] $x=(\frac{n}{2},0,\dots,0,\frac{n}{2})$ (in this case $n$ is even);
        \item[(b)] $x=(\frac{n-m}{2},1,\dots,1,\frac{n-m}{2})$ (in this case $n-m\geq 2$ and $2\divides n-m$).
       \end{enumerate}
   	Or $p_l=\frac{m-n}{2}$, $p_r=\frac{m+n}{2}$, and for some $0\leq t_1\leq\frac{m-n}{2}+1$ and $\frac{m+n}{2}\leq t_2\leq m$:
       \begin{enumerate}
        \item[(c)] $x=(-(\frac{n-m}{2}+1-t_1),0,\dots, 0, \overbrace{1,\dots,1}^{x_{t_1}, \dots, x_{t_2}},0,\dots,0,-(\frac{n+m}{2}-t_2))$ (in this case $m\geq n$ and $2\divides m-n$).
        \end{enumerate}
\end{lemma}
\begin{proof}
We split the proof into the following three cases:

Assume that $p_l=-1$ and $p_r=m+1$.
Then $x_0\neq 0$, $x_{m+1}\neq 0$ and $\sum_{i=0}^{m+1}x_i=n$.
Note that all of the non-bucket entries of $x$ must be equal to each other, otherwise $n^0(x)>2$.
If $(x_1,\dots,x_m)=(0,\dots, 0)$, then $x=(\frac{n}{2},0,\dots,0,\frac{n}{2})$, the only vertex in which $(x_1,\dots,x_m)=(0,\dots, 0)$ and $\ps_{-1}(x)=\ps_{m+1}(x)$.
Similarly, if $(x_1,\dots,x_m)=(1,\dots, 1)$, then $x=(\frac{n-m}{2},1,\dots,1,\frac{n-m}{2})$ ($m<n$, $2\divides m-n$ and $n-m\geq 2$).

Assume that both $p_l$ and $p_r$ are inner pivots.
Then $x_{p_l+1}=x_{p_r}=1$ and $\sum_{i\in I_c}x_i=n$.
Note that all of the non-bucket entries of $x$ that are equal to $1$ are contiguous, otherwise $n^0(x)>2$.
Therefore, $p_r-p_l=n$, $p_l=\frac{m-n}{2}$ and $p_r=\frac{m+n}{2}$ (and $m\geq n$ and $2\divides m-n$).
This implies that $x$ is in the form of Case (c) of this lemma, since $\ps_{p_l}(x)=\ps_{p_r}(x)$.

Finally, we show that if one of $p_l$ and $p_r$ is an inner pivot, the other an outer pivot and $n^0(x)=2$, then $p_l$ and $p_r$ cannot both be geodesic pivots.
Assume, without loss of generality, that $p_l=-1$ and $\frac{m}{2}<p_r\leq m$.
Then $x_0\neq 0$, $x_{p_r}=1$ and $\sum_{i=0}^{p_r}x_i=n$.
Note, again, that all of the non-bucket entries of $x$ that are equal to $1$ are contiguous, otherwise $n^0(x)>2$.
This implies that $\ps_{-1}(x)=\ps_{p_l}(x)>\ps_{p_r}(x)$.
\end{proof}

\begin{observation}\label{obs:no_such_pair}
    Let $u,v,w\in \Ynm$, such that $u$ and $v$ are each with one of the forms in Lemma \ref{lem:possible_uv}. 
    Then $(u,v,w)$ is not in $V_{i,j}$-position.
\end{observation}

\begin{lemma}\label{lem:N_zero_size_only_one}
	Let $(u, v, w)$ be in $\Vij$-position.
    Then $n^0(u)=n^0(v)=1$.
\end{lemma}
\begin{proof}
	By Lemma \ref{lem:N_zero_size_far}, necessarily $|i-j|>1$.
	Therefore, $s_i$ and $s_j$ commute and both $v=s_{j}s_{i}(u)$ and $v=s_{i}s_{j}(u)$, are geodesics between $u$ and $v$.
	We first show that these two geodesics are, in fact, the only two geodesics between $u$ and $v$ and therefore $n^0(u)=n^0(v)\leq 2$.
	
	If $1<n$ or $1\leq i,j\leq m-1$, then $u$ and $v$ have four distinct entries (indexed by $i,i+1,j$ and $j+1$) in which they are not equal.
	This implies that $v=s_{j}s_{i}(u)$ and $v=s_{i}s_{j}(u)$ are the only two geodesics between $u$ and \nolinebreak$v$, as required.
	
	Assume that $n=1$ (and $3<m$) and that at least one of $i$ and $j$ is in $\{0,m\}$.
    Note that $s_i$ and $s_j$ shift in $s_js_i(v)$ and $s_is_j(v)$ in opposite directions than they do in $v=s_js_i(u)$ and $v=s_is_j(u)$.
    Therefore, we can assume without loss of generality, that $i\in\{0,m\}$ (by interchanging the roles of $s_i$ and $s_j$ if $i\notin\{0,m\}$; and if $s_j(u)$ shifts right, then also by interchanging the roles of $u$ and $v$).
	Assume also that $i=0$ (the proof for the case $i=m$ follows by similar arguments).

	If $j=m$, then $v=s_{j}s_{i}(u)$ and $v=s_{i}s_{j}(u)$ are the only two geodesics from $u$ to $v$, since $3<m$.
	Therefore $2\leq j\leq m-1$.
    In this case, $\sum_{i=1}^{m}u_i - v_i = 1$; therefore either $\overleftarrow{s}_0$ or $\overrightarrow{s}_m$ is in every word corresponding to a geodesic from $u$ to $v$. 
    It is sufficient to show that $\overleftarrow{s}_0$ is in every word corresponding to a geodesic from $u$ to $v$.
    If $j< m-1$, then clearly, $\overleftarrow{s}_0$ is in every word corresponding to a geodesic from $u$ to $v$.
    The same result follows if $j= m-1$, again, since $m>3$.
    
	It remains to show that $n^0(u) = n^0(v) \neq 2$.
	Assume to the contrary that $n^0(u) = n^0(u) = 2$.
	Since $v=s_js_i(u)$ and $v=s_is_j(u)$ are the only two geodesics from $u$ to $v$, $N^0(u)=N^0(v)=\{s_i(u), s_j(u)\}=\{s_j(v), s_i(v)\}$.
    Therefore $s_j(u)\lessdot u$ and the vertices $u$ and $v$ are as in Corollary \ref{cor:uvw_setup_in_Ic}.
    By the second property of Corollary \ref{cor:uvw_setup_in_Ic}, $I_c(u)$ and $I_c(v)$ are not trivial (and $n>1$).
    Therefore, $u$ and $v$ are two distinct vertices, each with one of the forms as in Lemma \ref{lem:possible_uv}.
    This is a contradiction by Observation \ref{obs:no_such_pair}.
\end{proof}

\begin{lemma}\label{lem:N_zero_equal}
	Let $(u, v, w)$ be in $V_{i,j}$-position.
    Then
    \begin{enumerate}
    \item $u$ and $v$ are the unique pair of vertices $(n-1,0,0,1,0,\dots,0)$ and \linebreak $(n-2,1,1,0,0,\dots,0)$, and $w=(n-1,0,1,0,0,\dots,0)$;
    \item if $n=1$, then $6\leq m$.
    \end{enumerate}
\end{lemma}
\begin{proof}
	By Lemma \ref{lem:N_zero_size_only_one}, $n^0(u)=n^0(v)=1$.
    We first prove that if $n=1$, then $6\leq m$.
	Assume to the contrary that $n=1$ and $m\in\{4,5\}$.
	In this case, there are precisely three vertices in $\Ynm$ which cover only one vertex and that the covered vertex is obtained via a left unit shift:
    $0_\ell$, $(0,1,1,0,\dots,0)$ and $(0,0,1,0,\dots,0)$.
	Therefore either $u=(0,1,1,0,\dots,0)$ and $s_i=\overleftarrow{s}_0$ or $u=(0,0,1,0,\dots,0)$ and $s_i=\overleftarrow{s}_1$, since $w\neq 0$.
	In both cases, $n^0(v)=2$ for every possible $s_j$.
    A contradiction.
    
	It is simple to verify that the pair of vertices given in this lemma indeed satisfies the conditions of the lemma.
	Let $(u,v,w)$ be in $V_{i,j}$-position, where $u$ and $v$ are any other pair of vertices.
	We show that $n^0(v)>1$.
	
	Since $n^0(u)=1$ and $s_i(u)$ shifts left, every geodesic between $u$ and $0$ contains only left unit shifts and $u$ consists of a sequence of $k\geq 1$ consecutive non-zero entries which are to be shifted into the left bucket in every geodesic from $u$ to $0$ (including the case for $k=1$ where $u=(u_0,0,\dots,0,u_{m+1})$ with $0<u_{m+1}<\frac{n}{2}$).
	
	If $k=1$, then $i\notin\{0,2\}$, since $u\notin\{0_\ell, (n-1,0,0,1,0,\dots,0)\}$.
	Clearly, $s_j\in\{s_0,s_m\}$ and in both cases, $n^0(v)=2$.
	Similarly, if $k>1$, then it is simple to verify that $u=(n-2,1,1,0,0,\dots,0)$ is the only case in which there exists $s_j$ such that $n^0(v)=1$.
\end{proof}
\pagebreak
\begin{observation}
	By symmetric arguments to the ones in the proof of Lemma \ref{lem:N_zero_equal}, if $s_i$ were to shift right in Definition \ref{def:uvw_setup}, then the only pair of vertices for which $n^0(u) = n^0(v)=1$ was $(0,\dots,0,1,0,0,n-1)$ and $(0,\dots,0,0,1,1,n-2)$. 
\end{observation}

\begin{observation}\label{obs:assume_fixing_zero}
    If $f\in\aut(\Ynm)$, then there exist $\alpha\in\{0,\dots,n-1\}$ and $\beta, \gamma\in\{0,1\}$ such that $0$, $0_\ell$ and $0_r$ are fixed points of $\pi=\psi^\gamma\tau^\beta\varphi^\alpha f$.
\end{observation}
\begin{proof}
    A vertex in $\Ynm$ has valency $2$ if and only if all of its non-bucket entries are equal to each other.
    Therefore, there exist $\alpha\in\{0,\dots,n-1\}$ and $\beta\in\{0,1\}$ such that $0$ is a fixed point of $\pi=\tau^\beta\varphi^\alpha f$.
    $0$ is also a fixed point of $\psi$, and $\psi$ interchanges the only two neighbors of $0$; $0_\ell$ and $0_r$.
    Therefore, there exists $\gamma\in\{0,1\}$ such that $0$, $0_\ell$ and $0_r$ are fixed points of $\pi=\psi^\gamma\tau^\beta\varphi^\alpha f$.
\end{proof}    

%
\begin{lemma}\label{lem:assume_fixing_zero_one}
	Let $\pi\in\aut(\Ynm)$ such that $0$, $0_\ell$ and $0_r$ are fixed points of $\pi$.
	Then $0_1$ and $0_{n-1}$ are also fixed points of $\pi$.
\end{lemma}
\begin{proof}
	If $n=1$, then $0_1=0_{n-1}=0$, which is a fixed point of $\pi$.
	If $n>2$, then $0_1$ and $0_{n-1}$ are the only (non zero) vertices of valency $2$ with a single geodesic between them and $0$.
	Therefore, $0_1$ and $0_{n-1}$ are either fixed or interchanged by $\pi$.
	Note that the geodesic between $0_1$ and $0$ contains $0_r$ and doesn't contain $0_\ell$ whereas the geodesic between $0_{n-1}$ and $0$ contains $0_\ell$ and doesn't contain $0_r$.
	Therefore, $0_1$ and $0_{n-1}$ are fixed by $\pi$, since $0_r$ and $0_\ell$ are fixed by $\pi$.
	
	Assume that $n=2$ and $m>3$. In this case, $0_1=0_{n-1}$ and there are four vertices with valency $2$: $0, 0_1, 1_0$ and $1_1$.
	Note that $d(0_1,0)=m+1$ and, by Observation \ref{obs:closer_pivot_is_better} and Fact \ref{fac:split_sum_center},         
        $d(1_k,0)\geq \binom{\lfloor\frac{m}{2}\rfloor+1}{2}+\binom{\lceil\frac{m}{2}\rceil+1}{2}$ (for $k\in\{0,1\}$).
	Therefore, since $m>3$, $d(0_1,0)<d(1_k,0)$ for both $k=0$ and $k=1$.
	This proves that $0_1$ is fixed by $\pi$.
	
	The only case left is $n=2$ and $m=3$.
	Note that $x=(1,1,0,1,1)\in\Ynm[2][3]$ is fixed by $\pi$, since it is the only vertex at distance $2$ from $0$ that is a neighbor of both $0_\ell$ and $0_r$.
	Both vertices $1_0$ and $1_1$ are at distance $2$ from $x$ whereas the vertex $0_1$ is at distance $4$ from $x$.
	Therefore $0_1$ is fixed by $\pi$.
\end{proof}

\begin{lemma}\label{lem:Anm_is_aut}
	Let $n\geq 1$ and $m\geq 3$ be two integers such that $(n,m)\neq (1,3)$.
	Then $\aut(\Ynm)=\AGnm$.
\end{lemma}
\begin{proof}
	Let $f\in\aut(\Ynm)$.
	By Observation \ref{obs:assume_fixing_zero} and Lemma \ref{lem:assume_fixing_zero_one}, there exist integers $\alpha, \beta, \gamma$ such that $0$, $0_\ell$, $0_r$, $0_1$ and $0_{n-1}$ are fixed points of $\pi = \psi^\gamma \tau^\beta \varphi^\alpha f$.
	We show that $\pi$ is, in fact, the trivial automorphism.	
	Assume to the contrary that $\pi\neq Id$ and let $u\in\Ynm$ be a vertex at minimal distance $d$ from $0$ such that $u$ is not fixed by $\pi$.
	Denote $v=\pi(u)$.
	Note that $N^0(u)=N^0(v)$, by the minimality of $d$, and let $w\in N^0(u)= N^0(v)$.
    Note that $d\geq 2$ and therefore $w\neq 0$.
	
	Let $0\leq i,j\leq m$ such that $s_i(u)=s_j(v)=w$, so that $v=s_j s_i(u)$. 
	Assume that $s_i$ shifts left (the proof of the other case follows by symmetric arguments).
	Therefore, $(u, v, w)$ are in $V_{i,j}$-position.
	By Lemma \nolinebreak \ref{lem:N_zero_equal}, $u$ and $v$ are the unique pair of vertices 
		$(n-1,0,0,1,0,\dots,0)$ and \linebreak
		$(n-2,1,1,0,0,\dots,0)$;
	and if $n=1$, then $6\leq m$.

    If $m>3$, then $u$ and $v$ have different valencies, in contradiction to $\pi(u)=v$.
    Finally, assume $m=3$ (and $n>1$).
    Then $u$ and $v$ are the pair of vertices $u_1=(n-1,0,0,1,0)$ and $u_2=(n-2,1,1,0,0)$.
    
    Since $0_{n-1}$ is a fixed point of $\pi$,  $d(u_1,0_{n-1})=d(u_2,0_{n-1})$.
    Therefore, 
    $$d(\varphi(u_1), 0)=d(\varphi(u_1),\varphi(0_{n-1})) = d(\varphi(u_2),\varphi(0_{n-1}))=d(\varphi(u_2),0).$$
    But $d(\varphi(u_1),0)=1$ and $d(\varphi(u_2),0)=3$.
    A contradiction.
\end{proof}

\cleardoublepage
\phantomsection

\end{document}